\documentclass[11pt,reqno]{amsproc} \title[The inviscid limit for the Navier-Stokes equations]{The inviscid limit for the Navier-Stokes equations with data\\ analytic only near the boundary} \author[I.~Kukavica]{Igor Kukavica} \address{Department of Mathematics, University of Southern California, Los Angeles, CA 90089} \email{kukavica@usc.edu} \author[V.~Vicol]{Vlad Vicol} \address{Courant Institute of Mathematical Sciences, New York University, New York, NY 10012} \email{vicol@cims.nyu.edu} \author[F.~Wang]{Fei Wang} \address{Department of Mathematics, University of Maryland, College Park, MD 20740} \email{fwang256@umd.edu} \usepackage{fancyhdr} \usepackage{comment} \usepackage[margin=1in]{geometry} \usepackage{amsmath, amsthm, amssymb} \usepackage{times} \usepackage{graphicx} \usepackage[usenames,dvipsnames,svgnames,table]{xcolor} \usepackage[colorlinks=true, pdfstartview=FitV, linkcolor=blue, citecolor=blue, urlcolor=blue]{hyperref} \begin{document} \def\XX{X} \def\YY{Y} \def\ZZZ{Z} \def\intint{\int\!\!\!\!\int} \def\OO{\mathcal O} \def\SS{\mathbb S} \def\CC{\mathbb C} \def\RR{\mathbb R} \def\TT{\mathbb T} \def\ZZ{\mathbb Z} \def\HH{\mathbb H} \def\RSZ{\mathcal R} \def\LL{\mathcal L} \def\SL{\LL^1} \def\ZL{\LL^\infty} \def\GG{\mathcal G} \def\eps{\varepsilon} \def\tt{\langle t\rangle} \def\erf{\mathrm{Erf}} \def\red#1{\textcolor{red}{#1}} \def\blue#1{\textcolor{blue}{#1}} \def\mgt#1{\textcolor{magenta}{#1}} \def\ff{\rho} \def\gg{G} \def\tilde{\widetilde} \def\sqrtnu{\sqrt{\nu}} \def\ww{w} \def\ft#1{#1_\xi} \def\les{\lesssim} \renewcommand*{\Re}{\ensuremath{\mathrm{{\mathbb R}e\,}}} \renewcommand*{\Im}{\ensuremath{\mathrm{{\mathbb I}m\,}}} \def\llabel#1{\notag} \newcommand{\norm}[1]{\left\|#1\right\|} \newcommand{\nnorm}[1]{\lVert #1\rVert} \newcommand{\abs}[1]{\left|#1\right|} \newcommand{\NORM}[1]{|\!|\!| #1|\!|\!|} \newtheorem{theorem}{Theorem}[section] \newtheorem{corollary}[theorem]{Corollary} \newtheorem{proposition}[theorem]{Proposition} \newtheorem{lemma}[theorem]{Lemma} \theoremstyle{definition} \newtheorem{definition}{Definition}[section] \newtheorem{remark}[theorem]{Remark} \def\theequation{\thesection.\arabic{equation}} \numberwithin{equation}{section} \def\ll{{\color{red}\ell}} \def\ee{\epsilon_0} \def\startnewsection#1#2{\section{#1}\label{#2}\setcounter{equation}{0}} \def\nnewpage{\newpage} \def\sgn{\mathop{\rm sgn\,}\nolimits} \def\Tr{\mathop{\rm Tr}\nolimits} \def\div{\mathop{\rm div}\nolimits} \def\supp{\mathop{\rm supp}\nolimits} \def\indeq{\quad{}} \def\period{.} \def\semicolon{\,;} \def\nts#1{{\cor #1\cob}} \def\colr{\color{red}} \def\colb{\color{black}} \definecolor{colorgggg}{rgb}{0.1,0.5,0.3} \definecolor{colorllll}{rgb}{0.0,0.7,0.0} \definecolor{colorhhhh}{rgb}{0.3,0.75,0.4} \definecolor{colorpppp}{rgb}{0.7,0.0,0.2} \definecolor{coloroooo}{rgb}{0.9,0.4,0} \definecolor{colorqqqq}{rgb}{0.1,0.7,0} \def\colg{\color{colorgggg}} \def\collg{\color{colorllll}} \def\cole{\color{black}} \def\colu{\color{blue}} \def\colc{\color{colorhhhh}} \def\colW{\colb} \def\comma{ {\rm ,\qquad{}} } \def\commaone{ {\rm ,\quad{}} } \def\les{\lesssim} \def\nts#1{{\color{red}\hbox{\bf ~#1~}}} \def\blackdot{{\color{red}{\hskip-.0truecm\rule[-1mm]{4mm}{4mm}\hskip.2truecm}}\hskip-.3truecm} \def\bluedot{{\color{blue}{\hskip-.0truecm\rule[-1mm]{4mm}{4mm}\hskip.2truecm}}\hskip-.3truecm} \def\purpledot{{\color{colorpppp}{\hskip-.0truecm\rule[-1mm]{4mm}{4mm}\hskip.2truecm}}\hskip-.3truecm} \def\greendot{{\color{colorgggg}{\hskip-.0truecm\rule[-1mm]{4mm}{4mm}\hskip.2truecm}}\hskip-.3truecm} \def\cyandot{{\color{cyan}{\hskip-.0truecm\rule[-1mm]{4mm}{4mm}\hskip.2truecm}}\hskip-.3truecm} \def\reddot{{\color{red}{\hskip-.0truecm\rule[-1mm]{4mm}{4mm}\hskip.2truecm}}\hskip-.3truecm} \def\gdot{\greendot} \def\bdot{\bluedot} \def\ydot{\cyandot} \def\rdot{\cyandot} \def\fractext#1#2{{#1}/{#2}} \def\ii{\hat\imath} \def\fei#1{\textcolor{blue}{#1}} \def\vlad#1{\textcolor{cyan}{#1}} \def\igor#1{\textcolor{colorqqqq}{#1}} \begin{abstract} We address the inviscid limit for the Navier-Stokes equations in a half space, with initial datum that is analytic only close to the boundary of the domain, and has finite Sobolev regularity in the complement. We prove that for such data the solution of the Navier-Stokes equations converges in the vanishing viscosity limit to the solution of the Euler equation, on a constant time interval. \hfill \today \end{abstract} \maketitle \startnewsection{Introduction}{sec01} We consider the Cauchy problem for the 2D incompressible Navier-Stokes equations \begin{align} &\partial_t u - \nu\Delta u + u\cdot\nabla u + \nabla p = 0 \label{EQthisijhhssdinthegdsfg:01}\\ &\div u = 0 \label{EQthisijhhssdinthegdsfg:02}\\ &u|_{t=0} = u_0 \label{EQthisijhhssdinthegdsfg:03} \end{align} on the half-space domain $\HH =\TT \times \RR_+ = \{ (x,y) \in \TT \times \RR \colon y \geq 0\}$, with $\TT = [-\pi,\pi]$-periodic boundary conditions in $x$, and the {\em no-slip boundary condition} \begin{align} &u|_{y=0} = 0 \label{EQthisijhhssdinthegdsfg:03b} \end{align} on $\partial \HH = \TT \times \{y=0\}$. Here $\nu>0$ is the kinematic viscosity. Formally setting $\nu = 0$ in \eqref{EQthisijhhssdinthegdsfg:01}--\eqref{EQthisijhhssdinthegdsfg:03} we arrive at the 2D incompressible Euler equations, which are supplemented with the {\em slip boundary condition} given by $u_2|_{y=0}=0$. \par A fundamental problem in mathematical fluid dynamics is to determine whether in the {\em inviscid limit} $\nu \to 0$ the solutions of the Navier-Stokes equations converge to those of the Euler equations, in the {\em energy norm} $L^\infty(0,T;L^2(\HH))$, on an $\OO(1)$ (with respect to $\nu$) time interval. A classical result of Kato~\cite{Kato84b} relates this problem to the {anomalous dissipation of kinetic energy}: A necessary and sufficient condition for the inviscid limit to hold in the energy norm is that the total dissipation of the energy in a boundary layer of width $\OO(\nu)$, vanishes as $\nu\to0$. To date it remains an open problem whether this condition holds for any smooth initial datum. \par In this paper we prove that the inviscid limit holds for initial datum $u_0$ for which the associated vorticity $\omega_0 = \nabla^\perp \cdot u_0$ is analytic in an $\OO(1)$ strip next to the boundary, and is only Sobolev smooth on the complement of this strip. In particular, our main theorem (cf.~Theorem~\ref{T02} below) both implies the seminal result of Sammartino-Caflisch~\cite{SammartinoCaflisch98b}, which assumes analyticity on the entire half-plane, and also the more recent remarkable result of Maekawa~\cite{Maekawa14}, which assumes that the initial vorticity vanishes identically in an $\OO(1)$ strip next to the boundary. \par The fundamental source of difficulties in studying the inviscid limit is the mismatch in boundary conditions between the viscous Navier-Stokes flow (no-slip, $u_1|_{y=0} = u_2|_{y=0} = 0$) and the inviscid Euler flow (slip, $u_2|_{y=0} = 0$). Mathematically, this prohibits us from obtaining $\nu$-independent a~priori estimates for solutions of \eqref{EQthisijhhssdinthegdsfg:01}--\eqref{EQthisijhhssdinthegdsfg:03b} in the uniform norm. The main obstacle is to quantify the creation of vorticity at $\partial \HH$, which is expected to become unbounded as $\nu \to 0$, at least very close to the boundary. \par Concerning the validity of the inviscid limit in the energy norm, in the presence of solid boundaries, for smooth initial datum, two types of results are known. First, we have results which make $\nu$-dependent {\em assumptions on the family of solutions} $u$ of \eqref{EQthisijhhssdinthegdsfg:01}--\eqref{EQthisijhhssdinthegdsfg:03b}, and prove that these assumptions imply (a-posteriori, they are equivalent to) the $L^\infty_t L^2_x$ inviscid limit. This program was initiated in the influential paper of Kato~\cite{Kato84b}, who showed that the condition \begin{align} \lim_{\nu \to 0} \int_0^T \!\!\!\! \int_{\{ y\les \nu \}} \nu |\nabla u|^2 dx dy dt \to 0 \label{EQthisijhhssdinthegdsfg:Kato} \end{align} is equivalent to the validity of the strong inviscid limit in the energy norm. Refinements and extensions based on Kato's original argument of introducing a boundary layer corrector were obtained for instance in~\cite{BardosTiti13, ConstantinElgindiIgnatovaVicol17,ConstantinKukavicaVicol15, Kelliher08,Kelliher17,Masmoudi98,TemamWang97b,Wang01}; see also the recent review~\cite{MaekawaMazzucato16} and references therein. These results are important because they yield explicit properties that the sequence of Navier-Stokes solutions must obey as $\nu\to 0$ in order for them to have a strong $L^\infty_t L^2_x$ Euler limit. On the other hand, verifying these conditions based on the knowledge of the initial datum only, is in general an outstanding open problem. We emphasize that to date, even the question of whether the weak $L^2_t L^2_x$ inviscid limit holds (against test functions compactly supported in the interior of the domain), remains open. Conditional results have been established recently in terms of interior structure functions~\cite{ConstantinVicol18,DrivasNguyen18}, or in terms of interior vorticity concentration measures~\cite{ConstantinBrazil18}. \par The second class of results are those which only make {\em assumptions on the initial data} $u_0$, as $\nu \to 0$. In the seminal works~\cite{SammartinoCaflisch98a,SammartinoCaflisch98b}, Sammartino-Caflisch consider initial data $u_0$ which are analytic in both the $x$ and $y$ variables on the entire half space, and are well-prepared, in the sense that $u_0$ satisfies the Prandtl ansatz \eqref{EQthisijhhssdinthegdsfg:Prandtl} below, at time $t=0$. Sammartino-Caflisch do not just prove the strong inviscid limit in the energy norm, but they in fact establish the validity of the Prandtl expansion \begin{align} u(x,y,t) = \bar u(x,y,t) + u^{P}\left(x,\frac{y}{\sqrt{\nu}},t\right) + \OO(\sqrt{\nu}) \label{EQthisijhhssdinthegdsfg:Prandtl} \end{align} for the solution $u$ of \eqref{EQthisijhhssdinthegdsfg:01}--\eqref{EQthisijhhssdinthegdsfg:03b}. Here $\bar u$ denotes the real-analytic solution of the Euler equations, and $u^P$ is the real-analytic solution of the Prandtl boundary layer equations. We refer the reader to~\cite{AlexandreWangXuYang14,DietertGerardVaret19,GerardVaretMasmoudi13,IgnatovaVicol16,KukavicaMasmoudiVicolWong14,KukavicaVicol13,LiuYang16,LombardoCannoneSammartino03,MasmoudiWong15,Oleinik66,SammartinoCaflisch98a} for the well-posedness theory for the Prandtl equations, to~\cite{GerardVaretDormy10,GuoNguyen11,GerardVaretNguyen12,LiuYang17} for the identification of ill-posed regimes, and to~\cite{Grenier00,GrenierGuoNguyen15,GrenierGuoNguyen16,GrenierNguyen17,GrenierNguyen18a} for recent works which show the invalidity of the Prandtl expansion at the level of Sobolev regularity. In~\cite{SammartinoCaflisch98a,SammartinoCaflisch98b} Sammartino-Caflisch carefully analyze the error terms in the expansion \eqref{EQthisijhhssdinthegdsfg:Prandtl}, and show that they remain $\OO(\sqrt{\nu})$ for an $\OO(1)$ time interval, by appealing to real-analyticity and an abstract Cauchy-Kowalevski theorem. This strategy has been proven successful for treating the case of a channel~\cite{LombardoSammartino01,KukavicaLombardoSammartino16} and the exterior of a disk~\cite{CaflischSammartino97}. Subsequently, in a remarkable work~\cite{Maekawa14}, Maekawa proved that the inviscid limit also holds for initial datum whose associated vorticity is Sobolev smooth and is supported at an $\OO(1)$ distance away from the boundary of the domain. The main new device in~\cite{Maekawa14} is the use of the vorticity boundary condition in the case of the half space~\cite{Anderson89,Maekawa13}, using which one may actually establish the validity of the expansion~\eqref{EQthisijhhssdinthegdsfg:Prandtl}. Using conormal Sobolev spaces, the authors of \cite{WangWangZhang17} have obtained an energy based proof for the Caflisch-Sammartino result, while in~\cite{FeiTaoZhang16,FeiTaoZhang18} it is shown that Maekawa's result can also be proven solely using energy methods, in 2D and 3D respectively. More recently, Nguyen-Nguyen have found in~\cite{NguyenNguyen18} a very elegant proof of the Sammartino-Caflisch result, which for the first time completely avoids the usage of Prandtl boundary layer correctors. Instead, Nguyen-Nguyen appeal to the boundary vorticity formulation, precise bounds for the associated Green's function, and an analysis in boundary-layer weighted spaces. In this paper we use a number of estimates from~\cite{NguyenNguyen18}, chief among which are the ones for the Green's function for the Stokes system (see Lemma~\ref{Green:fun} below). Lastly, we mention that in a recent remarkable result~\cite{GerardVaretMaekawaMasmoudi16}, Gerard-Varet-Maekawa-Masmoudi establish the stability in a Gevrey topology in $x$ and a Sobolev topology in $y$, of Euler+Prandtl shear flows (cf.~\eqref{EQthisijhhssdinthegdsfg:Prandtl}), when the Prandtl shear flow is both monotonic and concave. It is worth noting that in all the above cases the Prandtl expansion \eqref{EQthisijhhssdinthegdsfg:Prandtl} is valid, and thus the Kato-criterion~\eqref{EQthisijhhssdinthegdsfg:Kato} holds. However, in general there is a large discrepancy between the question of the vanishing viscosity limit in the energy norm, and the problem of the validity of the Prandtl expansion. It is not clear to which degree these two problems are related. \par Finally, we mention that the vanishing viscosity limit is also known to hold in the presence of certain symmetry assumptions on the initial data, which is maintained by the flow; see e.g.~\cite{BonaWu02,HanMazzucatoNiuWang12,Kelliher09,LopesMazzucatoLopes08,LopesMazzucatoLopesTaylor08,MaekawaMazzucato16,Matsui94,MazzucatoTaylor08} and references therein. This symmetry implies that the influence of the Prandtl layer to the bulk flow is weak, and thus in these situations the vanishing viscosity limit may be established by verifying Kato's criterion~\eqref{EQthisijhhssdinthegdsfg:Kato}. Also, the very recent works~\cite{GerardVaretMaekawa18,GuoIyer18,GuoIyer18a,Iyer18} establish the vanishing viscosity limit and the validity of the Prandtl expansion for the stationary Navier-Stokes equation, in certain regimes. \par The main goal of this paper is to bridge the apparent gap between the Sammartino-Caflisch~\cite{SammartinoCaflisch98a,SammartinoCaflisch98b} and the Maekawa~\cite{Maekawa14} results, by proving in Theorem~\ref{T02} that the inviscid limit in the energy norm holds for initial datum $\omega_0$ which is analytic in a strip of $\OO(1)$ width close to the boundary, and is Sobolev smooth on the complement of this strip. Evidently, this type of data includes the one considered in~\cite{Maekawa14,SammartinoCaflisch98a,SammartinoCaflisch98b}. Informally, one expects analyticity to only be required near the boundary in order to control the catastrophic growth of boundary layer instabilities, and we confirm this intuition. To the best of our knowledge, our result establishes the inviscid limit in the energy norm for the largest class of initial data, in the {\em absence of structural or symmetry assumptions}. Theorem~\ref{T02} is a direct consequence of Theorem~\ref{T01}, which establishes uniform in $\nu$ bounds on the vorticity in a mixed analytic-Sobolev norm. In order to prove Theorem~\ref{T01}, we use the mild vorticity formulation approach of Nguyen-Nguyen, which avoids the explicit use of Prandtl correctors, and instead relies on pointwise estimates on the Green's function for the associated Stokes equation~\cite[Proposition 3.3]{NguyenNguyen18}. The main technical difficulty we need to overcome is the treatment of the layer where the analyticity and the Sobolev regions meet. It is known that analytic functions are not localizable, and that the Biot-Savart law is non-local. Thus, one cannot avoid that the analytic and the Sobolev regions communicate. In order to overcome this difficulty we consider an analyticity radius with respect to both the $x$ and $y$ variables, which vanishes in a precisely controlled time-dependent fashion at an $\OO(1)$ distance from the boundary. Moreover, since we cannot afford a derivative loss in the Sobolev region, this estimate is carried over using an energy method, with error terms arising to the spill into the analytic region. Compared to~\cite{NguyenNguyen18}, we employ several simplifications which provide additional information on the solution of the Navier-Stokes equations in the boundary layer. First, we remove the need for the time dependent weight function, thus not allowing time dependent bursts of vorticity in the secondary boundary layer of size $\sqrt{\nu t}$ from \cite{NguyenNguyen18}. Second, since our solutions are no longer analytic away from the constant sized strip, we no longer require them to decay exponentially as $y\to\infty$. Lastly, the approach considered here allows a wider choice of weights functions in the analytic norm (cf.~Remark~\ref{R04} below) which may be used to provide a detailed information about the degeneration of the vorticity as $\nu\to 0$ in a suitably defined boundary layer. \par This paper is organized as follows. In Section~\ref{sec02} we introduce the analytic norms $X$ and $Y$ and the Sobolev norm $Z$ used in this paper. Section~\ref{sec04} contains the statements and the proofs of our main results, Theorems~\ref{T01} and~\ref{T02}. For this purpose, we also recall there the integral representation of the vorticity formulation of the Navier-Stokes equations, and we collect in Lemmas~\ref{lem:main:X},~\ref{lem:main:Y},~\ref{lem:main:N}, and~\ref{L12} the main analytic and Sobolev estimates needed to establish our main results. These lemmas are then proven in Section~\ref{secX} for the $X$-norm, Section~\ref{secL} for the $Y$-norm, Section~\ref{sec-N} for the nonlinear terms, and Section~\ref{sec-sobolev} for the Sobolev norm. \par \startnewsection{Functional setting}{sec02} \subsection{Notation} \begin{itemize} \item We use $f_\xi(y) \in \CC$ to denote the Fourier transform of $f(x,y)$ with respect to the $x$ variable at frequency $\xi \in \ZZ$, i.e. $f(x,y) = \sum_{\xi \in \ZZ} f_\xi(y) e^{i x \xi }$. We also use the notation $u_{i,\xi}(y)$ or $(u_{i})_{\xi}(y)$ for the Fourier transform of $u_i$ in $x$ for $i=1, 2$. \item $\Re z$ and $\Im z$ stand for the real and imaginary parts of the complex number $z$. \par \item For $\mu >0$ we define the complex domain $\Omega_\mu=\{z\in \CC: 0 \leq \Re z \leq 1, |\Im z|\le \mu \Re z\} \cup \{z\in \CC: 1 \leq \Re z \leq 1+\mu, |\Im z|\le 1 + \mu - \Re z \}$, which is represented in Figure~\ref{fig:Omega}. We assume that $\mu<\mu_0$, where $\mu_0\in(0,1/10]$ is a fixed constant. \begin{center} \begin{figure}[h!] \includegraphics[width=.5\textwidth]{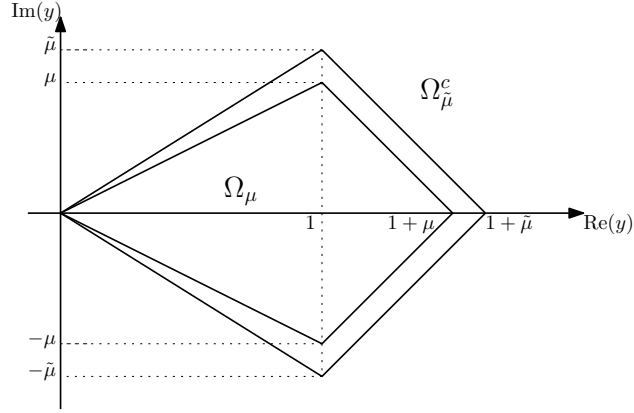} \caption{Representation of the complex domains $\Omega_\mu$ and $\Omega_{\tilde \mu}$ for $0 < \mu < \tilde \mu$.} \label{fig:Omega} \end{figure} \end{center} \item For $y\in \Omega_\mu$ we represent exponential terms of the form $e^{\ee(1+\mu-\Re y)_+|\xi|}$ simply as $e^{\ee(1+\mu-y)_+|\xi|}$. That is, in order to simplify the notation we write $y$ instead of $\Re y$ at the exponential. \item We assume that $\nu\in(0,1]$ and $t \in (0,1]$ throughout. \item The implicit constants in $\les$ depend only on $\mu_0$ and $\theta_0$ (cf.~\eqref{EQthisijhhssdinthegdsfg101}), and are thus universal. \end{itemize} \par \subsection{Norms} In this paper, we use norms which capture the features of a solution that is analytic near the boundary and is $H^4$ smooth at an $\OO(1)$ distance away from it. We include two types of analytic norms: the $L^\infty$ based $X$ norm and the $L^1$ based $\YY$ norm, defined in \eqref{EQthisijhhssdinthegdsfg17} and \eqref{EQthisijhhssdinthegdsfg25} respectively. Before these definitions we introduce some notation. \par For a sufficiently large constant $\gamma>0$ to be determined in the proof, which depends only on $\mu_0$ and the size of the initial datum via the constant $M$ in \eqref{EQthisijhhssdinthegdsfg:omega:0}, throughout the paper we let $t$ obey \begin{align} t \in \left(0, \frac{\mu_0}{2\gamma} \right)\,. \label{EQthisijhhssdinthegdsfg:time:restrict} \end{align} \par In order to define the weighted $L^\infty$ based analytic norm $X$, we introduce the weight function $w \colon [0,1+\mu_0] \to [0,1]$ given by \begin{equation} w(y) = \begin{cases} \sqrtnu &, 0<y\leq \sqrtnu \\ y&, \sqrtnu\leq y\leq 1\\ 1&, 1\leq y\leq 1+\mu_0 \end{cases} \label{EQthisijhhssdinthegdsfg55} \end{equation} and use it to define a weighted analytic norm, with respect to $y$, as \begin{equation} \lVert f\rVert_{\ZL_{\mu,\nu}} = \sup_{y\in\Omega_\mu} \ww(\Re y) |f(y)| \llabel{Ko svojo moč narbolj vihar razklada,
okrog vrat straža na pomoč zavpije,
in vstane šum, de mož za možam pada.
Ko se neurnik o povodnji vlije,
iz hriba strmega v doline plane,
z derečimi valovami ovije,
kar se mu zoper stavi, se ne ugane,
in ne počije préd, de jez omaga;
tak vrže se Valjhun na nekristjane.
Ne jenja préd, dokler ni zadnja sraga
krvi prelita, dokler njih kdo sope,
ki jim bila je vera čez vse draga.
Ko zor zasije na mrličov trope,
ležé, k' ob ajde žetvi, al pšenice
po njivah tam ležé snopovja kope.
Leži kristjanov več od polovice,
med njimi, ki so padli za malike,
Valjhun zastonj tam iše mlado licesldkfjg;ls dslksfgj s;lkdfgj ;sldkfjg ;sldkfgj s;ldfkgj s;ldfkgj gfisasdoifaghskcx,.bvnliahglidhgsd gs sdfgh sdfg sldfg sldfgkj slfgj sl;dfgkj sl;dfgjk sldgfkj jiurwe alskjfa;sd fasdfEQthisijhhssdinthegdsfg15} \end{equation} for a complex function $f$ of the variable $y\in \Omega_\mu$. Throughout the paper, in order to simplify the notation we write $w(y)$ instead of $w(\Re y)$. \par Let $\ee \in (0,1)$ be a sufficiently small constant to be determined below, which only depends on the parameter $\theta_0$ in \eqref{EQthisijhhssdinthegdsfg101}. Moreover, let $ \alpha\in\left( 0, \frac{1}{2} \right) \ $ be a fixed constant. Using the $\ZL_{\mu,\nu}$ norm, we define \begin{equation} \llabel{Mož in oblakov vojsko je obojno
končala temna noč, kar svetla zarja
zlatí z rumenmi žarki glavo trojno
snežnikov kranjskih sivga poglavarja,
Bohinjsko jezero stoji pokojno,
sledu ni več vunanjega viharja;
al somov vojska pod vodó ne mine,
in drugih roparjov v dnu globočine. 
Al jezero, ki na njega pokrajni
stojiš, ni, Črtomír! podoba tvoja? -
To noč je jenjal vojske šum vunajni,
potihnil ti vihar ni v prsih boja;
le hujši se je zbudil črv nekdajni,
ak prav uči me v revah skušnja moja,
bolj grize, bolj po novi krvi vpije,
požrešniši obupa so harpíje.
sldkfjg;ls dslksfgj s;lkdfgj ;sldkfjg ;sldkfgj s;ldfkgj s;ldfkgj gfisasdoifaghskcx,.bvnliahglidhgsd gs sdfgh sdfg sldfg sldfgkj slfgj sl;dfgkj sl;dfgjk sldgfkj jiurwe alskjfa;sd fasdfEQthisijhhssdinthegdsfg18} \lVert f \rVert_{X_\mu} = \sum_{\xi \in \ZZ} \lVert e^{\ee(1+\mu-y)_+|\xi|}f_\xi\rVert_{\ZL_{\mu,\nu}} \, , \end{equation} and then, with $t$ as in \eqref{EQthisijhhssdinthegdsfg:time:restrict}, we define the analytic $X$ norm as \begin{align} \lVert f\rVert_{X(t)} = \sup_{\mu<\mu_0-\gamma t} &\biggl( \sum_{0\le i+j\le1}\lVert \partial_x^i(y\partial_{y})^j f\rVert_{X_\mu} +\sum_{i+j=2} (\mu_0-\mu-\gamma t)^{1/2+\alpha}\lVert \partial_x^i(y\partial_{y})^j f\rVert_{X_\mu} \biggr) \, . \label{EQthisijhhssdinthegdsfg17} \end{align} We state in Lemma~\ref{ana:rec:x} a useful analyticity recovery estimate for the $X_\mu$ norm. \par \begin{remark} \label{R04} Throughout the paper, the following properties of the weight are needed: \begin{itemize} \item[(a)] $w(y) \les w(z) $ for $y\le z$, \item[(b)] $w(y) \les w(z)$ for $0 < y/2 \leq z \leq 1+\mu_0$, \item[(c)] $\sqrt{\nu} \les w(y) \les 1$ for $y \in [0,1+\mu_0]$, \item[(d)] $y \les w(y)$ for $y\in [0,1+\mu_0]$ \item[(e)] $w(y) e^{-\frac{y}{C \sqrt{\nu}}} \les \sqrt{\nu}$ for $y\in [0,1+\mu_0]$ where $C>0$ is sufficiently large constant, depending only on $\theta_0$ in \eqref{EQthisijhhssdinthegdsfg101}. \end{itemize} It is easy to check that the weight $w$ in \eqref{EQthisijhhssdinthegdsfg55} satisfies (a)--(e). We justify (e) by simply distinguishing the cases $y\leq\sqrtnu$ and $y\geq \sqrtnu$ separately. \par Note that, by the above statement, there are other weights for which Theorem~\ref{T01} holds. For instance, we may take \begin{equation} w(y) = \min \{ \sqrtnu e^{\frac{y}{C\sqrtnu}},1 \} \llabel{sldkfjg;ls dslksfgj s;lkdfgj ;sldkfjg ;sldkfgj s;ldfkgj s;ldfkgj gfisasdoifaghskcx,.bvnliahglidhgsd gs sdfgh sdfg sldfg sldfgkj slfgj sl;dfgkj sl;dfgjk sldgfkj jiurwe alskjfa;sd fasdfEQthisijhhssdinthegdsfg57} \end{equation} for a sufficiently large universal constant $C$. Note that this weight is larger than \eqref{EQthisijhhssdinthegdsfg55}, up to a multiplicative constant, and that it satisfies (a)--(e) above. \end{remark} \par Next, we define the analytic $L^1$ based norm. For a complex valued function $f$ defined on $\Omega_\mu$, let \begin{equation} \llabel{sldkfjg;ls dslksfgj s;lkdfgj ;sldkfjg ;sldkfgj s;ldfkgj s;ldfkgj gfisasdoifaghskcx,.bvnliahglidhgsd gs sdfgh sdfg sldfg sldfgkj slfgj sl;dfgkj sl;dfgjk sldgfkj jiurwe alskjfa;sd fasdfEQthisijhhssdinthegdsfg24} \lVert f\rVert_{\SL_\mu} = \sup_{0\le\theta<\mu} \lVert f\rVert_{L^1(\partial\Omega_\theta)} \, . \end{equation} Using $\SL_\mu$ we introduce \begin{equation} \llabel{sldkfjg;ls dslksfgj s;lkdfgj ;sldkfjg ;sldkfgj s;ldfkgj s;ldfkgj gfisasdoifaghskcx,.bvnliahglidhgsd gs sdfgh sdfg sldfg sldfgkj slfgj sl;dfgkj sl;dfgjk sldgfkj jiurwe alskjfa;sd fasdfEQthisijhhssdinthegdsfg27} \lVert f \rVert_{\YY_\mu} = \sum_\xi \lVert e^{\ee(1+\mu-y)_+|\xi|}f_\xi\rVert_{\SL_\mu} \end{equation} and then for $t$ as in \eqref{EQthisijhhssdinthegdsfg:time:restrict} we define the analytic $Y$ norm by \begin{align} \label{EQthisijhhssdinthegdsfg25} \lVert f\rVert_{\YY(t)}= \sup_{\mu<\mu_0-\gamma t} &\biggl(\sum_{0\le i+j\le1} \norm{\partial_x^i(y\partial_{y})^jf}_{Y_\mu} +\sum_{i+j=2} (\mu_0-\mu-\gamma t)^\alpha \norm{\partial_x^i(y\partial_{y})^jf}_{Y_\mu} \biggr) \,. \end{align} Note the different time weights when comparing the highest order terms in \eqref{EQthisijhhssdinthegdsfg17} and \eqref{EQthisijhhssdinthegdsfg25}. We refer to Lemma~\ref{ana:rec:l} for a useful analyticity recovery estimate for the $Y_\mu$ norm. \par We also define a weighted $L^2$ norm (with respect to both $x$ and $y$) by \begin{equation} \llabel{sldkfjg;ls dslksfgj s;lkdfgj ;sldkfjg ;sldkfgj s;ldfkgj s;ldfkgj gfisasdoifaghskcx,.bvnliahglidhgsd gs sdfgh sdfg sldfg sldfgkj slfgj sl;dfgkj sl;dfgjk sldgfkj jiurwe alskjfa;sd fasdfEQthisijhhssdinthegdsfg19} \lVert f\rVert_{S}^2 = \norm{y f}_{L^2(y\geq 1/2)}^2 = \sum_\xi \lVert y f_\xi\rVert_{L^2(y\ge 1/2)}^2 \end{equation} and a weighted version of the Sobolev $H^3$ norm as \begin{align} \norm{f}_{\ZZZ} = \sum_{0\le i+j\le3} \norm{\partial_x^i \partial_y^j f}_S \, . \llabel{Znan ribič privesla od une stráni,
opomni ga, kak sam sebe pozabi,
kako povsod ga išejo kristjani,
kak z vjetimi Valjhun srditi rabi,
prijazno delj mu tam ostati brani,
stopiti k sebi ga v čolnič povabi,
de ga pripelje v varniši zavetje;
vda Črtomír se v to, kar ribič svétje.sldkfjg;ls dslksfgj s;lkdfgj ;sldkfjg ;sldkfgj s;ldfkgj s;ldfkgj gfisasdoifaghskcx,.bvnliahglidhgsd gs sdfgh sdfg sldfg sldfgkj slfgj sl;dfgkj sl;dfgjk sldgfkj jiurwe alskjfa;sd fasdfEQthisijhhssdinthegdsfg54} \end{align} Further below, it is convenient to also use a weighted $L^2$ in $y$, $\ell^1$ in $\xi$ norm $S_\mu$ given by \begin{equation} \llabel{sldkfjg;ls dslksfgj s;lkdfgj ;sldkfjg ;sldkfgj s;ldfkgj s;ldfkgj gfisasdoifaghskcx,.bvnliahglidhgsd gs sdfgh sdfg sldfg sldfgkj slfgj sl;dfgkj sl;dfgjk sldgfkj jiurwe alskjfa;sd fasdfEQthisijhhssdinthegdsfg20} \lVert f \rVert_{S_\mu} = \sum_\xi \lVert y f_\xi\rVert_{L^2(y\ge 1+\mu)} \, . \end{equation} Lastly, for fixed $\mu_0,\gamma>0$, and with $t$ which obeys \eqref{EQthisijhhssdinthegdsfg:time:restrict}, we introduce the notation \begin{align} \NORM{\omega}_t = \norm{\omega}_{X(t)} + \norm{\omega}_{\YY(t)} + \norm{\omega}_{\ZZZ} \notag \end{align} for the {\em cumulative} time-dependent norm used in this paper. \par \startnewsection{Main results}{sec04} \par We denote by $\omega = \nabla^\perp \cdot u$ the scalar vorticity associated to the the velocity field $u$, where $\nabla^{\perp}=(-\partial_{y},\partial_{x})$. The following is the main result of the paper. \par \cole \begin{theorem} \label{T01} Let $\mu_0>0$ and assume that $\omega_0$ is such that \begin{align} \sum_{i+j\leq 2} \norm{\partial_x^i (y \partial_y)^j \omega_0}_{X_{\mu_0}} + \sum_{i+j\leq 2} \norm{\partial_x^i (y \partial_y)^j \omega_0}_{Y_{\mu_0}} + \sum_{i+j\leq 4} \norm{\partial_x^i (y \partial_y)^j \omega_0}_{S} \leq M <\infty\,. \label{EQthisijhhssdinthegdsfg:omega:0} \end{align} Then there exists a $\gamma>0$ and a time $T>0$ depending on $M$ and $ \mu_0$, such that the solution $\omega$ to the system \eqref{EQthisijhhssdinthegdsfg:vot} satisfies \begin{equation} \label{EQthisijhhssdinthegdsfg35} \sup_{t\in [0,T]} \NORM{\omega(t)}_t \le C M \, . \end{equation} \end{theorem} \colb \par The above result immediately implies the following statement. \cole \begin{theorem} \label{T02} Let $\omega_0$ be as in Theorem~\ref{T01}. Denote by $u^\nu$ the solution of the Navier-Stokes equation \eqref{EQthisijhhssdinthegdsfg:01}--\eqref{EQthisijhhssdinthegdsfg:03b} with viscosity $\nu>0$, defined on $[0,T]$, where $T$ is as given in Theorem~\ref{T01}. Also, denote by $\bar u$ the solution of the Euler equations with initial datum $\omega_0$, which is defined globally in time. Then we have \begin{align} \lim_{\nu \to 0} \sup_{t \in [0,T]} \norm{u^\nu(t) - \bar u(t)}_{L^2(\HH)} = 0 \,. \notag \end{align} \end{theorem} \colb \par The proof of Theorem~\ref{T01} is given in Section~\ref{sec05}, while the proof of Theorem~\ref{T02} is given in Section~\ref{sec:inviscid}. \par \begin{remark} \label{R05} Note that the condition on the initial datum in both theorems depends on $\nu$ since the first norm in \eqref{EQthisijhhssdinthegdsfg:omega:0}, the $X$ norm, depends on it. However, it is easy to find sufficient $\nu$-independent conditions which guarantee the bound. For instance, by using $w(y) \leq 1$ we see that a sufficient condition for \begin{equation} \sum_{i+j\leq 2} \norm{\partial_x^i (y \partial_y)^j \omega_0}_{X_{\mu_0}} \leq \frac{M}{3} \llabel{Želi dat Črtomír mu povračilo,
al v vojski dnarji so bili razdani;
de Staroslav, se spomni, z Bogomilo
mu v skrivnem kraji tovor zláta hrani,
nju poiskati da mu naročilo,
in da mu prstan sámo njima znani,
de bo pri njima storil mu resnico; -
prinesti zláta reče četrtníco.sldkfjg;ls dslksfgj s;lkdfgj ;sldkfjg ;sldkfgj s;ldfkgj s;ldfkgj gfisasdoifaghskcx,.bvnliahglidhgsd gs sdfgh sdfg sldfg sldfgkj slfgj sl;dfgkj sl;dfgjk sldgfkj jiurwe alskjfa;sd fasdfEQthisijhhssdinthegdsfg48} \end{equation} to hold is that \begin{align} \sum_{i+j\leq 2} \sum_{\xi \in \ZZ} \biggl( \sup_{y\in\Omega_{\mu_0}} \bigl| e^{\ee(1+\mu_0-y)_+|\xi|}\partial_x^i (y \partial_y)^j \omega_{0,\xi}(y)\bigr| \biggr) \leq \frac{M}{C} \, , \label{EQthisijhhssdinthegdsfg52} \end{align} for a sufficiently large universal constant $C>0$. A sufficient condition for \eqref{EQthisijhhssdinthegdsfg52} is \begin{align} \sum_{\xi \in \ZZ} \biggl( \sup_{y\in\Omega_{\bar \mu_0}} \bigl| e^{\ee(1+\bar\mu_0-y)_+|\xi|}\omega_{0,\xi}(y)\bigr| \biggr) \leq \frac{M}{C} \llabel{sldkfjg;ls dslksfgj s;lkdfgj ;sldkfjg ;sldkfgj s;ldfkgj s;ldfkgj gfisasdoifaghskcx,.bvnliahglidhgsd gs sdfgh sdfg sldfg sldfgkj slfgj sl;dfgkj sl;dfgjk sldgfkj jiurwe alskjfa;sd fasdfEQthisijhhssdinthegdsfg53} \end{align} where $\bar \mu_0>\mu_0$. \end{remark} \par \subsection{Vorticity formulation} In this paper, we use the vorticity formulation of the Navier-Stokes equations~\eqref{EQthisijhhssdinthegdsfg:01}--\eqref{EQthisijhhssdinthegdsfg:03b}. Taking the curl of the momentum equation~\eqref{EQthisijhhssdinthegdsfg:01}, i.e. by applying $\nabla^\perp \cdot$ to it, gives \begin{equation} \omega_t + u\cdot\nabla\omega -\nu\Delta\omega =0 \, , \label{EQthisijhhssdinthegdsfg:vot} \end{equation} where $u$ is recovered by the Biot-Savart law $u=\nabla^{\perp}\Delta^{-1}\omega$. The boundary condition in this setting was introduced in \cite{Anderson89,Maekawa13,Maekawa14} and is given by \begin{equation} \nu(\partial_{y}+|\partial_{x}|)\omega=\partial_{y}\Delta^{-1}(u\cdot\nabla\omega)|_{y=0} \,. \label{EQthisijhhssdinthegdsfg:bdry:vot} \end{equation} The condition \eqref{EQthisijhhssdinthegdsfg:bdry:vot} follows from $\partial_{t}u_1|_{y=0} =0$, the Biot-Savart law, and the vorticity equation \eqref{EQthisijhhssdinthegdsfg:vot}. \par \subsection{Integral representation of the solution to the Navier-Stokes equations}\label{sec03} For $\xi \in \ZZ$, denote by \begin{equation} N_\xi(s, y) = - (u\cdot\nabla\omega)_\xi(s, y) \llabel{sldkfjg;ls dslksfgj s;lkdfgj ;sldkfjg ;sldkfgj s;ldfkgj s;ldfkgj gfisasdoifaghskcx,.bvnliahglidhgsd gs sdfgh sdfg sldfg sldfgkj slfgj sl;dfgkj sl;dfgjk sldgfkj jiurwe alskjfa;sd fasdfEQthisijhhssdinthegdsfg59} \end{equation} the Fourier transform in the $x$ variable of the nonlinear term in the vorticity formulation of the Navier-Stokes system. Also, let \begin{equation} B_\xi(s)=(\partial_{y}\Delta^{-1}(u\cdot\nabla\omega))_\xi(s)|_{y=0} = - (\partial_{y}\Delta_\xi^{-1}N_\xi(s))|_{y=0} \, , \llabel{sldkfjg;ls dslksfgj s;lkdfgj ;sldkfjg ;sldkfgj s;ldfkgj s;ldfkgj gfisasdoifaghskcx,.bvnliahglidhgsd gs sdfgh sdfg sldfg sldfgkj slfgj sl;dfgkj sl;dfgjk sldgfkj jiurwe alskjfa;sd fasdfEQthisijhhssdinthegdsfg60} \end{equation} where \begin{equation} \Delta_\xi= -\xi^2+\partial_{y}^2 \llabel{sldkfjg;ls dslksfgj s;lkdfgj ;sldkfjg ;sldkfgj s;ldfkgj s;ldfkgj gfisasdoifaghskcx,.bvnliahglidhgsd gs sdfgh sdfg sldfg sldfgkj slfgj sl;dfgkj sl;dfgjk sldgfkj jiurwe alskjfa;sd fasdfEQthisijhhssdinthegdsfg179} \end{equation} is considered with a Dirichlet boundary condition at $y=0$. After taking a Fourier transform in the tangential $x$ variable, the system~\eqref{EQthisijhhssdinthegdsfg:vot}--\eqref{EQthisijhhssdinthegdsfg:bdry:vot} may be rewritten as \begin{align} \label{EQthisijhhssdinthegdsfg:Stokes} \partial_{t}\omega_\xi - \nu\Delta_\xi\omega_\xi &= N_\xi \nonumber\\ \nu(\partial_{y}+|\xi|)\omega_\xi &= B_\xi \,, \end{align} for $\xi \in \ZZ$. Denoting the Green's function for this system by $G_\xi(t, y, z)$, we may represent the solution of this system as \begin{align} \label{EQthisijhhssdinthegdsfg:kernel:est} \omega_\xi (t,y) &= \int_0^\infty G_\xi(t, y, z)\omega_{0\xi}(z)\,dz + \int_0^t \!\! \int_0^\infty G_\xi(t-s, y, z)N_\xi(s, z) \,dzds + \int_0^t G_\xi(t-s, y,0)B_\xi(s)\,ds , \end{align} where $\omega_{0\xi}(z)$ is the Fourier transform of the initial data. A proof of this formulation can be found in \cite{NguyenNguyen18}. \par The next lemma gives an estimate of the Green's function $G_\xi$ of the Stokes system. For its proof, we refer to~\cite[Proposition 3.3 and Section 3.3]{NguyenNguyen18}. \par \cole \begin{lemma}[\cite{NguyenNguyen18}] \label{Green:fun} The Green's function $G_\xi$ for the system~\eqref{EQthisijhhssdinthegdsfg:Stokes} is given by \begin{equation} \label{EQthisijhhssdinthegdsfg30} G_\xi = \tilde H_\xi + R_\xi \, , \end{equation} where \begin{align} \label{EQthisijhhssdinthegdsfg31} \tilde H_\xi(t, y ,z) = \frac{1}{\sqrt{\nu t}}\left(e^{-\frac{(y-z)^2}{4\nu t}} + e^{-\frac{(y+z)^2}{4\nu t}} \right)e^{-\nu\xi^2t} \end{align} is the one dimensional heat kernel for the half space with homogeneous Neumann boundary condition. The residual kernel $R_\xi$ has the property $(\partial_y - \partial_z) R_\xi(t,y,z) = 0 $, meaning that it is a function of $y+z$, and it satisfies the bounds \begin{equation} |\partial_{z}^{k} R_\xi(t, y, z)| \lesssim b^{k+1} e^{-\theta_0 b(y+z)} + \frac{1}{(\nu t)^{(k+1)/2}} e^{-\theta_0\frac{(y+z)^2}{\nu t}}e^{-\frac{\nu\xi^2t}{8}} \comma k\in{\mathbb N}_0 \, , \label{EQthisijhhssdinthegdsfg101} \end{equation} where $\theta_0>0$ is a constant independent of $\nu$. The boundary remainder coefficient $b$ in \eqref{EQthisijhhssdinthegdsfg101} is given by \begin{equation} b=b(\xi, \nu) = |\xi| +\frac{1}{\sqrtnu} \, . \llabel{sldkfjg;ls dslksfgj s;lkdfgj ;sldkfjg ;sldkfgj s;ldfkgj s;ldfkgj gfisasdoifaghskcx,.bvnliahglidhgsd gs sdfgh sdfg sldfg sldfgkj slfgj sl;dfgkj sl;dfgjk sldgfkj jiurwe alskjfa;sd fasdfEQthisijhhssdinthegdsfg259} \end{equation} The implicit constant in \eqref{EQthisijhhssdinthegdsfg101} depends only on $k$ and $\theta_0$. \end{lemma} \colb \par \begin{remark} \label{R01} Based on \eqref{EQthisijhhssdinthegdsfg101}, the residual kernel $R_{\xi}$ satisfies \begin{align} | (y\partial_{y})^k R_\xi (t,y,z)| &\lesssim b \, ((y b)^k +1) e^{-\theta_0b(y+z)} + \biggl(\left(\frac{y}{\sqrt{\nu t}}\right)^k +1 \biggr) \frac{1}{\sqrt{\nu t}}e^{- \theta_0 \frac{(y+z)^2}{\nu t}}e^{-\frac{\nu\xi^2t}{8}} \notag\\ &\lesssim be^{-\frac{\theta_0}{2}b(y+z)} + \frac{1}{\sqrt{\nu t}}e^{- \frac{\theta_0}{2} \frac{(y+z)^2}{\nu t}}e^{-\frac{\nu\xi^2t}{8}} \label{EQthisijhhssdinthegdsfg33} \end{align} for $k \in \{0,1,2\}$, pointwise in $y,z \geq 0$. \end{remark} \par \begin{remark} \label{R03} As explained in \cite{NguyenNguyen18}, the Duhamel formula \eqref{EQthisijhhssdinthegdsfg:kernel:est} holds not just for real values of $y, z \in [0,\infty)$ but in general for all complex values $y,z\in \Omega_{\mu}\cup [1+\mu,\infty)$. In this case, for $y\in \Omega_\mu$, we may find $\theta \in [0,\mu)$ such that $y \in \partial\Omega_\theta$. If $\Im y \geq 0$, the integrals from $0$ to $\infty$ in \eqref{EQthisijhhssdinthegdsfg:kernel:est} become integrals over the complex contour $\gamma_\theta^+ = (\partial \Omega_\theta \cap \{ z \colon \Im z \geq 0 \}) \cup [1+\theta ,\infty)$. A similar contour may be defined for $\Im y < 0$. Moreover, the Green's function $G_\xi(t,y,z)$ from Lemma~\ref{Green:fun}, which appears in \eqref{EQthisijhhssdinthegdsfg:kernel:est}, has a natural extension to the complex domain $\Omega_\mu \cup [1+\mu,\infty)$, by complexifying the heat kernels involved. Since for $y\in \Omega_\mu$ we have $|\Im y| \leq \mu \Re y$, for $\mu$ small, we have that $|y|$ is comparable to $\Re y$. Therefore, the upper bounds we have available for the complexified heat kernel $\tilde H_\xi$ and for the residual kernel $R_\xi$ may be written in terms of $\Re y, \Re z \geq 0$. Because of this, as in \cite{NguyenNguyen18}, when we prove inequalities for the analytic norms $X$ and $Y$ we provide details only for the case when $y$ and $z$ are real-valued. The complex versions of \eqref{EQthisijhhssdinthegdsfg:kernel:est} and Lemma~\ref{Green:fun} only lead to notational complications due to integration over complex paths, and due to having to write $\Re y$, $\Re z$ at the exponentials in all upper bounds. We omit these details. \end{remark} \par \subsection{Main estimates} We denote \begin{align} \mu_1 &= \mu+\frac{1}{4}(\mu_0-\mu-\gamma s) \label{EQthisijhhssdinthegdsfg02} \\ \mu_2 &= \mu+\frac{1}{2}(\mu_0-\mu-\gamma s) \label{EQthisijhhssdinthegdsfg02b} \end{align} and observe that \begin{equation} 0<\mu<\mu_1<\mu_2<\mu_0-\gamma s \, . \llabel{Odločeni so roži kratki dnóvi,
ki pride nanjo pomladanska slana,
al v cvetji jo zapadejo snegovi!
tak mladi deklici, ki zgodnja rana 
srce ji gloda, vsmŕti mir njegovi,
le kratka pot je skoz življenje dana;
al je za majhen čas se združit vredno,
de bi ločitve spet se bala vedno?
De bi od smrti rešil te nesrečne,
in tamkej mili Bog v nebeškem raji
z menoj te, dragi! sklenil čase večne,
pustila vnémar sem željé narslaji,
pustila vnémar dni na sveti srečne,
sem odpovedala se zvezi naji; - 
je uslišana bila molitev moja. -
Ne smem postati jaz nevesta tvoja.
Bogú sem večno čistost obljubila,
in Jezusu, in materi Mariji;
kar doživela let bom še števila
v željá bridkosti, v upa rajskem siji,
nobena me ne bo premogla sila,
bila de svojemu, svetá Mesiji,
nebeškemu bi ženinu nezvesta,
nikdár ne morem tvoja bit nevesta!" sldkfjg;ls dslksfgj s;lkdfgj ;sldkfjg ;sldkfgj s;ldfkgj s;ldfkgj gfisasdoifaghskcx,.bvnliahglidhgsd gs sdfgh sdfg sldfg sldfgkj slfgj sl;dfgkj sl;dfgjk sldgfkj jiurwe alskjfa;sd fasdfEQthisijhhssdinthegdsfg01} \end{equation} \par \begin{lemma}[\bf Main $X$ norm estimate] \label{lem:main:X} \cole With $\mu_1$ and $\mu_2$ as in \eqref{EQthisijhhssdinthegdsfg02} and \eqref{EQthisijhhssdinthegdsfg02b}, the nonlinear term in \eqref{EQthisijhhssdinthegdsfg:kernel:est} is bounded in the $X_\mu$ norm as \begin{align} &(\mu_0-\mu-\gamma s) \sum_{i+j= 2} \left\lVert \partial_x^i (y\partial_{y})^j\int_0^\infty G(t-s, y, z)N(s, z) \,dz\right\rVert_{X_\mu} \notag\\ &\quad + \sum_{i+j \leq 1} \left\lVert \partial_x^i (y\partial_{y})^j\int_0^\infty G(t-s, y, z)N(s, z) \,dz\right\rVert_{X_{\mu_1}} \notag\\ &\quad \quad \les \sum_{i+j \leq 1}\lVert \partial_x^i (y\partial_{y})^j N(s)\rVert_{X_{\mu_2}} + \frac{1}{(\mu_0-\mu-\gamma s)^{1/2}} \sum_{i+j \leq 1} \lVert \partial_x^i \partial_{y}^j N(s)\rVert_{S_{\mu_2}} \, . \label{EQthisijhhssdinthegdsfg701} \end{align} The $X_\mu$ norm of the trace kernel term in \eqref{EQthisijhhssdinthegdsfg:kernel:est} is estimated as \begin{align} &(\mu_0-\mu-\gamma s) \sum_{i+j= 2} \norm{ \partial_x^i (y \partial_y)^j G(t-s, y,0)B(s)}_{X_\mu} \notag\\ &\quad + \sum_{i+j \leq 1 } \norm{ \partial_x^i (y \partial_y)^j G(t-s, y,0)B(s)}_{X_{\mu_1}} \notag\\ &\quad \quad \les \frac{1}{\sqrt{t-s}}\left(\sum_{i\leq1} \lVert\partial_{x}^{i}N(s)\rVert_{\YY_{\mu_1}} + \lVert \partial_x^i N(s)\rVert_{S_{\mu_1}} \right) + \sum_{i\leq 1} \norm{\partial_x^i N(s)}_{X_{\mu_1}} \, . \label{EQthisijhhssdinthegdsfg703} \end{align} Lastly, the initial datum term in \eqref{EQthisijhhssdinthegdsfg:kernel:est} may be bounded in the $X_\mu$ norm as \begin{align} & \sum_{i+j\leq 2} \norm{ \partial_x^i (y \partial_y)^j \int_0^\infty G(t, y,z)\omega_{0}(z) \,dz}_{X_\mu} \notag\\&\indeq \les \sum_{i+j\leq 2}\lVert \partial_x^i (y\partial_{y})^j \omega_0\rVert_{X_{\mu}} + \sum_{i+j\leq 2} \sum_{\xi} \lVert \xi^i \partial_{y}^j \omega_{0\xi}\rVert_{L^\infty(y\ge 1+\mu)} \,. \llabel{Ko je minul, kar misli, de bo v sili
zlatá mu treba, si od mož ga vzame; 
dar ribču da, njim, ki so ga nosili.
"Kar Staroslav zlatá še hrani zame,
daj ga sirotam," reče Bogomili,
se bliža ji, presrčno jo objame,
molče poda desnico ji k slovesi,
solzé stojijo v vsakem mu očesi.
"O čakaj, mi dopolni prošnjo eno!
préd ko se lóčva," Bogomila pravi,
"de mi v skrbeh ne bo srce utopljeno,
de lóžej se bridkosti v bran postavi,
préd, ko greš v Oglej čez goró zeleno,
se pričo mene odpovej zmotnjavi, 
dokler te posveti krst, se zamúdi, 
vodá je blizo, in duhovni tudi."sldkfjg;ls dslksfgj s;lkdfgj ;sldkfjg ;sldkfgj s;ldfkgj s;ldfkgj gfisasdoifaghskcx,.bvnliahglidhgsd gs sdfgh sdfg sldfg sldfgkj slfgj sl;dfgkj sl;dfgjk sldgfkj jiurwe alskjfa;sd fasdfEQthisijhhssdinthegdsfg704} \end{align} \colb \end{lemma} \par The proof of Lemma~\ref{lem:main:X} is given at the end of Section~\ref{secX}. \par \begin{lemma}[\bf Main $\YY$ norm estimate] \label{lem:main:Y} \cole Let $\mu_1$ be as defined in \eqref{EQthisijhhssdinthegdsfg02}. Then the nonlinear term in \eqref{EQthisijhhssdinthegdsfg:kernel:est} is bounded in the $\YY_\mu$ norm as \begin{align} & (\mu_0-\mu-\gamma s) \sum_{i+j= 2} \left\lVert \partial_x^i (y\partial_{y})^j\int_0^\infty G(t-s, y, z)N(s, z) \,dz\right\rVert_{\YY_\mu} \notag\\ &\quad + \sum_{i+j\leq 1} \left\lVert \partial_x^i (y\partial_{y})^j\int_0^\infty G(t-s, y, z)N(s, z) \,dz\right\rVert_{\YY_{\mu_1}} \notag\\ &\quad \quad \les \sum_{i+j\leq 1}\lVert \partial_x^i (y\partial_{y})^j N(s)\rVert_{\YY_{\mu_1}} +\sum_{i+j\leq 1} \lVert \partial_x^i \partial_{y}^j N(s)\rVert_{S_{\mu_1}} \, . \label{EQthisijhhssdinthegdsfg706} \end{align} The $\YY_\mu$ norm of the trace kernel term in \eqref{EQthisijhhssdinthegdsfg:kernel:est} is estimated as \begin{align} & (\mu_0-\mu - \gamma s) \sum_{i+j = 2} \norm{ \partial_x^i (y \partial_y)^j G(t-s, y,0)B(s)}_{\YY_\mu} + \sum_{i+j\leq 1} \norm{ \partial_x^i (y \partial_y)^j G(t-s, y,0)B(s)}_{\YY_{\mu_1}} \notag \\ & \quad \les \sum_{i\leq1}\left( \lVert\partial_{x}^{i}N(s)\rVert_{\YY_{\mu_1}} + \lVert \partial_{x}^{i} N(s)\rVert_{S_\mu} \right) \, . \label{EQthisijhhssdinthegdsfg707} \end{align} Lastly, the initial datum term in \eqref{EQthisijhhssdinthegdsfg:kernel:est} may be bounded as \begin{align} & \sum_{i+j\leq 2} \norm{ \partial_x^i (y \partial_y)^j \int_0^\infty G(t, y,z)\omega_{0}(z) \,dz}_{\YY_\mu} \notag \\&\indeq \les \sum_{i+j\leq 2}\lVert \partial_x^i (y\partial_{y})^j \omega_0\rVert_{\YY_{\mu}} + \sum_{i+j\leq 2} \sum_{\xi} \lVert \xi^i \partial_{y}^j \omega_{0\xi}\rVert_{L^1(y\ge 1+\mu)} \, . \label{EQthisijhhssdinthegdsfg708} \end{align} \colb \end{lemma} The proof of Lemma~\ref{lem:main:Y} is provided at the end of Section~\ref{secL}. The next lemma provides inequalities for the nonlinearity. \colb \par \begin{lemma}[\bf The $X$, $\YY$, and $S_\mu$ norm estimates for the nonlinearity] \label{lem:main:N} \cole For any $\mu \in (0,\mu_0-\gamma s)$ we have the inequalities \begin{align} \sum_{i+j\leq1}\lVert \partial_{x}^{i}(y\partial_{y})^jN(s)\rVert_{X_{\mu}} &\lesssim \sum_{i\leq 1}\left(\lVert \partial_{x}^i\omega\rVert_{\YY_{\mu}} +\lVert \partial_{x}^i\omega\rVert_{S_{\mu}}\right) \sum_{i+j\leq2}\lVert \partial_{x}^i(y\partial_{y})^j\omega\rVert_{X_{\mu}} \notag \\&\indeq + \sum_{i\leq 2}\left(\lVert \partial_{x}^i\omega\rVert_{\YY_{\mu}} +\lVert \partial_{x}^i\omega\rVert_{S_{\mu}}\right) \sum_{i+j\leq1}\lVert \partial_{x}^i(y\partial_{y})^j\omega\rVert_{X_{\mu}} \notag \\&\indeq + \lVert \omega\rVert_{X_{\mu}} \sum_{i+j=1}\lVert \partial_{x}^i(y\partial_{y})^j\omega\rVert_{X_{\mu}} \label{EQthisijhhssdinthegdsfg709} \end{align} and \begin{align} \sum_{i+j\leq1} \lVert \partial_x^i(y\partial_{y})^jN(s)\rVert_{\YY_{\mu}} &\lesssim \sum_{i\leq 1}\left(\lVert \partial_x^i\omega\rVert_{\YY_{\mu}} + \lVert \partial_x^i\omega\rVert_{S_{\mu}}\right) \sum_{i+j\leq2}\lVert \partial_x^i(y\partial_{y})^j\omega\rVert_{\YY_{\mu}} \notag \\&\indeq + \sum_{i\leq 2}\left(\lVert \partial_x^i\omega\rVert_{\YY_{\mu}} + \lVert \partial_x^i\omega\rVert_{S_{\mu}}\right) \sum_{i+j\leq1}\lVert \partial_x^i(y\partial_{y})^j\omega\rVert_{\YY_{\mu}} \notag \\&\indeq +\lVert \omega\rVert_{X_{\mu}} \sum_{i+j=1}\lVert \partial_x^i(y\partial_{y})^j\omega\rVert_{\YY_{\mu}} \, . \llabel{sldkfjg;ls dslksfgj s;lkdfgj ;sldkfjg ;sldkfgj s;ldfkgj s;ldfkgj gfisasdoifaghskcx,.bvnliahglidhgsd gs sdfgh sdfg sldfg sldfgkj slfgj sl;dfgkj sl;dfgjk sldgfkj jiurwe alskjfa;sd fasdfEQthisijhhssdinthegdsfg710} \end{align} For the Sobolev norm we have the estimate \begin{align} \sum_{i+j \leq 1} \norm{\partial_x^i \partial_y^j N(s)}_{S_\mu} &\les \NORM{\omega}_s \sum_{i+j\leq 3} \norm{\partial_x^i \partial_y^j \omega}_{S} \, . \label{EQthisijhhssdinthegdsfg711} \end{align} \colb \end{lemma} The proof of Lemma~\ref{lem:main:N} is given at the end of Section~\ref{sec-N}. Lastly, the following statement provides the estimate of the Sobolev part of the norm. \cole \begin{lemma} \label{L12} For any $0 < t < \frac{\mu_0}{2\gamma}$ the estimate \begin{align} \sum_{i+j\le 3} \lVert y \partial_x^i\partial_{y}^j\omega(t)\rVert^2_{L^2(y\geq 1/2)} \lesssim \Bigl(1+ t \sup_{s\in[0,t]} \NORM{\omega(s)}_s^3 \Bigr) e^{C t (1+ \sup_{s\in[0,t]} \NORM{\omega(s)}_s)} \sum_{i+j\le 3} \lVert y \partial_x^i\partial_{y}^j\omega_0\rVert^2_{L^2(y\geq 1/4)} \label{EQthisijhhssdinthegdsfg286} \end{align} holds, where $C>0$ is a constant independent of $\gamma$. \end{lemma} \colb \par This statement follows from Lemma~\ref{L12a} below. \par \subsection{Closing the a~priori estimates}\label{sec05} In this section, we provide the a~priori estimates needed to prove Theorem~\ref{T01}. \par \begin{proof}[Proof of Theorem~\ref{T01}] Define \begin{align} \tilde M = \sum_{i+j\leq 2}\lVert \partial_x^i (y\partial_{y})^j \omega_0\rVert_{X_{\mu_0}} + \sum_{i+j\leq 2} \sum_{\xi} \lVert \partial_x^i \partial_{y}^j \omega_{0,\xi}\rVert_{L^\infty(y\ge 1+\mu_0)} \notag \end{align} and \begin{align} \overline M &= \sum_{i+j\leq 2}\lVert \partial_x^i (y\partial_{y})^j \omega_0\rVert_{\YY_{\mu_0}} + \sum_{i+j\leq 2} \sum_{\xi} \lVert \partial_x^i \partial_{y}^j \omega_{0,\xi}\rVert_{L^1(y\ge 1+\mu_0)} \, . \notag \end{align} Note that by \eqref{EQthisijhhssdinthegdsfg:omega:0}, \eqref{EQthisijhhssdinthegdsfg:Shaq} below, and Lemma~\ref{L06}, we have \begin{align} \tilde M + \overline M \les M\, . \notag \end{align} \par Let $t < \frac{\mu_0}{2\gamma}$, $s\in (0,t)$, and $\mu < \mu_0-\gamma t$. First we estimate the $X(t)$ norm of $\omega(t)$. From the mild formulation \eqref{EQthisijhhssdinthegdsfg:kernel:est}, the estimates \eqref{EQthisijhhssdinthegdsfg701}--\eqref{EQthisijhhssdinthegdsfg703}, and the bounds \eqref{EQthisijhhssdinthegdsfg709}--\eqref{EQthisijhhssdinthegdsfg711} for the nonlinear term, we obtain \begin{align} \sum_{i+j=2} \norm{\partial_x^i (y \partial_y)^j \omega(t)}_{X_\mu} &\les \int_0^t \left(\frac{\NORM{\omega(s)}_s^2}{(\mu_0-\mu-\gamma s)^{3/2+\alpha}} + \frac{1}{\sqrt{t-s}} \frac{\NORM{\omega(s)}_s^2}{(\mu_0-\mu-\gamma s)^{1+\alpha}}\right) \,ds + \tilde M \notag\\ & \les \sup_{0\leq s \leq t} \NORM{\omega(s)}_s^2 \left( \frac{1}{\gamma (\mu_0-\mu-\gamma t)^{1/2+\alpha}} + \frac{1}{\sqrt{\gamma} (\mu_0-\mu-\gamma t)^{1/2+\alpha}} \right) + \tilde M \notag\\ &\les \frac{\sup_{0\leq s \leq t} \NORM{\omega(s)}_s^2}{\sqrt{\gamma} (\mu_0-\mu-\gamma t)^{1/2+\alpha}} + \tilde M \, , \label{EQthisijhhssdinthegdsfg305} \end{align} where we used Lemma~\ref{int:t}. Similarly, we obtain \begin{align} \sum_{i+j\leq 1} \norm{\partial_x^i (y \partial_y)^j \omega(t)}_{X_\mu} &\les \int_0^t \biggl(\frac{\NORM{\omega(s)}_s^2}{(\mu_0-\mu-\gamma s)^{1/2+\alpha}} + \frac{1}{\sqrt{t-s}} \frac{\NORM{\omega(s)}_s^2}{(\mu_0-\mu-\gamma s)^{\alpha}} \biggr) \,ds + \tilde M \notag\\ &\les \frac{\sup_{0\leq s \leq t} \NORM{\omega(s)}_s^2}{\sqrt{\gamma}} + \tilde M \, , \label{EQthisijhhssdinthegdsfg305a} \end{align} where we again used Lemma~\ref{int:t}. Combining \eqref{EQthisijhhssdinthegdsfg305} and \eqref{EQthisijhhssdinthegdsfg305a}, we obtain \begin{align} \norm{\omega(t)}_{X(t)} \les \frac{\sup_{0\leq s \leq t} \NORM{\omega(s)}_s^2}{\sqrt{\gamma}} + \tilde M \, . \label{EQthisijhhssdinthegdsfg305b} \end{align} \par Next we estimate the $\YY(t)$ norm of $\omega(t)$. From the mild formulation \eqref{EQthisijhhssdinthegdsfg:kernel:est}, the estimates \eqref{EQthisijhhssdinthegdsfg706}--\eqref{EQthisijhhssdinthegdsfg708}, and the bounds \eqref{EQthisijhhssdinthegdsfg709}--\eqref{EQthisijhhssdinthegdsfg711} for the nonlinear term, we obtain \begin{align} \sum_{i+j=2} \norm{\partial_x^i (y \partial_y)^j \omega(t)}_{\YY_\mu} &\les \int_0^t \frac{\NORM{\omega(s)}_s^2}{(\mu_0-\mu-\gamma s)^{1+\alpha}} \,ds + \overline M \les \frac{\sup_{0\leq s \leq t} \NORM{\omega(s)}_s^2}{ \gamma (\mu_0-\mu-\gamma t)^{\alpha}} + \overline M \, . \label{EQthisijhhssdinthegdsfg306} \end{align} For the lower order derivatives we obtain \begin{align} \sum_{i+j\leq 1} \norm{\partial_x^i (y \partial_y)^j \omega(t)}_{\YY_\mu} &\les \int_0^t \frac{\NORM{\omega(s)}_s^2}{(\mu_0-\mu-\gamma s)^{\alpha}} \,ds + \overline M \les \frac{\sup_{0\leq s \leq t} \NORM{\omega(s)}_s^2}{ \gamma} + \overline M \, . \label{EQthisijhhssdinthegdsfg306a} \end{align} By combining \eqref{EQthisijhhssdinthegdsfg306}--\eqref{EQthisijhhssdinthegdsfg306a}, we arrive at \begin{align} \norm{\omega(t)}_{\YY(t)} \les \frac{\sup_{0\leq s \leq t} \NORM{\omega(s)}_s^2}{\gamma} + \overline M \, . \label{EQthisijhhssdinthegdsfg306b} \end{align} To conclude, let \begin{align*} \mathring M = \sum_{i+j\le 3} \lVert \partial_x^i\partial_{y}^j\omega_0\rVert_{L^2(y\ge 1/4)} \les M \, . \end{align*} Recall that the Sobolev estimate \eqref{EQthisijhhssdinthegdsfg286} yields \begin{align} \norm{\omega(t)}_{\ZZZ} \les \left(1+ \frac{\sup_{s\in[0,t]} \NORM{\omega(s)}_s^{3/2}}{\sqrt{\gamma}} \right) e^{\frac{C}{\gamma} (1+ \sup_{s\in[0,t]} \NORM{\omega(s)}_s)} \mathring M \, , \label{EQthisijhhssdinthegdsfg307} \end{align} and this inequality holds pointwise in time for $t < \frac{\mu_0}{2\gamma}$. The constant $C$ and the implicit constants in $\les$ are independent of $\gamma$. \par Combining \eqref{EQthisijhhssdinthegdsfg305b}, \eqref{EQthisijhhssdinthegdsfg306b}, and \eqref{EQthisijhhssdinthegdsfg307}, and taking the supremum in time for $t < \frac{\mu_0}{2\gamma}$, we arrive at \begin{align} \sup_{t \in [0,\frac{\mu_0}{2\gamma}]} \NORM{\omega(t)}_t &\leq C (\tilde M + \overline M ) + \frac{C \sup_{t \in [0,\frac{\mu_0}{2\gamma}]} \NORM{\omega(t)}_t^2}{\sqrt{\gamma}} \notag\\ &\quad + C \mathring M \left(1+ \frac{\sup_{t \in [0,\frac{\mu_0}{2\gamma}]} \NORM{\omega(t)}_t^{3/2}}{\sqrt{\gamma}} \right) e^{\frac{C \mu_0}{\gamma} (1+ \sup_{t \in [0,\frac{\mu_0}{2\gamma}]} \NORM{\omega(t)}_t)} \, , \notag \end{align} where $C\geq 1$ is a constant that depends only on $\mu_0$. Using a standard barrier argument, one may show that if $\gamma$ is chosen sufficiently large, in terms of $\tilde M, \overline M, \mathring M, \mu_0$, we obtain \begin{align} \sup_{t \in [0,\frac{\mu_0}{2\gamma}]} \NORM{\omega(t)}_t \leq 2 C( \tilde M + \overline M + \mathring M ) \, , \notag \end{align} concluding the proof of the theorem. \end{proof} \par \begin{remark} \label{R08} In order to justify the above a~priori estimates, for each $\delta \in (0,1]$, we apply them on the approximate system \begin{align} \omega_t + u^{\delta}\cdot\nabla\omega -\nu\Delta\omega =0 \, , \label{EQthisijhhssdinthegdsfg03} \end{align} where $u^{\delta}$ is a regularization of the velocity in the Biot-Savart law \eqref{EQthisijhhssdinthegdsfg123}--\eqref{EQthisijhhssdinthegdsfg124}. The boundary condition \eqref{EQthisijhhssdinthegdsfg:bdry:vot} becomes $ \nu(\partial_{y}+|\partial_{x}|)\omega=\partial_{y}\Delta^{-1}(u^\delta \cdot\nabla\omega)|_{y=0}$, and the initial condition is replaced by an analytic approximation. The regularized velocity $u^\delta$ is obtained from $\omega$ by a heat extension to time $\delta$, using a homogeneous version of the boundary condition \eqref{EQthisijhhssdinthegdsfg:Stokes}, and then computing the Biot-Savart law for this regularized vorticity. Now, in order to justify our a~priori estimates, we approximate the initial datum $\omega_0$ with an entire one $\omega_0^\delta$. We may show using the approach in this paper that the system \eqref{EQthisijhhssdinthegdsfg03} with entire initial data has a solution which is entire for all time. Then, we perform all the estimates in the present paper on \eqref{EQthisijhhssdinthegdsfg03}, obtaining uniform-in-$\delta$ upper bounds for the norm $\NORM{\cdot}_t$ for all $t\in [0,\frac{\mu_0}{2\gamma})$, thus allowing us to pass those bounds to the limit $\delta\to 0$. \end{remark} \par \subsection{Inviscid limit} \label{sec:inviscid} This section is devoted to the proof of Theorem~\ref{T02}. \begin{proof}[Proof of Theorem~\ref{T02}] Let $T>0$ be as in Theorem~\ref{T01}. In view of the Kato criterion~\cite{Kato84b}, we only need to consider \begin{align} &\nu \int_{0}^T \!\!\!\! \int_{\HH} |\nabla u|^2 \,dxdyds = \nu \int_{0}^T\!\!\!\! \int_{\HH} |\omega|^2 \,dxdyds \notag \\&\indeq =\nu \int_{0}^T\!\!\!\! \int_{\{y\le1/2\}} |\omega|^2 \,dxdyds + \nu \int_{0}^T\!\!\!\! \int_{\{y\ge 1/2\}} |\omega|^2 \,dxdyds \notag \\&\indeq \les \sqrt{\nu} \int_{0}^T \sum_\xi \| e^{\ee(1-y)_+ |\xi|} w(y) \omega_\xi(s)\|_{L^\infty (y \leq 1/2)} \| e^{\ee(1-y)_+ |\xi|} \omega_\xi(s)\|_{L^1 (y \leq 1/2)} \,ds + \nu \int_{0}^T \norm{\omega(s)}_S^2 \, ds \notag \\&\indeq \les \sqrt{\nu} \int_{0}^T \norm{\omega(s)}_{X(s)} \norm{\omega(s)}_{\YY(s)} \,ds + \nu \int_{0}^T \norm{\omega(s)}_S^2 \, ds \notag \\&\indeq \les \sqrt{\nu} C M \, . \llabel{sldkfjg;ls dslksfgj s;lkdfgj ;sldkfjg ;sldkfgj s;ldfkgj s;ldfkgj gfisasdoifaghskcx,.bvnliahglidhgsd gs sdfgh sdfg sldfg sldfgkj slfgj sl;dfgkj sl;dfgjk sldgfkj jiurwe alskjfa;sd fasdfEQthisijhhssdinthegdsfg309} \end{align} Here we used that \colW $\sqrt{\nu} \les w(y)$ \colb and have appealed to the bound \eqref{EQthisijhhssdinthegdsfg35}. By~\cite{Kato84b} it follows that the inviscid limit holds in the topology of $L^\infty(0,T;L^2(\HH))$. \end{proof} \section{Estimates for the $X$ analytic norm} \label{secX} \par Throughout this section we fix $t >0$ and $s\in (0,t)$ and provide the $X$~norm estimate of the three integrals appearing in~\eqref{EQthisijhhssdinthegdsfg:kernel:est}. We first consider the kernel \begin{equation} \label{EQthisijhhssdinthegdsfg36} H_\xi(t,y,z) = \frac{1}{\sqrt{\nu t}}e^{-\frac{(y-z)^2}{4\nu t}} e^{-\nu\xi^2t} \, . \end{equation} \par In the following lemma, we estimate the derivatives up to order one of the integral involving the nonlinearity. \par \cole \begin{lemma} \label{high:y:x} Assume that $\mu$ and $\tilde \mu$ obey the conditions \begin{align} 0 < \mu < \tilde \mu < \mu_0-\gamma s, \qquad \tilde \mu - \mu \geq \frac{1}{C} (\mu_0 - \mu -\gamma s) \, , \label{EQthisijhhssdinthegdsfg:mu:cond} \end{align} for some constant $C\geq 1$. Then, for $(i,j)=(0,0),(1,0),(0,1)$, we have \begin{align} & \left\lVert \partial_x^i (y\partial_{y})^j \int_0^\infty H(t-s, y, z)N(s, z) \,dz\right\rVert_{X_\mu} \notag \\&\indeq \lesssim \lVert \partial_x^i (y\partial_{y})^j N(s)\rVert_{X_{\tilde \mu}} + \lVert N(s)\rVert_{X_{\tilde \mu}} + \frac{1}{(\mu_0-\mu-\gamma s)^{1/2}} \sum_\xi \lVert \partial_x^i \partial_{y}^j N_\xi(s) \rVert_{L^2(y\ge1+\tilde \mu)} \, . \label{EQthisijhhssdinthegdsfg38} \end{align} \end{lemma} \colb \par \begin{remark} \label{R06} Inspecting the proof of Lemma~\ref{high:y:x} below, we note that only the following properties of the kernel $H_\xi(y,z,t)$ are used. First, we use that either $\partial_y H_\xi(y,z,t) = \partial_z H_\xi(y,z,t)$ or $\partial_y H_\xi(y,z,t) = -\partial_z H_\xi(y,z,t)$, the property allowing us to transfer $y$ derivatives to $z$ derivatives. For the terms $I_1$ and $I_2$ we use \begin{align} \norm{\chi_{\{0\leq y \leq1+\mu\}} \chi_{\{0\leq z \leq 3y/4\}} \colW \frac{w(y)}{w(z)} \colb \bigl( | H_\xi(t,y,z) | + | y \partial_y H_\xi(t,y,z)| \bigr)}_{L^\infty_y L^1_z} \les 1 \, , \label{EQthisijhhssdinthegdsfg04} \end{align} for the term $I_3$ we need \begin{align} \norm{e^{\ee (z-y)_+ |\xi|} H_\xi(t,y,z)}_{L^\infty_y L^1_z} \les 1 \, , \label{EQthisijhhssdinthegdsfg08} \end{align} while for the term $I_4$ we additionally use \begin{align} \norm{\chi_{\{0\leq y \leq 1+\mu\}} \chi_{\{z\geq 1+\tilde \mu\}} e^{\ee (z-y)_+ |\xi|} H_\xi(t,y,z)}_{L^\infty_y L^2_z} \les \frac{1}{\sqrt{\tilde \mu-\mu}} \, . \label{EQthisijhhssdinthegdsfg39} \end{align} Observe that the kernel $\tilde H_\xi(t,y,z) - H_\xi(t,y,z) = \frac{1}{\sqrt{\nu t}}e^{-\frac{(y+z)^2}{4\nu t}} e^{-\nu\xi^2t}$ also obeys these three properties, and thus Lemma~\ref{high:y:x} holds with $H(t,y,z)$ replaced by $\tilde H(t,y,z)$. \end{remark} \par \begin{proof}[Proof of Lemma~\ref{high:y:x}] Let $y\in \Omega_{\mu}$. For simplicity, we only work with $y\in{\mathbb R}$; an adjustment for the complex case is straight-forward and leads only to notational complications. \par We start with the proof of \eqref{EQthisijhhssdinthegdsfg38} in the case $(i,j) = (0,1)$. Let $\psi \colon \mathbb{R}^+\rightarrow\mathbb{R}^+$ be a smooth non-increasing cut-off function such that $\psi(x)=1$ for $0\le x \le1/2$, and $\psi(x)=0$ for $x\ge3/4$. We first decompose \begin{align} &y \partial_y \int_0^\infty H_\xi(t-s, y, z)N_\xi(s, z) \,dz \notag \\& = - y \int_0^\infty \partial_z H_\xi(t-s,y,z) N_\xi(s,z) \,dz \notag \\& = - y\int_0^{\infty} \psi\left(\frac{z}{y}\right) \partial_z H_\xi(t-s, y, z)N_\xi(s, z) \,dz \notag \\&\indeq - y \int_0^{\infty} \left(1-\psi\left(\frac{z}{y}\right)\right)\partial_z H_\xi(t-s, y, z)N_\xi(s, z) \,dz \notag \\& = - y\int_0^{3y/4} \psi\left(\frac{z}{y}\right) \partial_z H_\xi(t-s, y, z)N_\xi(s, z) \,dz - \int_{y/2}^{3y/4} \psi'\left(\frac{z}{y}\right) H_\xi(t-s, y, z)N_\xi(s, z) \,dz \notag \\&\indeq + y \int_{y/2}^{1+\mu} \left(1-\psi\left(\frac{z}{y}\right)\right) H_\xi(t-s, y, z) \partial_z N_\xi(s, z) \,dz + y \int_{1+\mu}^{\infty} H_\xi(t-s, y, z) \partial_z N_\xi(s, z) \,dz \notag \\& =I_1+I_2+I_3+I_4 \, . \label{EQthisijhhssdinthegdsfg:dec:non} \end{align} The first three terms represent contributions from the analytic region and the last term from the Sobolev region. \par In order to bound $I_{1}$, we compute the derivative of $H_\xi$ as \begin{align} y\partial_{z} H_\xi = - y\partial_{y} H_\xi &= y\frac{(y-z)}{2\nu(t-s)}\frac{1}{\sqrt{\nu (t-s)}} e^{-\frac{(y-z)^2}{4\nu (t-s)}} e^{-\nu\xi^2(t-s)} \, . \label{EQthisijhhssdinthegdsfg10} \end{align} By \begin{equation} |y| \leq 4 |y-z| \comma 0\leq z\leq \frac{3y}{4} \label{EQthisijhhssdinthegdsfg09} \end{equation} we arrive at \begin{align} |y\partial_{z} H_\xi| \lesssim \frac{1}{\sqrt{\nu (t-s)}}e^{-\frac{(y-z)^2}{8\nu (t-s)}} e^{-\nu\xi^2(t-s)} \comma 0\leq z\leq \frac{3y}{4} \label{EQthisijhhssdinthegdsfg11} \end{align} and therefore \begin{align} |I_1| \les \int_0^{3y/4} \frac{1}{\sqrt{\nu (t-s)}} e^{-\frac{(y-z)^2}{8\nu (t-s)}} e^{-\nu\xi^2(t-s)} |N_\xi(s,z)| \,dz \, . \label{EQthisijhhssdinthegdsfg12} \end{align} Next, we claim that the weight function obeys the estimate \colW \begin{equation} \ww(y) \leq w(z) e^{\frac{(y-z)^2}{64\nu (t-s)}} \comma 0 \leq z \leq \frac{3y}{4} \, . \label{EQthisijhhssdinthegdsfg51} \end{equation} \colb In order to prove \eqref{EQthisijhhssdinthegdsfg51} we use that $t-s \leq T \leq 1$ and estimate \begin{align} \ww(y)e^{-\frac{(y-z)^2}{64\nu (t-s)}} & \leq \ww(y)e^{-\frac{y^2}{256\nu (t-s)}} \leq \ww(y)e^{-\frac{y^2}{256\nu}} \les \ww(y)e^{-\frac{y}{16\sqrtnu}} \les \sqrt\nu \, , \label{EQthisijhhssdinthegdsfg44} \end{align} where we used \eqref{EQthisijhhssdinthegdsfg09} in the first and Remark~\ref{R04}(e) in the last step. Then \eqref{EQthisijhhssdinthegdsfg51} follows from $\sqrtnu\les w(z)$ by Remark~\ref{R04}(c). Using \eqref{EQthisijhhssdinthegdsfg12}, \eqref{EQthisijhhssdinthegdsfg51}, and \begin{equation} e^{\ee(1+\mu-y)_+|\xi|} \leq e^{\ee(1+\mu-z)_+|\xi|} \comma 0\le z\le y \, , \label{EQthisijhhssdinthegdsfg06} \end{equation} we obtain \begin{align} |e^{\ee(1+\mu-y)_+|\xi|}\ww(y)I_{1}| &\lesssim \int_0^{3y/4} e^{\ee(1+\mu-z)_+|\xi|} \frac{1}{\sqrt{\nu (t-s)}} e^{-\frac{(y-z)^2}{16\nu (t-s)}} e^{-\nu\xi^2(t-s)} \ww(z) |N_\xi(s,z)| \,dz \notag \\&\les \norm{N_\xi(s)}_{\ZL_{\mu,\nu}} \int_0^{\infty} \frac{1}{\sqrt{\nu (t-s)}} e^{-\frac{(y-z)^2}{16\nu (t-s)}} \,dz \notag \\&\les \norm{N_\xi(s)}_{\ZL_{\mu,\nu}} \, . \llabel{sldkfjg;ls dslksfgj s;lkdfgj ;sldkfjg ;sldkfgj s;ldfkgj s;ldfkgj gfisasdoifaghskcx,.bvnliahglidhgsd gs sdfgh sdfg sldfg sldfgkj slfgj sl;dfgkj sl;dfgjk sldgfkj jiurwe alskjfa;sd fasdfEQthisijhhssdinthegdsfg43} \end{align} Summing in $\xi$ yields the bound \begin{equation} \lVert I_{1}\rVert_{X_\mu} \les \lVert N(s)\rVert_{X_\mu} \, . \label{EQthisijhhssdinthegdsfg45} \end{equation} Next, we consider the term $I_{2}$ on the right side of \eqref{EQthisijhhssdinthegdsfg:dec:non}. Since $\norm{\psi'}_{L^\infty} \les 1$, we directly obtain \begin{align} |I_2| \les \int_{y/2}^{3y/4} \frac{1}{\sqrt{\nu (t-s)}} e^{-\frac{(y-z)^2}{4\nu (t-s)}} e^{-\nu\xi^2(t-s)} |N_\xi(s,z)| \,dz \, , \llabel{sldkfjg;ls dslksfgj s;lkdfgj ;sldkfjg ;sldkfgj s;ldfkgj s;ldfkgj gfisasdoifaghskcx,.bvnliahglidhgsd gs sdfgh sdfg sldfg sldfgkj slfgj sl;dfgkj sl;dfgjk sldgfkj jiurwe alskjfa;sd fasdfEQthisijhhssdinthegdsfg12b} \end{align} which shows that $I_2$ obeys the same estimate as $I_1$ (cf.~\eqref{EQthisijhhssdinthegdsfg12} above). Since the regions of integration also match, the same proof as for \eqref{EQthisijhhssdinthegdsfg45} gives \begin{equation} \lVert I_{2}\rVert_{X_\mu} \le \lVert N(s)\rVert_{X_\mu} \, . \llabel{sldkfjg;ls dslksfgj s;lkdfgj ;sldkfjg ;sldkfgj s;ldfkgj s;ldfkgj gfisasdoifaghskcx,.bvnliahglidhgsd gs sdfgh sdfg sldfg sldfgkj slfgj sl;dfgkj sl;dfgjk sldgfkj jiurwe alskjfa;sd fasdfEQthisijhhssdinthegdsfg45b} \end{equation} \par The term $I_3$ in \eqref{EQthisijhhssdinthegdsfg:dec:non}, which we recall equals \begin{equation} I_{3} = y\int_{y/2}^{1+\mu} \left(1-\psi\left(\frac{z}{y}\right)\right)H_\xi(t-s, y, z)\partial_{z}N_\xi(s, z) \,dz \, , \label{EQthisijhhssdinthegdsfg322} \end{equation} is treated slightly differently. Since $z \geq y/2$ we may trade a power of $y$ for a power of $z$, and we also have \colW $w(y) \les w(z)$ for $z \geq y/2$ \colb by Remark~\ref{R04}(b), where the implicit constant is independent of $\nu$. Therefore, \begin{align} \lVert I_{3}\rVert_{X_\mu}&= \sum_\xi \sup_{y\in\Omega_\mu} \ww(y)e^{\ee(1+\mu-y)_+|\xi|} |I_{3}| \notag \\& \lesssim \sum_\xi \sup_{y\in\Omega_\mu} \int_{y/2}^{1+\mu} e^{\ee(1+\mu-y)_+|\xi|} H_\xi(t-s, y, z)\ww(z)|z\partial_{z}N_\xi(s, z)| \,dz \, . \llabel{sldkfjg;ls dslksfgj s;lkdfgj ;sldkfjg ;sldkfgj s;ldfkgj s;ldfkgj gfisasdoifaghskcx,.bvnliahglidhgsd gs sdfgh sdfg sldfg sldfgkj slfgj sl;dfgkj sl;dfgjk sldgfkj jiurwe alskjfa;sd fasdfEQthisijhhssdinthegdsfg338} \end{align} Now, we use \begin{align} e^{\ee(1+\mu-y)_+|\xi|} \le e^{\ee(1+\mu-z)_+|\xi|}e^{\ee(z-y)_+|\xi|} \le e^{\ee(1+\mu-z)_+|\xi|}e^{\fractext{\ee(y-z)^2}{2\nu (t-s)}} e^{\ee\nu\xi^2(t-s)/2} \, , \label{EQthisijhhssdinthegdsfg:y:z} \end{align} which follows from \begin{align} e^{2 a |\xi|} \leq e^{\fractext{a^2}{c}} e^{c \xi^2} \label{EQthisijhhssdinthegdsfg:y:z:new} \end{align} with suitable $a, c >0$. Choosing $\ee$ sufficiently small, we obtain \begin{align} \lVert I_{3}\rVert_{X_\mu} &\lesssim \sum_\xi \sup_{y\in\Omega_\mu} \int_{y/2}^{1+\mu} \frac{1}{\sqrt{\nu (t-s)}}e^{-\frac{(y-z)^2}{8\nu (t-s)}} e^{-\nu\xi^2(t-s)}e^{\ee(1+\mu-z)_+|\xi|} \ww(z)|z\partial_{z}N_\xi(s, z)| \,dz \notag \\&\lesssim \lVert z\partial_{z}N(s)\rVert_{X_{\mu}} \int_{0}^{\infty} \frac{1}{\sqrt{\nu (t-s)}} e^{-\frac{(y-z)^2}{8\nu (t-s)}} \,dz \, , \llabel{sldkfjg;ls dslksfgj s;lkdfgj ;sldkfjg ;sldkfgj s;ldfkgj s;ldfkgj gfisasdoifaghskcx,.bvnliahglidhgsd gs sdfgh sdfg sldfg sldfgkj slfgj sl;dfgkj sl;dfgjk sldgfkj jiurwe alskjfa;sd fasdfEQthisijhhssdinthegdsfg323} \end{align} whence \begin{equation} \lVert I_{3}\rVert_{X_\mu} \lesssim \lVert z\partial_{z}N(s)\rVert_{X_{\mu}} \, . \label{EQthisijhhssdinthegdsfg77} \end{equation} \par It remains to estimate the term $I_4$ in \eqref{EQthisijhhssdinthegdsfg:dec:non}, which we recall equals \begin{align} I_4 &= y \int_{1+\mu}^{\infty} H_\xi(t-s, y, z) \partial_z N_\xi(s, z) \,dz \,. \notag \end{align} Using Remark~\ref{R04}(a) and~(c), the bound \eqref{EQthisijhhssdinthegdsfg:y:z}, and choosing $\ee$ sufficiently small, we obtain \begin{align} &e^{\ee(1+\mu-y)_+|\xi|} w(y) |I_{4}| \notag\\ &\les \int_{1+\mu}^{1+\tilde\mu} \frac{1}{\sqrt{\nu(t-s)}} e^{\ee(z-y)_+|\xi|} e^{-\frac{(y-z)^2}{4\nu (t-s)}} e^{-\frac{1}{2}\nu\xi^2(t-s)} e^{\ee(1+\mu-z)_+|\xi|} w(z) |z \partial_z N_\xi(s,z)| \,dz \notag\\ &\quad + \int_{1+\tilde \mu}^{\infty} \frac{1}{\sqrt{\nu(t-s)}} e^{\ee(z-y)_+|\xi|} e^{-\frac{(y-z)^2}{4\nu (t-s)}} e^{-\nu\xi^2(t-s)} |\partial_z N_\xi(s,z)| \,dz \notag\\ &\les \norm{z \partial_z N_\xi(s)}_{\ZL_{\tilde \mu,\nu}} \int_{1+\mu}^{1+\tilde\mu} \frac{1}{\sqrt{\nu(t-s)}} e^{-\frac{(y-z)^2}{8\nu (t-s)}} \,dz \notag\\&\indeq + \int_{1+\tilde\mu}^{\infty} \frac{1}{\sqrt{\nu(t-s)}} e^{-\frac{(y-z)^2}{16\nu (t-s)}} e^{-\frac{(\tilde\mu-\mu)^2}{16\nu (t-s)}} |\partial_z N_\xi(s,z)| \,dz \notag\\ &\les \norm{z \partial_z N_\xi(s)}_{\ZL_{\tilde\mu,\nu}} + \frac{e^{-\frac{(\tilde\mu-\mu)^2}{16\nu (t-s)}} }{(\nu(t-s))^{1/4}} \left( \int_{1+\tilde\mu}^{\infty} \frac{1}{\sqrt{\nu(t-s)}} e^{-\frac{(y-z)^2}{8\nu (t-s)}} \,dz\right)^{1/2} \norm{\partial_z N_\xi(s,z)}_{L^2( z\geq 1+\tilde\mu)} \notag\\ &\les \norm{z \partial_z N_\xi(s)}_{\ZL_{\tilde\mu,\nu}} + \frac{1}{(\tilde\mu-\mu)^{1/2}} \norm{\partial_z N_\xi(s,z)}_{L^2( z\geq 1+\tilde\mu)} \, . \notag \end{align} Taking a supremum over $y \in \Omega_\mu$, summing over $\xi$, and recalling the definition of $\tilde\mu$ in \eqref{EQthisijhhssdinthegdsfg:mu:cond}, we deduce \begin{align} \norm{I_4}_{X_\mu} \les \norm{z \partial_z N_\xi(s)}_{X_{\tilde\mu}} + \frac{1}{\sqrt{\mu_0-\mu-\gamma s}} \sum_\xi \norm{\partial_z N_\xi(s,z)}_{L^2( z\geq 1+\tilde\mu)} \, . \notag \end{align} This concludes the proof of \eqref{EQthisijhhssdinthegdsfg38} with $(i,j) = (0,1)$. \par Next, we note that the estimate \eqref{EQthisijhhssdinthegdsfg38} for $(i,j)=(1,0)$ follows from the bound \eqref{EQthisijhhssdinthegdsfg38} with $(i,j)=(0,0)$, by applying the bound to $\partial_x N$ instead of $N$. Therefore, it only remains to establish \eqref{EQthisijhhssdinthegdsfg38} for $(i,j)=(0,0)$. With $\psi$ as above, we decompose the convolution integral into three integrals as \begin{align} & \int_0^\infty H_\xi(t-s, y, z)N_\xi(s, z) \,dz \notag \\&\indeq = \int_0^{3y/4} \psi\left(\frac{z}{y}\right)H_\xi(t-s, y, z) N_\xi(s, z) \,dz + \int_{y/2}^{1+\mu} \left(1-\psi\left(\frac{z}{y}\right)\right)H_\xi(t-s, y, z) N_\xi(s, z) \,dz \notag \\&\indeq\indeq + \int_{1+\mu}^\infty H_\xi(t-s, y, z) N_\xi(s, z) \,dz \notag \\&\indeq =J_1+J_2+J_3 \, . \label{EQthisijhhssdinthegdsfg215} \end{align} Upon inspection, $J_1$ may be bounded in exactly the same way as the term $I_1$ earlier, which gives the bound \begin{align} \norm{J_1}_{X_{\mu}} \les \norm{N(s)}_{X_{\mu}} \, . \notag \end{align} On the other hand, $J_2$ is estimated exactly as the term $I_3$ above, and we obtain \begin{align} \norm{J_2}_{X_{\mu}} \les \norm{N(s)}_{X_{\mu}} \, . \notag \end{align} Lastly, $J_3$ is bounded just as $I_4$, by splitting the integral on $[1+\mu,\infty)$ into an integral on $[1+\mu,1+\tilde\mu]$ and one on $[1+\tilde\mu,\infty)$. This results in the bound \begin{align} \norm{J_3}_{X_\mu} \les \norm{N_\xi(s)}_{X_{\tilde\mu}} + \frac{1}{\sqrt{\mu_0-\mu-\gamma s}} \sum_\xi \norm{N_\xi(s,z)}_{L^2( z\geq 1+\tilde\mu)} \, , \notag \end{align} concluding the proof of the lemma. \end{proof} \par \cole \begin{lemma} \label{rem:x} Let $\mu < \tilde \mu$ obey \eqref{EQthisijhhssdinthegdsfg:mu:cond}. For $(i,j)=(0,0),(1,0),(0,1)$, we have \begin{align} & \left\lVert \partial_{x}^{i}(y\partial_{y})^{j} \int_0^\infty R(t-s, y, z)N(s, z) \,dz\right\rVert_{X_\mu} \notag \\&\indeq \lesssim \lVert \partial_{x}^{i}(y\partial_{y})^{j}N(s)\rVert_{X_{\tilde \mu}} + \lVert N(s)\rVert_{X_{\tilde \mu}} + \frac{1}{(\mu_0-\mu-\gamma s)^{1/2}} \sum_\xi \lVert \partial_{x}^{i}\partial_{y}^{j} N_\xi\rVert_{L^2(y\ge1+\tilde \mu)} \, . \label{EQthisijhhssdinthegdsfg38aa} \end{align} \end{lemma} \colb \par \begin{proof}[Proof of Lemma~\ref{rem:x}] In order to establish \eqref{EQthisijhhssdinthegdsfg38aa}, it suffices to verify the assumptions in Remark~\ref{R06} for the kernel $R_{\xi}(t,y,z)$. We recall that $\partial_y R_\xi(t,y,z) = - \partial_z R_\xi(t,y,z)$, which is necessary to change $y$ derivatives to $z$ derivatives. First fix $y\in[0,1+\mu]$ and $z\in[0,3y/4]$. Then, since \colW $w(z) \gtrsim \sqrt{\nu}$\colb, from \eqref{EQthisijhhssdinthegdsfg33} we have \begin{align} \frac{\ww(y)}{\ww(z)} \bigl( |R_{\xi}(t,y,z)| + |y\partial_{y}R_{\xi}(t,y,z)| \bigr) & \les \frac{\ww(y)}{\ww(z)} \left( be^{-\frac{\theta_0}{2}b(y+z)} + \frac{1}{\sqrt{\nu t}}e^{-\frac{\theta_0}{2} \frac{(y+z)^2}{\nu t}}e^{-\frac{\nu\xi^2t}{8}} \right) \notag \\& \les \frac{\ww(y)}{\sqrtnu} \left( be^{-\frac{\theta_0}{2}b(y+z)} + \frac{1}{\sqrt{\nu t}}e^{-\frac{\theta_0}{2} \frac{y^2+z^2}{\nu t}} \right) \, . \label{EQthisijhhssdinthegdsfg41} \end{align} Next, we use $b\geq \sqrt{\nu}^{-1}$ and thus by Remark~\ref{R04}(e) we have \begin{align} \frac{\ww(y)}{\sqrtnu} e^{-\frac{\theta_0}{4} b y} &\les e^{\frac{y}{C\sqrtnu}} e^{-\frac{\theta_0}{4} \frac{y}{\sqrt{\nu}} } \les 1 \, , \label{EQthisijhhssdinthegdsfg185} \end{align} provided that $C$ is sufficiently large (in terms of $\theta_0$). Similarly to \eqref{EQthisijhhssdinthegdsfg44}, using Remark~\ref{R04}(e) and the fact that $t\leq 1$ we have \begin{align} \frac{w(y)}{\sqrt{\nu}} e^{-\frac{\theta_0}{2} \frac{y^2}{\nu t}} \les \frac{w(y)}{\sqrt{\nu}} e^{-\frac{\theta_0}{2} \frac{y}{\sqrt{\nu}}} \les 1 \, . \llabel{sldkfjg;ls dslksfgj s;lkdfgj ;sldkfjg ;sldkfgj s;ldfkgj s;ldfkgj gfisasdoifaghskcx,.bvnliahglidhgsd gs sdfgh sdfg sldfg sldfgkj slfgj sl;dfgkj sl;dfgjk sldgfkj jiurwe alskjfa;sd fasdfEQthisijhhssdinthegdsfg158a} \end{align} Therefore, the right side of \eqref{EQthisijhhssdinthegdsfg41} is bounded pointwise in $y$ by \begin{align} b e^{-\frac{\theta_0}{4} bz} + \frac{1}{\sqrt{\nu t}} e^{-\frac{\theta_0}{2} \frac{z^2}{\nu t}} \, . \llabel{sldkfjg;ls dslksfgj s;lkdfgj ;sldkfjg ;sldkfgj s;ldfkgj s;ldfkgj gfisasdoifaghskcx,.bvnliahglidhgsd gs sdfgh sdfg sldfg sldfgkj slfgj sl;dfgkj sl;dfgjk sldgfkj jiurwe alskjfa;sd fasdfEQthisijhhssdinthegdsfg58} \end{align} Using that $\nnorm{be^{-\frac{\theta_0 b z}{4} }}_{L^1_z} \les 1$ and $\nnorm{\frac{1}{\sqrt{\nu t}} e^{-\frac{\theta_0 z^2}{2\nu t}}}_{L_{z}^{1}}\les 1$, the condition \eqref{EQthisijhhssdinthegdsfg04} for $R$ follows. \par In order to verify the condition~\eqref{EQthisijhhssdinthegdsfg08}, we use that $b\geq |\xi|$, and provided $\epsilon_0$ is sufficiently small in terms of $\theta_0$, we obtain from \eqref{EQthisijhhssdinthegdsfg:y:z:new} that \begin{align} e^{\ee (z-y)_+ |\xi|} |R_\xi(y,z,t)| &\les e^{\ee (z-y)_+ |\xi|} \biggl( b e^{-\theta_0 b z} + \frac{1}{\sqrt{\nu t}} e^{-\frac{\theta_0}{2} \frac{y^2+z^2}{\nu t}} e^{-\frac{\nu \xi^2 t}{8}} \biggr) \notag \\& \les be^{-\frac12\theta_0 b z} + \biggl( e^{\epsilon_0 z |\xi|} e^{-\frac{\theta_0}{4} \frac{z^2}{\nu t}} e^{-\frac{\nu \xi^2 t}{8}} \biggr) \frac{1}{\sqrt{\nu t}} e^{-\frac{\theta_0}{4} \frac{z^2}{\nu t}} \notag \\& \les be^{-\frac12\theta_0 b z} + \frac{1}{\sqrt{\nu t}} e^{-\frac{\theta_0}{4} \frac{z^2}{\nu t}} \, , \llabel{sldkfjg;ls dslksfgj s;lkdfgj ;sldkfjg ;sldkfgj s;ldfkgj s;ldfkgj gfisasdoifaghskcx,.bvnliahglidhgsd gs sdfgh sdfg sldfg sldfgkj slfgj sl;dfgkj sl;dfgjk sldgfkj jiurwe alskjfa;sd fasdfEQthisijhhssdinthegdsfg49} \end{align} and \eqref{EQthisijhhssdinthegdsfg08}, for this kernel, follows by integrating in $z$. Finally, we check \eqref{EQthisijhhssdinthegdsfg39}. For this, let $y\in[0,1+\mu]$ and $z\geq 1+\tilde\mu$. Then \begin{align} e^{\ee(z-y)_+|\xi|} | R_{\xi}(y,z,t) | &\les e^{\ee(z-y)_+|\xi|} \biggl( b e^{-\theta_0 b (y+z)} + \frac{1}{\sqrt{\nu t}} e^{-\frac{\theta_0}{2} \frac{y^2+z^2}{\nu t}} e^{-\frac{\nu \xi^2 t}{8}} \biggr) \notag \\& \les b e^{-\frac12\theta_0 b z} + \frac{1}{\sqrt{\nu t}}e^{-\frac{\theta_0}{4} \frac{ z^2}{\nu t}} \les b e^{-\frac12\theta_0 b z} + \frac{1}{(\nu t)^{1/4}}\left( \frac{z^{2}}{\nu t} \right)^{1/4}  e^{-\frac{\theta_0}{4} \frac{z^2}{\nu t}} \notag \\& \les b^{1/2}e^{-\frac14\theta_0 b z} + \frac{1}{(\nu t)^{1/4}}e^{-\frac{\theta_0}{8} \frac{ z^2}{\nu t}} \, , \label{EQthisijhhssdinthegdsfg50} \end{align} where we used $z\geq 1+\tilde \mu \geq 1$ in the third inequality. Finally, note that the $L^{2}$ norm of the far right hand side of \eqref{EQthisijhhssdinthegdsfg50} over $[1+\tilde\mu,\infty)$ is less than a constant. \end{proof} \par Next, we consider the trace kernel. \par \cole \begin{lemma} \label{tra:x} Let $\mu \in (0,\mu_0-\gamma s)$ be arbitrary. For $(i,j)=(0,0),(1,0),(0,1)$, we have the inequality \begin{align} \left\lVert \partial_{x}^{i}(y\partial_{y})^{j}G(t-s, y,0)\partial_{z}\Delta^{-1}N (s, z)|_{z=0} \right\rVert_{X_\mu} &\lesssim \frac{1}{\sqrt{t-s}}(\lVert \partial_{x}^{i}N(s)\rVert_{\YY_{\mu}} + \lVert \partial_x^i N(s)\rVert_{S_\mu}) + \norm{\partial_x^i N(s)}_{X_\mu}\, . \label{EQthisijhhssdinthegdsfg171} \end{align} \end{lemma} \colb \par \begin{proof}[Proof of Lemma~\ref{tra:x}] For $\xi \in \ZZ$, the kernel $G_\xi(t-s,y,0)$ is the sum of two trace operators \begin{equation} T_1(t-s, y) = \tilde H_\xi(t-s,y,0) = \frac{2}{\sqrt{\nu (t-s)}}e^{-\frac{y^2}{4\nu (t-s)}} e^{-\nu\xi^2(t-s)} \label{EQthisijhhssdinthegdsfg167} \end{equation} and \begin{equation} T_2(t-s, y) = R_{\xi}(t-s,y,0) \, . \llabel{sldkfjg;ls dslksfgj s;lkdfgj ;sldkfjg ;sldkfgj s;ldfkgj s;ldfkgj gfisasdoifaghskcx,.bvnliahglidhgsd gs sdfgh sdfg sldfg sldfgkj slfgj sl;dfgkj sl;dfgjk sldgfkj jiurwe alskjfa;sd fasdfEQthisijhhssdinthegdsfg168} \end{equation} Recall that \begin{equation} |y\partial_{y}T_1 (t-s,y)|=\frac{1}{\sqrt{\nu (t-s)}}e^{-\frac{y^2}{4\nu (t-s)}} e^{-\nu\xi^2(t-s)}\frac{y^2}{2\nu(t-s)}\lesssim \frac{1}{\sqrt{\nu (t-s)}}e^{-\frac{y^2}{8\nu (t-s)}} e^{-\nu\xi^2(t-s)} \label{EQthisijhhssdinthegdsfg169} \end{equation} and \begin{align} |T_2(t-s,y) | + |y \partial_y T_2(t-s,y)| \les b e^{-\frac 12 \theta_0 b y } + \frac{1}{\sqrt{\nu (t-s)}} e^{-\frac{\theta_0}{2} \frac{y^2}{\nu (t-s)}} e^{-\frac{\nu\xi^2(t-s)}{8}} \label{EQthisijhhssdinthegdsfg169b} \, . \end{align} \par We first prove \eqref{EQthisijhhssdinthegdsfg171} in the case $i=j=0$. Similarly to the equation~\eqref{EQthisijhhssdinthegdsfg123} in Lemma~\ref{bio:sav} below, we have the representation formula \begin{align} \left( \partial_{z}\Delta^{-1}N_\xi(s, z)\right) |_{z=0} &= - \int_0^\infty e^{-|\xi| z}N_\xi(s, z) \,d z \notag \\&= - \int_0^{1+\mu} e^{-|\xi| z}N_\xi(s, z) \,d z - \int_{1+\mu}^\infty e^{-|\xi| z}N_\xi(s, z) \,d z = I_1 + I_2 \, . \label{EQthisijhhssdinthegdsfg229} \end{align} First we treat the $T_1$ contribution. Using \eqref{EQthisijhhssdinthegdsfg44} and choosing $\ee$ sufficiently small, we have \begin{align} &|e^{\ee(1+\mu-y)_+|\xi|}\ww(y)T_1(t-s, y)I_1| \notag \\&\indeq\lesssim \frac{1}{\sqrt{t-s}} e^{\ee(1+\mu-y)_+|\xi|} e^{-\frac{y^2}{8\nu (t-s)}} e^{-\nu\xi^2(t-s)}\int_0^{1+\mu} e^{-|\xi| z} |N_\xi(s, z)| \,dz \notag \\&\indeq\lesssim \frac{1}{\sqrt{t-s}} \int_0^{1+\mu} e^{-|\xi| z}e^{\ee(z-y)_+|\xi|}e^{\ee(1+\mu-z)_+|\xi|} |N_\xi(s, z)| \,dz \notag \\&\indeq\lesssim \frac{1}{\sqrt{t-s}} \int_0^{1+\mu} e^{\ee(1+\mu-z)_+|\xi|} |N_\xi(s, z)| \,dz \, , \llabel{sldkfjg;ls dslksfgj s;lkdfgj ;sldkfjg ;sldkfgj s;ldfkgj s;ldfkgj gfisasdoifaghskcx,.bvnliahglidhgsd gs sdfgh sdfg sldfg sldfgkj slfgj sl;dfgkj sl;dfgjk sldgfkj jiurwe alskjfa;sd fasdfEQthisijhhssdinthegdsfg174} \end{align} leading to \begin{align} \left\lVert T_1(t-s, y)I_1\right\rVert_{X_\mu} \lesssim \frac{1}{\sqrt{t-s}}\lVert N(s)\rVert_{\YY_{\mu}} \label{EQthisijhhssdinthegdsfg175} \end{align} upon summing in $\xi$. For the integral $I_2$ in \eqref{EQthisijhhssdinthegdsfg229}, we similarly use \eqref{EQthisijhhssdinthegdsfg44} and obtain the inequality \begin{align} |e^{\ee(1+\mu-y)_+|\xi|}\ww(y)T_1(t-s, y)I_2| &\lesssim \frac{1}{\sqrt{t-s}} \int^\infty_{1+\mu} | N_\xi(s, z) | \,dz \notag\\ &\lesssim \frac{1}{\sqrt{t-s}} \norm{z N_\xi(s, z)}_{L^2(z\geq 1+\mu)} \, , \llabel{sldkfjg;ls dslksfgj s;lkdfgj ;sldkfjg ;sldkfgj s;ldfkgj s;ldfkgj gfisasdoifaghskcx,.bvnliahglidhgsd gs sdfgh sdfg sldfg sldfgkj slfgj sl;dfgkj sl;dfgjk sldgfkj jiurwe alskjfa;sd fasdfEQthisijhhssdinthegdsfg176} \end{align} implying \begin{align} \left\lVert T_1(t-s, y)I_2\right\rVert_{X_\mu} \lesssim \frac{1}{\sqrt{t-s}}\lVert N(s)\rVert_{S_\mu} \, . \label{EQthisijhhssdinthegdsfg177} \end{align} For the $T_2$ contribution, appealing to \eqref{EQthisijhhssdinthegdsfg185} and using $b \sqrt{\nu} = \colW 1 + |\xi| \sqrt{\nu} \les 1 + |\xi| w(z)\colb $ for any $z\geq0$, we have \begin{align} &|e^{\ee(1+\mu-y)_+|\xi|}\ww(y) b e^{-\frac 12 \theta_0 b y} I_1| \notag \\&\indeq\lesssim e^{\ee(1+\mu-y)_+|\xi|} b \sqrt{\nu} \int_0^{1+\mu} e^{-|\xi| z} |N_\xi(s, z)| \,dz \notag \\&\indeq\lesssim \int_0^{1+\mu} e^{-|\xi| z} e^{\ee(z-y)_+|\xi|}e^{\ee(1+\mu-z)_+|\xi|} |N_\xi(s, z)| \,dz \notag \\&\indeq\indeq + \int_0^{1+\mu} |\xi| e^{-|\xi| z} e^{\ee(z-y)_+|\xi|}e^{\ee(1+\mu-z)_+|\xi|} w(z) |N_\xi(s, z)| \,dz \notag \\&\indeq\lesssim \int_0^{1+\mu} e^{\ee(1+\mu-z)_+|\xi|} |N_\xi(s, z)| \,dz \notag \\&\indeq\indeq + \int_0^{1+\mu} |\xi| e^{-\frac 12 |\xi| z} e^{\ee(1+\mu-z)_+|\xi|} w(z) |N_\xi(s, z)| \,dz \, . \llabel{sldkfjg;ls dslksfgj s;lkdfgj ;sldkfjg ;sldkfgj s;ldfkgj s;ldfkgj gfisasdoifaghskcx,.bvnliahglidhgsd gs sdfgh sdfg sldfg sldfgkj slfgj sl;dfgkj sl;dfgjk sldgfkj jiurwe alskjfa;sd fasdfEQthisijhhssdinthegdsfg186} \end{align} The first of the above terms is estimated using the $\YY_\mu$~norm, while the second one is bounded using the $X_\mu$~norm. Here we use that $\norm{ |\xi| e^{-\frac 12 |\xi| z}}_{L^1_z} \les 1$. The above estimate, combined with the fact that the second term in the upper bound for $T_2$ (cf.~\eqref{EQthisijhhssdinthegdsfg169b}) is estimated just as $T_1$, leads to the bound \begin{align} \left\lVert T_2(t-s, y)I_1\right\rVert_{X_\mu} \lesssim \frac{1}{\sqrt{t-s}}\lVert N(s)\rVert_{\YY_{\mu}} + \norm{N(s)}_{X_\mu} \, . \label{EQthisijhhssdinthegdsfg187} \end{align} For the contribution from $T_2$ to the second integral in \eqref{EQthisijhhssdinthegdsfg229} we use \eqref{EQthisijhhssdinthegdsfg185} and the bound $\sqrt{\nu} b \les 1 + |\xi|$, to obtain \begin{align} |e^{\ee(1+\mu-y)_+|\xi|}\ww(y) b e^{-\frac 12 \theta_0 b y} I_2| &\les e^{\ee(1+\mu-y)_+|\xi|} \sqrt{\nu} b e^{-|\xi|(1+\mu)} \int^\infty_{1+\mu} |N_\xi(s, z)| \,dz \notag\\ &\lesssim \lVert z N_\xi(s)\rVert_{L^2(z\ge1+\mu)} \, . \llabel{sldkfjg;ls dslksfgj s;lkdfgj ;sldkfjg ;sldkfgj s;ldfkgj s;ldfkgj gfisasdoifaghskcx,.bvnliahglidhgsd gs sdfgh sdfg sldfg sldfgkj slfgj sl;dfgkj sl;dfgjk sldgfkj jiurwe alskjfa;sd fasdfEQthisijhhssdinthegdsfg188} \end{align} Again, since the second term in the upper bound for $T_2$ (cf.~\eqref{EQthisijhhssdinthegdsfg169b}) is estimated just as $T_1$ we obtain \begin{align} \left\lVert T_{2}(t-s, y)I_2\right\rVert_{X_\mu} \lesssim \frac{1}{\sqrt{t-s}} \lVert y N(s)\rVert_{L^2(y\ge1+\mu)} \, . \label{EQthisijhhssdinthegdsfg189} \end{align} Combining \eqref{EQthisijhhssdinthegdsfg175}, \eqref{EQthisijhhssdinthegdsfg177}, \eqref{EQthisijhhssdinthegdsfg187}, and \eqref{EQthisijhhssdinthegdsfg189} concludes the proof of \eqref{EQthisijhhssdinthegdsfg171} when $(i,j)=(0,0)$. For $(i,j)=(0,1)$, we use the fact that the conormal derivative of the kernel obeys similar estimates as the kernel itself, which holds in view of \eqref{EQthisijhhssdinthegdsfg169} and \eqref{EQthisijhhssdinthegdsfg169b}. Lastly, for $(i,j)=(1,0)$, the $\partial_x$ derivative simply acts on the $N_\xi$ term in \eqref{EQthisijhhssdinthegdsfg229}, and the above proof applies. \end{proof} \par Finally, we estimate the first term in the mild representation of the solution~\eqref{EQthisijhhssdinthegdsfg:kernel:est}. \par \cole \begin{lemma} \label{L02} Let $\mu \in (0, \mu_0 - \gamma s)$ be arbitrary. For $i+j\leq 2$, the initial datum term in \eqref{EQthisijhhssdinthegdsfg:kernel:est} satisfies \begin{align} & \sum_{i+j\leq 2} \norm{ \partial_x^i (y \partial_y)^j \int_0^\infty G(t, y,z)\omega_{0}(z) \,dz}_{X_\mu} \notag \\&\indeq \les \sum_{i+j\leq 2}\lVert \partial_x^i (y\partial_{y})^j \omega_0\rVert_{X_{\mu}} + \sum_{i+j\leq 2} \sum_{\xi} \lVert \xi^i \partial_{y}^j \omega_{0,\xi}\rVert_{L^\infty(y\ge 1+\mu)} \, . \label{EQthisijhhssdinthegdsfg37} \end{align} \end{lemma} \colb \par {\begin{proof}[Proof of Lemma~\ref{L02}] Let $i+j\leq2$. We recall from \eqref{EQthisijhhssdinthegdsfg30}, \eqref{EQthisijhhssdinthegdsfg31}, and \eqref{EQthisijhhssdinthegdsfg36} that \begin{align} G_\xi(t,y,z) = H_\xi(t,y,z) + H_\xi(t,y,-z) + R_\xi(t,y,z)\, . \label{EQthisijhhssdinthegdsfg42} \end{align} Accordingly, we divide \begin{align} & \left( \partial_x^i (y \partial_y)^j \int_0^\infty G(t, y,z)\omega_{0}(z) \,dz \right)_\xi \notag\\ &\indeq = \int_0^\infty (\ii \xi)^i (y \partial_y)^j H_\xi(t, y,z)\omega_{0,\xi}(z) \,dz + \int_0^\infty (\ii \xi)^i (y \partial_y)^j H_\xi(t, y,-z)\omega_{0,\xi}(z) \,dz\notag\\ &\indeq\indeq+ \int_0^\infty (\ii \xi)^i (y \partial_y)^j R_\xi(t, y,z)\omega_{0,\xi}(z) \,dz \notag\\ &\indeq = J_1 + J_2 + J_3\, . \label{EQthisijhhssdinthegdsfg46} \end{align} We first treat the term $J_1$. Using that \begin{align} (y \partial_y)^j = y^j \partial_y^j + {\bf 1}_{\{ j=2 \}} y\partial_y \notag \end{align} and $y \les 1$, we have, similarly to \eqref{EQthisijhhssdinthegdsfg:dec:non}, \begin{align} |J_1| &\les \int_0^{3y/4} | (y \partial_y)^j H_\xi(t, y, z)| |\xi|^i |\omega_{0,\xi}(z)| \,dz \notag\\ &\indeq + \int_{y/2}^{3y/4} \left( \left|\psi'\left(\frac{z}{y}\right)\right| + \left|\psi''\left(\frac{z}{y}\right)\right| \right) |H_\xi(t, y, z)| |\xi|^i |\omega_{0,\xi}(z)| \,dz \notag \\ &\indeq + \int_{y/2}^{3y/4} \left|\psi'\left(\frac{z}{y}\right)\right| |H_\xi(t, y, z)| |\partial_z \omega_{0,\xi}(z)| \,dz \notag \\ &\indeq + \int_{y/2}^{1+\mu} | H_\xi(t, y, z)| |\partial_z^2 \omega_{0,\xi}(z)| \,dz + \int_{1+\mu}^{\infty} | H_\xi(t, y, z)| |\partial_z^2 \omega_{0,\xi}(z)| \,dz \notag\\ &\indeq + \int_{y/2}^{1+\mu} | H_\xi(t, y, z)| |\partial_z \omega_{0,\xi}(z)| \,dz + \int_{1+\mu}^{\infty} | H_\xi(t, y, z)| |\partial_z \omega_{0,\xi}(z)| \,dz \notag\\ &\indeq +{\bf 1}_{\{j\leq 1\}} \int_{y/2}^{1+\mu} |H_\xi(t, y, z)| |\xi|^i |\partial_z^j \omega_{0,\xi}(z)| \,dz \notag\\ &\indeq +{\bf 1}_{\{j\leq 1\}} \int_{1+\mu}^{\infty} |H_\xi(t, y, z)| |\xi|^i |\partial_z^j \omega_{0,\xi}(z)| \,dz \notag\\ & =J_{11}+J_{12}+J_{13}+J_{14} + J_{15} + J_{16}+J_{17} + J_{18} + J_{19} \, . \label{EQthisijhhssdinthegdsfg47} \end{align} The terms $J_{11}$, $J_{12}$, and $J_{13}$ are bounded in the same way as the term $I_1$ in \eqref{EQthisijhhssdinthegdsfg:dec:non} (see \eqref{EQthisijhhssdinthegdsfg12}--\eqref{EQthisijhhssdinthegdsfg45}), leading to the first term in \eqref{EQthisijhhssdinthegdsfg37}. The terms $J_{14}$, $J_{16}$, and $J_{18}$ are estimated in the same way as the term $I_3$ in \eqref{EQthisijhhssdinthegdsfg:dec:non} (see \eqref{EQthisijhhssdinthegdsfg322}--\eqref{EQthisijhhssdinthegdsfg77}), and are bounded by the first term in \eqref{EQthisijhhssdinthegdsfg37}. It is only the Sobolev contributions $J_{15}$, $J_{17}$, and $J_{19}$ which need to be treated differently than the term $I_4$ in \eqref{EQthisijhhssdinthegdsfg:dec:non}, because here we do not wish to increase the value of $\mu$ to $\tilde \mu$. These Sobolev terms are treated in the same way. For instance, for $J_{15}$ we use \eqref{EQthisijhhssdinthegdsfg:y:z} and obtain \begin{align} e^{\ee(1+\mu-y)_+|\xi|} w(y) |J_{15}| &\les \int_{1+\mu}^{\infty} \frac{1}{\sqrt{\nu t}} e^{\ee(z-y)_+|\xi|} e^{-\frac{(y-z)^2}{4\nu t}} e^{-\frac{1}{2}\nu\xi^2 t} | \partial_z^2 \omega_{0,\xi}(z)| \,dz \notag\\ &\les \int_{1+\mu}^{\infty} \frac{1}{\sqrt{\nu t}} e^{-\frac{(y-z)^2}{8\nu t}} | \partial_z^2 \omega_{0,\xi}(z)| \,dz \les \norm{ \partial_z^2 \omega_{0,\xi}(z)}_{L^\infty(z \geq 1+\mu)} \, . \label{EQthisijhhssdinthegdsfg56} \end{align} Thus, the terms $J_{15}$, $J_{17}$, and $J_{19}$ contribute to the second term on the right side of \eqref{EQthisijhhssdinthegdsfg37}. \par The second kernel in \eqref{EQthisijhhssdinthegdsfg42} is treated the same as the first. Likewise, the third kernel in \eqref{EQthisijhhssdinthegdsfg42} is a function of $y+z$ and we may write the analog of the inequality \eqref{EQthisijhhssdinthegdsfg56} and the proof concludes similarly. \end{proof} \par We conclude this section with the proof of Lemma~\ref{lem:main:X}. \begin{proof}[Proof of Lemma~\ref{lem:main:X}] Given $\mu \in (0,\mu_0-\gamma t)$ and $s\in(0,t)$, recall that $\mu_1$ and $\mu_2$ are given by \eqref{EQthisijhhssdinthegdsfg02} and \eqref{EQthisijhhssdinthegdsfg02b}. By the analyticity recovery for the $X$~norm, cf.~Lemma~\ref{ana:rec:x} below, the bound for the first term on the left side of \eqref{EQthisijhhssdinthegdsfg701} is a direct consequence of the bound for the second term. Indeed, we have \begin{align} &\sum_{i+j= 2} \left\lVert \partial_x^i (y\partial_{y})^j\int_0^\infty G(t-s, y, z)N(s, z) \,dz\right\rVert_{X_\mu} \notag\\ &\qquad \les \frac{1}{\mu_0-\mu - \gamma s} \sum_{i+j= 1} \left\lVert \partial_x^i (y\partial_{y})^j\int_0^\infty G(t-s, y, z)N(s, z) \,dz\right\rVert_{X_{\mu_1}} \, . \notag \end{align} The bound for the second term on the left side of \eqref{EQthisijhhssdinthegdsfg701} follows from Lemma~\ref{high:y:x}, Remark~\ref{R06}, and Lemma~\ref{rem:x}, applied with $\mu = \mu_1$ and $\tilde \mu = \mu_2$. \par Concerning the trace kernel, the estimate for the first term on the left side of inequality \eqref{EQthisijhhssdinthegdsfg703} is a consequence of the bound for the second term in \eqref{EQthisijhhssdinthegdsfg703}, the analyticity recovery for the $X$~norm in Lemma~\ref{ana:rec:x}, and the increase in analyticity domain from $\mu$ to $\mu_1$. The bound for the second term on the left side of \eqref{EQthisijhhssdinthegdsfg703} is a consequence of Lemma~\ref{tra:x} with $\mu = \mu_1$. \par Lastly, the initial datum term is bounded as in Lemma~\ref{L02}, which concludes the proof of the lemma. \end{proof} \par \section{Estimates for the $\YY$ analytic norm} \label{secL} \par \cole \begin{lemma} \label{high:y:x:L} Let $\mu \in (0,\mu_0-\gamma s)$ be arbitrary. For $(i,j)=(0,0),(1,0),(0,1)$, we have \begin{align} \left\lVert \partial_x^i (y\partial_{y})^j \int_0^\infty H(t-s, y, z)N(s, z) \,dz\right\rVert_{\YY_\mu} & \lesssim \lVert \partial_x^i (y\partial_{y})^j N(s)\rVert_{\YY_{\mu}} + \lVert N(s)\rVert_{\YY_{\mu}} + \lVert \partial_x^i \partial_{y}^j N(s)\rVert_{S_{\mu}} \, . \label{EQthisijhhssdinthegdsfg1024} \end{align} \end{lemma} \colb \par \begin{remark}\label{R07} Similarly to Remark~\ref{R06}, we emphasize that in the proof of Lemma~\ref{high:y:x:L} we only use several properties of the heat kernel $H_\xi(t,y,z)$. Examining the proof below, one may verify that these properties are: The kernel should be either a function of $y+z$ or $y-z$, and it should obey the estimates \begin{align} \norm{\chi_{\{0\leq y \leq1+\mu\}} \chi_{\{0\leq z \leq 3y/4\}} \bigl( | H_\xi(t,y,z) | + | y \partial_y H_\xi(t,y,z)| \bigr)}_{L^1_y L^\infty_z} &\les 1 \label{EQthisijhhssdinthegdsfg:R07:1}\\ \norm{e^{\ee (z-y)_+ |\xi|} H_\xi(t,y,z)}_{L^1_y L^\infty_z} &\les 1 \label{EQthisijhhssdinthegdsfg:R07:2} \, . \end{align} It is direct to check that the kernel $\tilde H_\xi(t,y,z) - H_\xi(t,y,z) = H_\xi(t,y,-z) = \frac{1}{\sqrt{\nu t}}e^{-\frac{(y+z)^2}{4\nu t}} e^{-\nu\xi^2t}$ obeys these two properties. Therefore, the bounds stated in Lemma~\ref{high:y:x:L} hold with $H(t,y,z)$ replaced by the full kernel $\tilde H(t,y,z)$. \end{remark} \par \begin{proof}[Proof of Lemma~\ref{high:y:x:L}] Let $y\in \Omega_{\mu}$. For simplicity, we only work with $y\in{\mathbb R}$; an adjustment for the complex case is straight-forward and leads only to notational complications. \par We start with the proof of \eqref{EQthisijhhssdinthegdsfg1024} in the case $(i,j) = (0,1)$. Let $\psi$ be the cut-off function from the proof of Lemma~\ref{high:y:x}. The first conormal derivative is given as in \eqref{EQthisijhhssdinthegdsfg:dec:non} by \begin{align} &y \partial_y \int_0^\infty H_\xi(t-s, y, z)N_\xi(s, z) \,dz \notag s\\& = - y\int_0^{3y/4} \psi\left(\frac{z}{y}\right) \partial_z H_\xi(t-s, y, z)N_\xi(s, z) \,dz - \int_{y/2}^{3y/4} \psi'\left(\frac{z}{y}\right) H_\xi(t-s, y, z)N_\xi(s, z) \,dz \notag \\&\indeq + y \int_{y/2}^{1+\mu} \left(1-\psi\left(\frac{z}{y}\right)\right) H_\xi(t-s, y, z) \partial_z N_\xi(s, z) \,dz + y \int_{1+\mu}^{\infty} H_\xi(t-s, y, z) \partial_z N_\xi(s, z) \,dz \notag \\& =I_1+I_2+I_3+I_4 \, . \label{EQthisijhhssdinthegdsfg:dec:non:L} \end{align} Using the bounds \eqref{EQthisijhhssdinthegdsfg10}, \eqref{EQthisijhhssdinthegdsfg09}, \eqref{EQthisijhhssdinthegdsfg11}, and \eqref{EQthisijhhssdinthegdsfg06}, we obtain \begin{align} e^{\ee(1+\mu-y)_+|\xi|} |I_1| \les \int_0^{3y/4} \frac{1}{\sqrt{\nu (t-s)}} e^{-\frac{(y-z)^2}{8\nu (t-s)}} e^{-\nu\xi^2(t-s)} e^{\ee(1+\mu-z)_+|\xi|} |N_\xi(s,z)| \,dz \, . \label{EQthisijhhssdinthegdsfg12:L} \end{align} Integrating in $y$, changing the order of integration, and using \begin{align} \norm{\frac{1}{\sqrt{\nu (t-s)}} e^{-\frac{(y-z)^2}{8\nu (t-s)}}}_{L^\infty_z L^1_y} \les 1 \, , \label{EQthisijhhssdinthegdsfg:unit:mass} \end{align} we arrive at \begin{align} \norm{e^{\ee(1+\mu-y)_+|\xi|} I_{1}}_{\SL_\mu} &\lesssim \int_0^{1+\mu} \int_0^{3y/4} \frac{1}{\sqrt{\nu (t-s)}} e^{-\frac{(y-z)^2}{16\nu (t-s)}} e^{\ee(1+\mu-z)_+|\xi|} |N_\xi(s,z)| \,dz dy \notag \\&\les \norm{e^{\ee(1+\mu-z)_+|\xi|} N_\xi(s)}_{\SL_\mu} \, . \llabel{sldkfjg;ls dslksfgj s;lkdfgj ;sldkfjg ;sldkfgj s;ldfkgj s;ldfkgj gfisasdoifaghskcx,.bvnliahglidhgsd gs sdfgh sdfg sldfg sldfgkj slfgj sl;dfgkj sl;dfgjk sldgfkj jiurwe alskjfa;sd fasdfEQthisijhhssdinthegdsfg61} \end{align} Summing over $\xi$ yields the bound \begin{equation} \lVert I_{1}\rVert_{\YY_\mu} \le \lVert N(s)\rVert_{\YY_\mu} \, . \label{EQthisijhhssdinthegdsfg45:L} \end{equation} The term $I_{2}$ in \eqref{EQthisijhhssdinthegdsfg:dec:non:L} is treated in the same way, by using $\norm{\psi'}_{L^\infty} \les 1$, leading to the same upper bound as in \eqref{EQthisijhhssdinthegdsfg45:L}. For the term $I_3$ in \eqref{EQthisijhhssdinthegdsfg:dec:non:L}, we use \eqref{EQthisijhhssdinthegdsfg:y:z}, the fact that $\eps_0$ is small, and the bound \eqref{EQthisijhhssdinthegdsfg:unit:mass}, in order to conclude \begin{align} \lVert I_{3}\rVert_{\YY_\mu}&= \sum_\xi \norm{e^{\ee(1+\mu-y)_+|\xi|} I_{3}}_{\SL_\mu} \notag \\& \lesssim \sum_\xi \norm{ \int_{y/2}^{1+\mu} \frac{1}{\sqrt{\nu (t-s)}}e^{-\frac{(y-z)^2}{8\nu (t-s)}} e^{\ee(1+\mu-z)_+|\xi|} |z\partial_{z}N_\xi(s, z)| \,dz}_{\SL_\mu} \notag \\& \lesssim \sum_\xi \norm{e^{\ee(1+\mu-z)_+|\xi|} |z\partial_{z}N_\xi(s)|}_{\SL_\mu} = \norm{z\partial_{z}N_\xi(s)}_{\YY_\mu} \, . \label{EQthisijhhssdinthegdsfg338:L} \end{align} In order to estimate the term $I_4$ in \eqref{EQthisijhhssdinthegdsfg:dec:non:L} we appeal to \eqref{EQthisijhhssdinthegdsfg:y:z}, use that $\ee$ is chosen sufficiently small, and the bound $y \leq 1+\mu \leq z$, to obtain \begin{align} e^{\ee(1+\mu-y)_+|\xi|} |I_{4}| &\les \int_{1+\mu}^{\infty} \frac{1}{\sqrt{\nu(t-s)}} e^{\ee(z-y)_+|\xi|} e^{-\frac{(y-z)^2}{4\nu (t-s)}} e^{-\frac{1}{2}\nu\xi^2(t-s)} | \partial_z N_\xi(s,z)| \,dz \notag\\ &\les \int_{1+\mu}^{\infty} \frac{e^{-\frac{(y-z)^2}{8\nu (t-s)}}}{\sqrt{\nu(t-s)}} | \partial_z N_\xi(s,z)| \,dz \, . \notag \end{align} Upon integrating in $y$, using \eqref{EQthisijhhssdinthegdsfg:unit:mass}, and summing in $\xi$, the above estimate yields \begin{align} \norm{I_4}_{\YY_\mu} \les \sum_\xi \norm{\partial_z N_\xi (s)}_{L^1(z\geq 1+\mu)} \les \norm{\partial_z N(s)}_{S_\mu} \, . \notag \end{align} This concludes the proof of \eqref{EQthisijhhssdinthegdsfg1024} with $(i,j) = (0,1)$. \par The estimate \eqref{EQthisijhhssdinthegdsfg1024} for $(i,j)=(1,0)$ follows from the bound \eqref{EQthisijhhssdinthegdsfg1024} with $(i,j)=(0,0)$, by applying the estimate to $\partial_x N$ instead of $N$. In order to prove \eqref{EQthisijhhssdinthegdsfg1024} for $(i,j)=(0,0)$, we decompose, as in \eqref{EQthisijhhssdinthegdsfg215}, \begin{align} & \int_0^\infty H_\xi(t-s, y, z)N_\xi(s, z) \,dz \notag \\&\indeq = \int_0^{3y/4} \psi\left(\frac{z}{y}\right)H_\xi(t-s, y, z) N_\xi(s, z) \,dz + \int_{y/2}^{1+\mu} \left(1-\psi\left(\frac{z}{y}\right)\right)H_\xi(t-s, y, z) N_\xi(s, z) \,dz \notag \\&\indeq\indeq + \int_{1+\mu}^\infty H_\xi(t-s, y, z) N_\xi(s, z) \,dz \notag \\&\indeq =J_1+J_2+J_3 \, . \llabel{sldkfjg;ls dslksfgj s;lkdfgj ;sldkfjg ;sldkfgj s;ldfkgj s;ldfkgj gfisasdoifaghskcx,.bvnliahglidhgsd gs sdfgh sdfg sldfg sldfgkj slfgj sl;dfgkj sl;dfgjk sldgfkj jiurwe alskjfa;sd fasdfEQthisijhhssdinthegdsfg215:N} \end{align} Upon inspection of the proof for $(i,j)=(0,1)$, we see that using \eqref{EQthisijhhssdinthegdsfg:unit:mass} we obtain \begin{align} \norm{J_1}_{\YY_{\mu}} + \norm{J_2}_{\YY_{\mu}} \les \norm{N(s)}_{\YY_{\mu}} \, . \notag \end{align} On the other hand, the term $J_3$ is estimated exactly as the term $I_4$ above, and we obtain \begin{align} \norm{J_3}_{\YY_{\mu}} \les \norm{N(s)}_{S_{\mu}} \, . \notag \end{align} This concludes the proof of the lemma. \end{proof} \par Next, we state the inequalities involving the remainder kernel $ R_\xi$. \par \cole \begin{lemma} \label{rem:l} Let $\mu \in (0, \mu_0 - \gamma s)$ be arbitrary. For $(i,j) \in \{ (0,0), (1,0), (0,1) \}$ we have the estimate \begin{align} \left\lVert \partial_x^i (y\partial_y)^j \int_0^\infty R_\xi(t-s, y, z)N_\xi(s, z) \,dz\right\rVert_{\YY_\mu} \lesssim \lVert \partial_x^i (y\partial_y)^j N(s)\rVert_{\YY_{\mu}} + \lVert N(s)\rVert_{\YY_{\mu}} + \lVert \partial_x^i \partial_y^j N(s) \rVert_{S_{\mu}} \, . \label{EQthisijhhssdinthegdsfg244} \end{align} \end{lemma} \colb \par \begin{proof}[Proof of Lemma~\ref{rem:l}] In order to establish \eqref{EQthisijhhssdinthegdsfg244}, we only need to verify that the kernel $R_\xi(t,y,z)$ obeys the conditions stated in Remark~\ref{R07}. According to Remark~\ref{R01}, in order to obtain \eqref{EQthisijhhssdinthegdsfg:R07:1}--\eqref{EQthisijhhssdinthegdsfg:R07:2}, we only need to prove that \begin{align} \norm{\chi_{\{0\leq y \leq1+\mu\}} \chi_{\{0\leq z \leq 3y/4\}} \bigl( be^{-\frac{1}{2}\theta_0b(y+z)} \bigr)}_{L^1_y L^\infty_z} &\les 1 \label{EQthisijhhssdinthegdsfg23}\\ \norm{e^{\ee (z-y)_+ |\xi|} be^{-\theta_0b(y+z)} }_{L^1_y L^\infty_z} &\les 1 \label{EQthisijhhssdinthegdsfg34} \, . \end{align} Indeed, the second term in the upper bound \eqref{EQthisijhhssdinthegdsfg33} for the residual kernel is treated in exactly the same way as $H_\xi(t,y,-z)$, but replacing $\frac{1}{4}$ with $\frac{\theta_0}{2}$, and this term was addressed in Remark~\ref{R07}. \par In order to establish \eqref{EQthisijhhssdinthegdsfg23}, let $y\in[0,1+\mu]$ and $z\in[0,3y/4]$. Then \begin{align} & \nnorm{be^{-\frac12 \theta_0 b (y+z)}}_{L^1_y} \les	 \nnorm{be^{-\frac12 \theta_0 b y}}_{L^1_y} \les 1 \, , \llabel{sldkfjg;ls dslksfgj s;lkdfgj ;sldkfjg ;sldkfgj s;ldfkgj s;ldfkgj gfisasdoifaghskcx,.bvnliahglidhgsd gs sdfgh sdfg sldfg sldfgkj slfgj sl;dfgkj sl;dfgjk sldgfkj jiurwe alskjfa;sd fasdfEQthisijhhssdinthegdsfg05} \end{align} and \eqref{EQthisijhhssdinthegdsfg23} follows. For \eqref{EQthisijhhssdinthegdsfg34}, let $\epsilon_0\leq \theta_0$, and observe that $ e^{\ee (z-y)_+ |\xi|} b e^{-\theta_0 b (z+y)} \les be^{- \theta_0 b y} $. The inequality \eqref{EQthisijhhssdinthegdsfg34} then follows upon integration in $y$. \end{proof} \par Next, we consider the $Y$ norm estimate for the trace kernel contribution to \eqref{EQthisijhhssdinthegdsfg:kernel:est}. \cole \begin{lemma} \label{L18} Let $\mu \in (0, \mu_0 - \gamma s)$ be arbitrary. For $0\leq i+j\leq 1$, we have the inequality \begin{align} \left\lVert \partial_x^i (y\partial_{y})^{j} G(t-s, y,0)\partial_{z}\Delta^{-1}N_\xi(s, z)|_{z=0}\right\rVert_{\YY_\mu} \lesssim \lVert \partial_x^i N(s)\rVert_{\YY_{\mu}} + \lVert \partial_x^i N \rVert_{S_{\mu}} \, . \label{EQthisijhhssdinthegdsfg270} \end{align} \end{lemma} \colb \par \begin{proof}[Proof of Lemma~\ref{L18}] First, we note that the case $(i,j) = (1,0)$ follows from the bound \eqref{EQthisijhhssdinthegdsfg270} with $(i,j)=(0,0)$, because the $\partial_x$ derivative commutes with the operator $G(t-s, y,0)\partial_{z}\Delta^{-1}|_{z=0}$ (see also the formula \eqref{EQthisijhhssdinthegdsfg265} below). Second, we emphasize that the case $(i,j) = (0,1)$ is treated in the same way as the case $(i,j) = (0,0)$, because the conormal derivative $y\partial_y$ of $G(t-s,y,0)$ obeys the same bounds as $G(t-s,y,0)$ itself (see the bounds~\eqref{EQthisijhhssdinthegdsfg169}--\eqref{EQthisijhhssdinthegdsfg169b} above). Therefore, we only need to consider the case $(i,j) = (0,0)$. As opposed to the proof of Lemma~\ref{tra:x}, we do not split the kernel $G_\xi(t-s,y,0)$ into two parts. The only property of the kernel which is used in this estimate is \begin{align} \norm{G_\xi(t-s,y,0)}_{L^1_y} \les 1 \, , \label{EQthisijhhssdinthegdsfg:mass:1} \end{align} which follows directly from \eqref{EQthisijhhssdinthegdsfg167} and \eqref{EQthisijhhssdinthegdsfg169b}. \par Using \eqref{EQthisijhhssdinthegdsfg229}, we have \begin{align} \partial_{z}\Delta^{-1}N_\xi(s, z)|_{z=0} &= - \int_0^\infty e^{-|\xi| z}N_\xi(s, z) \,dz \label{EQthisijhhssdinthegdsfg265} \end{align} and thus, since $\ee$ may be taken sufficiently small, we obtain \begin{align} &\left | e^{\ee(1+\mu-y)_+|\xi|}G_\xi(t-s,y,0) \partial_{z}\Delta^{-1}N_\xi(s, z)|_{z=0} \right| \notag \\&\indeq\les G_\xi(t-s,y,0) \int_0^{\infty} e^{-|\xi| z}e^{\ee(z-y)_+|\xi|}e^{\ee(1+\mu-z)_+|\xi|}|N_\xi(s, z)| \,dz \notag \\&\indeq\lesssim G_\xi(t-s,y,0) \int_0^{1+\mu} e^{\ee(1+\mu-z)_+|\xi|}|N_\xi(s, z)| \,dz + G_\xi(t-s,y,0) \int_{1+\mu}^\infty |N_\xi(s, z)| \,dz \, . \llabel{sldkfjg;ls dslksfgj s;lkdfgj ;sldkfjg ;sldkfgj s;ldfkgj s;ldfkgj gfisasdoifaghskcx,.bvnliahglidhgsd gs sdfgh sdfg sldfg sldfgkj slfgj sl;dfgkj sl;dfgjk sldgfkj jiurwe alskjfa;sd fasdfEQthisijhhssdinthegdsfg266} \end{align} Using \eqref{EQthisijhhssdinthegdsfg:mass:1} and summing over $\xi$, we arrive at \begin{align} \left\lVert G_\xi(t-s,y,0) \partial_{z}\Delta^{-1}N_\xi(s, z)|_{z=0} \right\rVert_{\YY_\mu} \lesssim \lVert N(s)\rVert_{\YY_{\mu}} + \norm{N(s)}_{S_\mu} \, , \llabel{sldkfjg;ls dslksfgj s;lkdfgj ;sldkfjg ;sldkfgj s;ldfkgj s;ldfkgj gfisasdoifaghskcx,.bvnliahglidhgsd gs sdfgh sdfg sldfg sldfgkj slfgj sl;dfgkj sl;dfgjk sldgfkj jiurwe alskjfa;sd fasdfEQthisijhhssdinthegdsfg267} \end{align} which concludes the proof of the lemma. \end{proof} \par Next, we provide an inequality corresponding to the initial datum. \par \cole \begin{lemma} \label{L03} Let $\mu \in (0,\mu_0 - \gamma t)$. For $i+j\leq 2$, the initial datum term in \eqref{EQthisijhhssdinthegdsfg:kernel:est} satisfies \begin{align} & \sum_{i+j\leq 2} \norm{ \partial_x^i (y \partial_y)^j \int_0^\infty G(t, y,z)\omega_{0}(z) \,dz}_{\YY_\mu} \notag \\&\indeq \les \sum_{i+j\leq 2}\lVert \partial_x^i (y\partial_{y})^j \omega_0\rVert_{\YY_{\mu}} + \sum_{i+j\leq 2} \sum_{\xi} \lVert \xi^i \partial_{y}^j \omega_{0,\xi}\rVert_{L^1(y\ge 1+\mu)} \, . \label{EQthisijhhssdinthegdsfg40} \end{align} \end{lemma} \colb \par {\begin{proof}[Proof of Lemma~\ref{L03}] Let $i+j\leq2$. Then we have the decomposition of the kernel \eqref{EQthisijhhssdinthegdsfg42}. We start with the first kernel in \eqref{EQthisijhhssdinthegdsfg42} and consider the inequality \eqref{EQthisijhhssdinthegdsfg47}, where $J_1$ is as in \eqref{EQthisijhhssdinthegdsfg46}. \par Now, the terms $J_{11}$, $J_{12}$, and $J_{13}$ are bounded the same as the term $I_1$ in \eqref{EQthisijhhssdinthegdsfg:dec:non:L} (see \eqref{EQthisijhhssdinthegdsfg12:L}--\eqref{EQthisijhhssdinthegdsfg45:L}), giving the first term in \eqref{EQthisijhhssdinthegdsfg40}. The terms $J_{14}$, $J_{16}$, and $J_{18}$ are estimated in the same way as the term $I_3$ in \eqref{EQthisijhhssdinthegdsfg:dec:non:L}, cf.~\eqref{EQthisijhhssdinthegdsfg338:L}, and are bounded by the first term in \eqref{EQthisijhhssdinthegdsfg37}. It remains to consider the Sobolev contributions $J_{15}$, $J_{17}$, and $J_{19}$. \par For $J_{15}$ we use \eqref{EQthisijhhssdinthegdsfg:y:z} and write \begin{align} e^{\ee(1+\mu-y)_+|\xi|} |J_{15}| &\les \int_{1+\mu}^{\infty} \frac{1}{\sqrt{\nu(t-s)}} e^{\ee(z-y)_+|\xi|} e^{-\frac{(y-z)^2}{4\nu (t-s)}} e^{-\frac{1}{2}\nu\xi^2(t-s)} | \partial_z^2 \omega_{0,\xi}(z)| \,dz \notag\\ &\les \int_{1+\mu}^{\infty} \frac{1}{\sqrt{\nu(t-s)}} e^{-\frac{(y-z)^2}{8\nu (t-s)}} | \partial_z^2 \omega_{0,\xi}(z)| \,dz \, . \notag \end{align} Upon integrating in $y$, using Fubini, the estimate \eqref{EQthisijhhssdinthegdsfg:unit:mass}, and summing in $\xi$, we obtain \begin{align} \norm{J_{15}}_{Y_\mu} \les \sum_\xi \norm{\partial_z^2 \omega_{0,\xi}}_{L^1(z\geq 1+\mu)} \,. \notag \end{align} With a similar treatment of $J_{17}$ and $J_{19}$ we obtain the second term on the right side of \eqref{EQthisijhhssdinthegdsfg40}. \par Since the other kernels in \eqref{EQthisijhhssdinthegdsfg42} are treated completely analogously, the proof is concluded. \end{proof} \par \begin{proof}[Proof of Lemma~\ref{lem:main:Y}] By increasing the analyticity domain from $\mu$ to $\mu_1$, which is defined in \eqref{EQthisijhhssdinthegdsfg02}, and using the analyticity recovery for the $Y$~norm in Lemma~\ref{ana:rec:l}, we obtain \begin{align} &\sum_{i+j= 2} \left\lVert \partial_x^i (y\partial_{y})^j\int_0^\infty G(t-s, y, z)N(s, z) \,dz\right\rVert_{\YY_\mu} \notag\\ &\qquad \les \frac{1}{\mu_0-\mu-\gamma s} \sum_{i+j= 1} \left\lVert \partial_x^i (y\partial_{y})^j\int_0^\infty G(t-s, y, z)N(s, z) \,dz\right\rVert_{\YY_{\mu_1}} \, . \notag \end{align} Therefore, the bound for the first term on the left of \eqref{EQthisijhhssdinthegdsfg706} is a direct consequence of the estimate for the second term in \eqref{EQthisijhhssdinthegdsfg706}. The bound for the second term on the left side of \eqref{EQthisijhhssdinthegdsfg706} follows from Lemma~\ref{high:y:x:L}, Remark~\ref{R07}, and Lemma~\ref{rem:l}, with $\mu$ replaced by $\mu_1 \in (0,\mu_0-\gamma s)$. \par Similarly, using analytic recovery for the $Y$ norm and increasing the analytic domain from $\mu$ to $\mu_1$, we see that the bound for the first term on the left side of \eqref{EQthisijhhssdinthegdsfg707} is a direct consequence of the estimate for the second term. For this later term, the estimate is established in Lemma~\ref{L18}, with $\mu$ replaced by $\mu_1$. Lastly, the bound \eqref{EQthisijhhssdinthegdsfg708} is proven in Lemma~\ref{L03}, concluding the proof of the lemma. \end{proof} \section{Estimates for the nonlinearity} \label{sec-N} In this section we provide estimates for the nonlinear term \begin{align} N_\xi=(u\cdot\nabla\omega)_\xi= (u_1\partial_{x}\omega)_\xi+\left(\frac{u_2}{y}y\partial_{y}\omega\right)_\xi \label{EQthisijhhssdinthegdsfg129} \end{align} and its $\partial_x^i (y\partial_y)^j$ derivatives, with $i+j\leq 1$, in the $X_\mu$, $\YY_\mu$, and $S_\mu$ norms. We first recall a representation formula of the velocity field in terms of the vorticity. \cole \begin{lemma}[Lemma~2.4 in~\cite{Maekawa14}] \label{bio:sav} The velocity for the system~\eqref{EQthisijhhssdinthegdsfg:vot}--\eqref{EQthisijhhssdinthegdsfg:bdry:vot} is given by \begin{align} u_{1,\xi}(y) =& \frac{1}{2}\left(-\int_0^y e^{-|\xi|(y-z)}(1-e^{-2|\xi| z})\omega_\xi(z) \,dz + \int_y^\infty e^{-|\xi|(z-y)}(1+e^{-2|\xi| y})\omega_\xi(z) \,dz\right) \label{EQthisijhhssdinthegdsfg123} \end{align} and \begin{align} u_{2,\xi}(y) =& \frac{-\ii \xi}{2|\xi|} \left(\int_0^y e^{-|\xi|(y-z)}(1-e^{-2|\xi| z})\omega_\xi(z) \,dz + \int_y^\infty e^{-|\xi|(z-y)}(1-e^{-2|\xi| y})\omega_\xi(z) \,dz\right) \, , \label{EQthisijhhssdinthegdsfg124} \end{align} where $\ii$ is the imaginary unit. \end{lemma} \colb \par As in Remark~\ref{R03} above, the Biot-Savart law of Lemma~\ref{bio:sav} also holds for $y$ in the complex domain $\Omega_\mu \cup [1+\mu,\infty)$. If $y \in \partial \Omega_\theta$ for some $\theta \in [0,\mu)$, and say $\Im y \geq 0$, then the integration from $0$ to $y$ in \eqref{EQthisijhhssdinthegdsfg123}--\eqref{EQthisijhhssdinthegdsfg124} is an integration over the complex line $\partial \Omega_\theta \cap \{z\colon \Im z \geq 0, \Re z \leq \Re y \}$, while the integration from $y$ to $\infty$ is an integration over $( \partial \Omega_\theta \cap \{z\colon \Im z \geq 0, \Re y \leq \Re z \leq 1+\theta \}) \cup [1+\theta, \infty)$. \par Moreover, we emphasize here that while \eqref{EQthisijhhssdinthegdsfg124} immediately implies the boundary condition $u_{2,\xi}(0) = 0$, from \eqref{EQthisijhhssdinthegdsfg123} it just follows that $u_{1,\xi}(0) = \int_0^\infty e^{-|\xi| z} \omega_\xi(z) dz$. To see that this integral vanishes, one has to use that it vanishes at time $t=0$, and that its time derivative is given using the vorticity boundary condition \eqref{EQthisijhhssdinthegdsfg:Stokes} as $\partial_t u_{1,\xi}(0) = (-\partial_y \Delta_\xi^{-1}(u\cdot \nabla \omega)_\xi)|_{y=0} - \int_0^\infty e^{-|\xi| z} (u\cdot \nabla \omega)_\xi(z) dz = 0$. In the last equality we have used explicitly that the kernel of the operator $(-\partial_y \Delta_\xi^{-1})|_{y=0}$ is given by $e^{-|\xi| z}$. Thus, \eqref{EQthisijhhssdinthegdsfg:Stokes} ensures that $u_{1,\xi}(0) = 0$ is maintained by the evolution. \par The main estimate concerning the $X_\mu$ norm is the following. \cole \begin{lemma} \label{L01} Let $\mu \in (0, \mu_0 - \gamma s)$ be arbitrary. We have the inequalities \begin{align} \lVert N(s)\rVert_{X_{\mu}} \lesssim \sum_{i\leq 1} \left(\lVert \partial_{x}^i\omega\rVert_{\YY_{\mu}}+\lVert \partial_{x}^i\omega\rVert_{S_{\mu}}\right) \sum_{i+j=1}\lVert \partial_{x}^i(y\partial_{y})^j\omega\rVert_{X_{\mu}} \label{EQthisijhhssdinthegdsfg128} \end{align} and \begin{align} \sum_{i+j=1}\lVert \partial_{x}^{i}(y\partial_{y})^jN(s)\rVert_{X_{\mu}} & \lesssim \biggl( \lVert \omega\rVert_{X_{\mu}} + \sum_{1\leq i\leq 2}\bigl( \lVert \partial_{x}^{i}\omega\rVert_{\YY_{\mu}} +\lVert \partial_{x}^{i}\omega\rVert_{S_{\mu}} \bigr) \biggr) \sum_{i+j=1}\lVert \partial_{x}^i(y\partial_{y})^j\omega\rVert_{X_{\mu}} \notag \\& \indeq + \sum_{i\leq 1} \left(\lVert \partial_{x}^i\omega\rVert_{\YY_{\mu}} +\lVert \partial_{x}^i\omega\rVert_{S_{\mu}}\right) \sum_{i+j=2}\lVert \partial_{x}^i(y\partial_{y})^j\omega\rVert_{X_{\mu}} \, . \label{EQthisijhhssdinthegdsfg127} \end{align} \end{lemma} \colb \par Before the proof of Lemma~\ref{L01}, we analyze the first order derivatives of the nonlinear term. By the Leibniz rule, for $i+j = 1$, we have \begin{align} \partial_{x}^i(y\partial_{y})^jN_\xi &= ( \partial_{x}^i(y\partial_{y})^j u_1\partial_{x}\omega)_\xi+\left((y\partial_{y})^j\left(\frac{ \partial_{x}^i u_2}{y}\right) y\partial_{y}\omega\right)_\xi \notag\\ & \quad +(u_1 \partial_{x}^{i+1}(y\partial_{y})^j\omega)_\xi+\left(\frac{u_2}{y} \partial_{x}^i(y\partial_{y})^{j+1}\omega\right)_\xi \, . \llabel{sldkfjg;ls dslksfgj s;lkdfgj ;sldkfjg ;sldkfgj s;ldfkgj s;ldfkgj gfisasdoifaghskcx,.bvnliahglidhgsd gs sdfgh sdfg sldfg sldfgkj slfgj sl;dfgkj sl;dfgjk sldgfkj jiurwe alskjfa;sd fasdfEQthisijhhssdinthegdsfg:der:non} \end{align} Using the triangle inequality we have $e^{\ee(1+\mu-y)_+ |\xi|} \leq e^{\ee(1+\mu-y)_+ |\eta|} e^{\ee(1+\mu-y)_+ |\xi-\eta|}$, and thus, by the definition of the $X_{\mu}$ norm and Young's inequality in $\xi$ and $\eta$, it follows that \begin{align} \lVert \partial_{x}^i(y\partial_{y})^j N(s)\rVert_{X_{\mu}} & \les \lVert \partial_{x}\omega\rVert_{X_{\mu}} \sum_{\xi}\sup_{y\in\Omega_{\mu}}e^{\ee(1+\mu-y)_+|\xi|}|(\partial_{x}^i(y\partial_{y})^ju_1)_\xi| \notag \\&\indeq +\lVert y\partial_{y}\omega\rVert_{X_{\mu}} \sum_{\xi}\sup_{y\in\Omega_{\mu}}e^{\ee(1+\mu-y)_+|\xi|}\left|\left((y\partial_{y})^j \left(\frac{\partial_{x}^i u_2}{y}\right)\right)_\xi\right| \notag \\&\indeq +\lVert \partial_{x}^{i+1} (y\partial_y)^j \omega\rVert_{X_{\mu}} \sum_{\xi}\sup_{y\in\Omega_{\mu}}e^{\ee(1+\mu-y)_+|\xi|}|(u_1)_\xi| \notag \\&\indeq +\lVert \partial_{x}^i(y\partial_{y})^{j+1} \omega\rVert_{X_{\mu}} \sum_{\xi}\sup_{y\in\Omega_{\mu}}e^{\ee(1+\mu-y)_+|\xi|}\left|\left(\frac{u_2}{y}\right)_\xi\right| \, . \label{EQthisijhhssdinthegdsfg130} \end{align} Thus, in order to prove \eqref{EQthisijhhssdinthegdsfg127}, we only need to estimate the above norms of the velocity terms. These inequalities are collected in the next lemma. \par \cole \begin{lemma} \label{L14}Let $\mu \in (0, \mu_0 - \gamma s)$ be arbitrary and let $0 \leq i+j \leq 1$. For the velocity $u_1$ and its derivatives, we have \begin{align} \sum_{\xi}\sup_{y\in\Omega_{\mu}}e^{\ee(1+\mu-y)_+|\xi|}|( \partial_x^i (y\partial_y)^j u_1)_\xi| \lesssim \lVert\partial_x^{i+j} \omega\rVert_{\YY_{\mu}}+\lVert\partial_x^{i+j} \omega\rVert_{S_{\mu}} + j \norm{\omega}_{X_\mu} \label{EQthisijhhssdinthegdsfg:X:mu:u1} \, , \end{align} while for the second velocity component $u_2$ the bound \begin{align} \sum_{\xi}\sup_{y\in\Omega_{\mu}}e^{\ee(1+\mu-y)_+|\xi|} \left| \left((y \partial_y)^j\left(\frac{\partial_x^i u_2}{y}\right)\right)_\xi\right| \lesssim \lVert \partial_{x}^{i+1}\omega\rVert_{\YY_{\mu}} +\lVert \partial_{x}^{i+1}\omega\rVert_{S_{\mu}} \label{EQthisijhhssdinthegdsfg:X:mu:u2} \end{align} holds. \end{lemma} \colb \par \begin{proof}[Proof of Lemma~\ref{L14}] First we prove \eqref{EQthisijhhssdinthegdsfg:X:mu:u1}, starting with the case $(i,j) = (0,0)$. We decompose the integral \eqref{EQthisijhhssdinthegdsfg123} for $u_1$ as \begin{align} u_{1,\xi}(y) =& \frac{1}{2}\biggl(-\int_0^y e^{-|\xi|(y-z)}(1-e^{-2|\xi| z})\omega_\xi(s, z) \,dz \notag \\&\indeq + \left(\int_y^{1+\mu}+\int_{1+\mu}^\infty\right) e^{-|\xi|(z-y)}(1+e^{-2|\xi| y})\omega_\xi(s, z) \,dz\biggr) \notag \\=& I_{1}+I_{2}+I_{3} \, . \notag \end{align} Note that we have \begin{align} &e^{\ee(1+\mu-y)_+|\xi|}e^{- |y-z||\xi|} \leq e^{\ee(1+\mu-z)_+|\xi|} e^{\ee(z-y)_+|\xi|} e^{- |y-z||\xi|} \le e^{\ee(1+\mu-z)_+|\xi|} \label{EQthisijhhssdinthegdsfg133} \end{align} provided $\epsilon_0\leq 1$. Hence, we obtain \begin{align} e^{\ee(1+\mu-y)_+|\xi|} (|I_1|+|I_2|) \les \int_{0}^{1+\mu} e^{\ee(1+\mu-z)_+|\xi|} |\omega_{\xi}(s,z)| \,dz \les \lVert e^{\ee(1+\mu-y)_+|\xi|}\omega\rVert_{\SL_{\mu}} \, . \label{EQthisijhhssdinthegdsfg134} \end{align} For the term $I_3$, using \eqref{EQthisijhhssdinthegdsfg133} we have \begin{align} e^{\ee(1+\mu-y)_+|\xi|}|I_{3}| &\les \int^{\infty}_{1+\mu} |\omega_\xi(s, z)| \,dz \lesssim \lVert z \omega_\xi\rVert_{L^2(z \geq 1+\mu)} \, . \label{EQthisijhhssdinthegdsfg136} \end{align} Summing the bounds \eqref{EQthisijhhssdinthegdsfg134} and \eqref{EQthisijhhssdinthegdsfg136} in $\xi$, we conclude the proof of \eqref{EQthisijhhssdinthegdsfg:X:mu:u1} when $i+j=0$. \par The case $(i,j) = (1,0)$ amounts to multiplying by $\ii \xi$, and thus the assertion follows by the same proof as for $(i,j) = (0,0)$. Consider now the case $(i,j)=(0,1)$. Taking the conormal derivative of \eqref{EQthisijhhssdinthegdsfg123} gives \begin{align} y\partial_{y}u_{1,\xi}=& \frac{y}{2} \biggl( \int_0^y e^{-|\xi|(y-z)}(1-e^{-2|\xi| z}) |\xi| \omega_\xi(s, z) \,dz \notag \\&\indeq + \int_y^\infty e^{-|\xi|(z-y)}(1+e^{-2|\xi| y}) |\xi| \omega_\xi(s, z) \,dz \notag \\&\indeq - 2 \int_y^\infty e^{-|\xi|(z-y)}e^{-2|\xi| y}|\xi| \omega_\xi(s, z) \,dz\biggr) -y\omega_\xi(y) \, . \label{EQthisijhhssdinthegdsfg:der:y:u1} \end{align} The first three terms in \eqref{EQthisijhhssdinthegdsfg:der:y:u1} are treated as in the case $i+j=0$. The presence of the additional factor $|\xi|$ causes $\omega$ to be replaced by $\partial_x \omega$ in the upper bounds. For the last term in \eqref{EQthisijhhssdinthegdsfg:der:y:u1}, we have \begin{align} \sum_{\xi}\sup_{y\in\Omega_{\mu}}e^{\ee(1+\mu-y)_+|\xi|}y|\omega_\xi(y)| &\les \sum_{\xi}\sup_{y\in\Omega_{\mu}}e^{\ee(1+\mu-y)_+|\xi|}w(y)|\omega_\xi(y)| \les \lVert \omega\rVert_{X_{\mu}} \, , \llabel{sldkfjg;ls dslksfgj s;lkdfgj ;sldkfjg ;sldkfgj s;ldfkgj s;ldfkgj gfisasdoifaghskcx,.bvnliahglidhgsd gs sdfgh sdfg sldfg sldfgkj slfgj sl;dfgkj sl;dfgjk sldgfkj jiurwe alskjfa;sd fasdfEQthisijhhssdinthegdsfg:pt:u1} \end{align} where we have used Remark~\ref{R04}(d). This concludes the proof of \eqref{EQthisijhhssdinthegdsfg:X:mu:u1} for $(i,j)= (0,1)$. \par Next, we prove \eqref{EQthisijhhssdinthegdsfg:X:mu:u2}, beginning with the case $(i,j) = (0,0)$. Using \eqref{EQthisijhhssdinthegdsfg124} we decompose \begin{align} \frac{u_{2,\xi}}{y} =& - \frac{\xi}{|\xi|}\frac{\ii}{2y} \biggl(\int_0^y e^{-|\xi|(y-z)}(1-e^{-2|\xi| z})\omega_\xi(s, z) \,dz \notag \\&\indeq\indeq\indeq\indeq\indeq\indeq+ \left(\int_y^{1+\mu}+\int_{1+\mu}^\infty\right) e^{-|\xi|(z-y)}(1-e^{-2|\xi| y})\omega_\xi(s, z) \,dz\biggr) \notag \\=& J_1 + J_2 + J_3 \, . \llabel{sldkfjg;ls dslksfgj s;lkdfgj ;sldkfjg ;sldkfgj s;ldfkgj s;ldfkgj gfisasdoifaghskcx,.bvnliahglidhgsd gs sdfgh sdfg sldfg sldfgkj slfgj sl;dfgkj sl;dfgjk sldgfkj jiurwe alskjfa;sd fasdfEQthisijhhssdinthegdsfg139} \end{align} Using the bound \begin{equation} \left|\frac{1-e^{-2|\xi| z}}{y}\right| \les |\xi| \comma z\leq y \, , \llabel{sldkfjg;ls dslksfgj s;lkdfgj ;sldkfjg ;sldkfgj s;ldfkgj s;ldfkgj gfisasdoifaghskcx,.bvnliahglidhgsd gs sdfgh sdfg sldfg sldfgkj slfgj sl;dfgkj sl;dfgjk sldgfkj jiurwe alskjfa;sd fasdfEQthisijhhssdinthegdsfg140} \end{equation} we arrive at \begin{align} \left| \frac{u_{2,\xi}}{y} \right| &\les \int_0^y e^{-|\xi|(y-z)} |\xi| |\omega_\xi(s, z)| \,dz + \left(\int_y^{1+\mu}+\int_{1+\mu}^\infty\right) e^{-|\xi|(z-y)} |\xi| |\omega_\xi(s, z)| \,dz \, . \label{EQthisijhhssdinthegdsfg:MJ} \end{align} Using \eqref{EQthisijhhssdinthegdsfg133} and the same bounds as in \eqref{EQthisijhhssdinthegdsfg134}--\eqref{EQthisijhhssdinthegdsfg136}, we obtain the inequality \eqref{EQthisijhhssdinthegdsfg:X:mu:u2} for $i+j=0$. The case $(i,j)=(1,0)$ follows from the same argument, by adding an extra $x$~derivative. \par It remains to consider the case $(i,j)=(0,1)$. From the incompressibility we have \begin{align} \label{EQthisijhhssdinthegdsfg:der:y:u2} y\partial_{y}\left(\frac{u_{2,\xi}}{y}\right) = \partial_y u_{2,\xi} - \frac{u_{2,\xi}}{y} = - \ii \xi u_{1,\xi} - \frac{u_{2,\xi}}{y} \, . \end{align} The bound for the second term on the right of \eqref{EQthisijhhssdinthegdsfg:der:y:u2} was established in \eqref{EQthisijhhssdinthegdsfg:MJ}, whereas the bound for the first term follows by setting $(i,j) = (1,0)$ in \eqref{EQthisijhhssdinthegdsfg:X:mu:u1}. \end{proof} \par Having established Lemma~\ref{L14}, we return to the proofs of \eqref{EQthisijhhssdinthegdsfg128} and \eqref{EQthisijhhssdinthegdsfg127}. \par \begin{proof}[Proof of Lemma~\ref{L01}] In order to prove \eqref{EQthisijhhssdinthegdsfg128}, we use \eqref{EQthisijhhssdinthegdsfg129} and similarly to \eqref{EQthisijhhssdinthegdsfg130} we obtain \begin{align} \lVert N(s)\rVert_{X_{\mu}} &\les \lVert \partial_{x}\omega\rVert_{X_{\mu}} \sum_{\xi}\sup_{y\in\Omega_{\mu}}e^{\ee(1+\mu-y)_+|\xi|}|(u_1)_\xi| \notag \\&\indeq + \lVert y\partial_{y}\omega\rVert_{X_{\mu}} \sum_{\xi}\sup_{y\in\Omega_{\mu}}e^{\ee(1+\mu-y)_+|\xi|}\left|\left(\frac{u_2}{y}\right)_\xi\right| \, . \label{EQthisijhhssdinthegdsfg137} \end{align} Using Lemma~\ref{L14} with $i+j = 0$ we get \begin{align} \lVert N(s)\rVert_{X_{\mu}} &\les \bigl(\lVert \omega\rVert_{\YY_{\mu}}+\lVert \omega\rVert_{S_{\mu}}\bigr) \lVert \partial_{x}\omega\rVert_{X_{\mu}} + \bigl(\lVert \partial_{x}\omega\rVert_{\YY_{\mu}}+\lVert \partial_{x}\omega\rVert_{S_{\mu}}\bigr) \lVert y\partial_{y}\omega\rVert_{X_{\mu}} \, , \llabel{sldkfjg;ls dslksfgj s;lkdfgj ;sldkfjg ;sldkfgj s;ldfkgj s;ldfkgj gfisasdoifaghskcx,.bvnliahglidhgsd gs sdfgh sdfg sldfg sldfgkj slfgj sl;dfgkj sl;dfgjk sldgfkj jiurwe alskjfa;sd fasdfEQthisijhhssdinthegdsfg135} \end{align} and \eqref{EQthisijhhssdinthegdsfg128} is established. \par For \eqref{EQthisijhhssdinthegdsfg127}, we use the bounds of Lemma~\ref{L14} in \eqref{EQthisijhhssdinthegdsfg130} to obtain \begin{align} \sum_{i+j=1}\lVert \partial_{x}^i(y\partial_{y})^jN(s)\rVert_{X_{\mu}} &\les \lVert \partial_{x}\omega\rVert_{X_{\mu}} \left( \lVert \partial_{x}\omega\rVert_{\YY_{\mu}} +\lVert \partial_{x}\omega\rVert_{S_{\mu}} + \lVert \omega\rVert_{X_{\mu}} \right) \notag \\&\indeq + \lVert y\partial_{y}\omega\rVert_{X_{\mu}} \biggl(\sum_{i\leq 1} \lVert \partial_{x}^{i+1}\omega\rVert_{\YY_{\mu}} +\lVert \partial_{x}^{i+1}\omega\rVert_{S_{\mu}} \biggr) \notag \\&\indeq + \biggl(\sum_{i+j=1} \lVert \partial_{x}^{i+1} (y\partial_y)^j \omega\rVert_{X_{\mu}} \biggr) \left( \lVert \omega\rVert_{\YY_{\mu}}+\lVert \omega\rVert_{S_{\mu}} \right) \notag \\&\indeq + \biggl(\sum_{i+j=1} \lVert \partial_{x}^{i} (y\partial_y)^{j+1} \omega\rVert_{X_{\mu}} \biggr) \left( \lVert \partial_{x}\omega\rVert_{\YY_{\mu}}+\lVert \partial_{x}\omega\rVert_{S_{\mu}} \right) \, , \llabel{sldkfjg;ls dslksfgj s;lkdfgj ;sldkfjg ;sldkfgj s;ldfkgj s;ldfkgj gfisasdoifaghskcx,.bvnliahglidhgsd gs sdfgh sdfg sldfg sldfgkj slfgj sl;dfgkj sl;dfgjk sldgfkj jiurwe alskjfa;sd fasdfEQthisijhhssdinthegdsfg141} \end{align} and \eqref{EQthisijhhssdinthegdsfg127} is proven. \end{proof} \par Next, we estimate the term $ \partial_x^i(y\partial_{y})^jN(s)$ for $0\le i+j\le 1$ in the $\YY$ norm. \cole \begin{lemma} \label{L09} Let $\mu \in (0, \mu_0 - \gamma s)$ be arbitrary. For the nonlinear term, we have the inequalities \begin{align} \lVert N(s)\rVert_{\YY_{\mu}} \lesssim \sum_{i \leq 1} \left(\lVert \partial_x^i\omega\rVert_{\YY_{\mu}}+\lVert \partial_x^i\omega\rVert_{S_{\mu}}\right) \sum_{i+j=1}\lVert \partial_x^i(y\partial_{y})^j\omega\rVert_{\YY_{\mu}} \label{EQthisijhhssdinthegdsfg258} \end{align} and \begin{align} \sum_{i+j=1} \lVert \partial_x^i(y\partial_{y})^jN(s)\rVert_{\YY_{\mu}} &\lesssim \biggl( \lVert \omega\rVert_{X_{\mu}} + \sum_{1\leq i\leq 2}\bigl( \lVert \partial_{x}^{i}\omega\rVert_{\YY_{\mu}} +\lVert \partial_{x}^{i}\omega\rVert_{S_{\mu}} \bigr) \biggr) \sum_{i+j=1}\lVert \partial_{x}^i(y\partial_{y})^j\omega\rVert_{\YY_{\mu}} \notag \\& \indeq + \sum_{i \leq 1} \left(\lVert \partial_x^i\omega\rVert_{\YY_{\mu}}+\lVert \partial_x^i\omega\rVert_{S_{\mu}}\right) \sum_{i+j=2}\lVert \partial_x^i(y\partial_{y})^j\omega\rVert_{\YY_{\mu}} \, . \label{EQthisijhhssdinthegdsfg257} \end{align} \end{lemma} \colb \par \begin{proof}[Proof of Lemma~\ref{L09}] By writing the nonlinear term as in \eqref{EQthisijhhssdinthegdsfg129}, and using the definition of the $\YY_\mu$ norm, we obtain, similarly to \eqref{EQthisijhhssdinthegdsfg137}, \begin{align*} \lVert N(s)\rVert_{\YY_{\mu}} &\les \lVert \partial_{x}\omega\rVert_{\YY_{\mu}} \sum_{\xi}\sup_{y\in\Omega_{\mu}}e^{\ee(1+\mu-y)_+|\xi|}|(u_1)_\xi| + \lVert y\partial_{y}\omega\rVert_{\YY_{\mu}} \sum_{\xi}\sup_{y\in\Omega_{\mu}}e^{\ee(1+\mu-y)_+|\xi|}\biggl|\left(\frac{u_2}{y}\right)_\xi\biggr| \, . \end{align*} Using the bounds in Lemma~\ref{L14} with $i+j = 0$, we arrive at \eqref{EQthisijhhssdinthegdsfg258}. \par For $i+j=1$, by the definition of $\YY_{\mu}$ norm and Young's inequality, we have as in \eqref{EQthisijhhssdinthegdsfg130} \begin{align} \lVert \partial_{x}^i(y\partial_{y})^j N(s)\rVert_{\YY_{\mu}} & \les \lVert \partial_{x}\omega\rVert_{\YY_{\mu}} \sum_{\xi}\sup_{y\in\Omega_{\mu}}e^{\ee(1+\mu-y)_+|\xi|}|(\partial_{x}^i(y\partial_{y})^ju_1)_\xi| \notag \\&\indeq +\lVert y\partial_{y}\omega\rVert_{\YY_{\mu}} \sum_{\xi}\sup_{y\in\Omega_{\mu}}e^{\ee(1+\mu-y)_+|\xi|}\left|\left((y\partial_{y})^j \left(\frac{\partial_{x}^i u_2}{y}\right)\right)_\xi\right| \notag \\&\indeq +\lVert \partial_{x}^{i+1} (y\partial_y)^j \omega\rVert_{\YY_{\mu}} \sum_{\xi}\sup_{y\in\Omega_{\mu}}e^{\ee(1+\mu-y)_+|\xi|}|(u_1)_\xi| \notag \\&\indeq +\lVert \partial_{x}^i(y\partial_{y})^{j+1} \omega\rVert_{\YY_{\mu}} \sum_{\xi}\sup_{y\in\Omega_{\mu}}e^{\ee(1+\mu-y)_+|\xi|}\left|\left(\frac{u_2}{y}\right)_\xi\right| \, . \llabel{sldkfjg;ls dslksfgj s;lkdfgj ;sldkfjg ;sldkfgj s;ldfkgj s;ldfkgj gfisasdoifaghskcx,.bvnliahglidhgsd gs sdfgh sdfg sldfg sldfgkj slfgj sl;dfgkj sl;dfgjk sldgfkj jiurwe alskjfa;sd fasdfEQthisijhhssdinthegdsfg260} \end{align} The proof of \eqref{EQthisijhhssdinthegdsfg257} is then concluded by an application of Lemma~\ref{L14}. \end{proof} \par To conclude this section we consider the Sobolev norm estimates for the nonlinear term. \par \cole \begin{lemma} \label{L04} Let $\mu \in (0, \mu_0 - \gamma s)$ be arbitrary. We have \begin{align} \lVert N(s)\rVert_{S_\mu} \lesssim \left( \norm{\omega}_{\YY_\mu} + \norm{\omega}_{S_\mu} \right) \sum_{i+j=1} \lVert \partial_x^i\partial_{y}^j\omega\rVert_{S_\mu} \label{EQthisijhhssdinthegdsfg164} \end{align} and \begin{align} \sum_{i+j=1} \lVert \partial_x^i\partial_{y}^jN (s)\rVert_{S_\mu} & \lesssim \sum_{i+j\le1} \left( \lVert \partial_x^i \partial_y^j \omega\rVert_{\YY_{\mu}} + \lVert \partial_x^i \partial_y^j \omega\rVert_{S_{\mu}} \right) \sum_{i+j\le1} \lVert\partial_x^i \partial_y^j \omega\rVert_{S_{\mu}} \notag \\&\indeq + \left( \lVert \omega\rVert_{\YY_{\mu}} + \lVert \omega\rVert_{S_{\mu}} \right) \sum_{i+j=2} \lVert\partial_x^i \partial_y^j \omega\rVert_{S_{\mu}} \, . \label{EQthisijhhssdinthegdsfg155} \end{align} \end{lemma} \colb \par \begin{proof}[Proof of Lemma~\ref{L04}] In order to prove \eqref{EQthisijhhssdinthegdsfg164} we write \begin{align} y (u\cdot \nabla \omega) = u_1 y \partial_x \omega + u_2 y \partial_y \omega \notag \end{align} and thus from H\"older's inequality in $y$ and Young's inequality in $\xi$ we deduce \begin{align} \sum_\xi \left(\lVert u_{1,\xi}\rVert_{L^\infty( y\ge1+\mu )} + \lVert u_{2,\xi}\rVert_{L^\infty( y\ge1+\mu )} \right) \les \sum_\xi \int_0^\infty |\omega_\xi (z)| \, dz \les \norm{\omega}_{\YY_\mu} + \norm{\omega}_{S_\mu} \, . \label{EQthisijhhssdinthegdsfg2048} \end{align} \par For \eqref{EQthisijhhssdinthegdsfg155}, when $i+j=1$, by the Leibniz rule we have \begin{align} y \partial_x^i \partial_y^j (u \cdot \nabla \omega) &= \partial_x^i \partial_y^j u_1 y \partial_x \omega + u_1 y \partial_x^{i+1} \partial_y^j \omega + \partial_x^i \partial_y^j u_2 y \partial_y \omega + u_2 y \partial_x^i \partial_y^{j+1} \omega \notag \end{align} and therefore from H\"older's inequality in $y$ and Young's inequality in $\xi$ we deduce \begin{align} \norm{\partial_x^i \partial_y^j N(s)}_{S_\mu} &\les \norm{\partial_x \omega}_{S_\mu} \sum_\xi \norm{|\xi|^i \partial_y^j u_{1,\xi}}_{L^\infty(y\geq 1+\mu)} + \norm{\partial_y \omega}_{S_\mu} \sum_\xi \norm{|\xi|^i \partial_y^j u_{2,\xi}}_{L^\infty(y\geq 1+\mu)} \notag\\ &\indeq + \norm{\partial_x^{i+1} \partial_y^j \omega}_{S_\mu} \sum_\xi \norm{ u_{1,\xi}}_{L^\infty(y\geq 1+\mu)} + \norm{\partial_x^i \partial_y^{j+1} \omega}_{S_\mu} \sum_\xi \norm{u_{2,\xi}}_{L^\infty(y\geq 1+\mu)} \, . \notag \end{align} For the last two terms in the above inequality we appeal to \eqref{EQthisijhhssdinthegdsfg2048}. For the first two terms, when $(i,j)=(1,0)$ the $L^\infty$ bound on the velocity field is again given by \eqref{EQthisijhhssdinthegdsfg2048} with an additional derivative in $x$, i.e., \begin{align} \sum_\xi \lVert (\partial_x u)_{1,\xi}\rVert_{L^\infty( y\ge1+\mu )} + \lVert (\partial_x u)_{2,\xi}\rVert_{L^\infty( y\ge1+\mu )} \les \norm{\partial_x \omega}_{\YY_\mu} + \norm{\partial_x \omega}_{S_\mu} \, . \label{EQthisijhhssdinthegdsfg2048b} \end{align} On the other hand, for $(i,j) = (0,1)$, we use incompressibility and the definition of $\omega$ to write \begin{align} \partial_y u_1 = - \omega + \partial_x u_2 \qquad \mbox{and} \qquad \partial_y u_2 = -\partial_x u_1 \, . \label{EQthisijhhssdinthegdsfg:Kobe} \end{align} For the $L^\infty$ bound on $\partial_x u$ we again appeal to \eqref{EQthisijhhssdinthegdsfg2048b} whereas for the $L^\infty$ norm of $\omega$ we use the fundamental theorem of calculus and H\"older's inequality to estimate \begin{align} \sum_\xi \lVert \omega_{\xi}\rVert_{L^\infty( y\ge1+\mu )} \les \sum_\xi \lVert y \partial_y \omega_{\xi}\rVert_{L^2( y\ge1+\mu )} = \norm{\partial_y \omega}_{S_\mu}\, . \label{EQthisijhhssdinthegdsfg:Shaq} \end{align} The bound \eqref{EQthisijhhssdinthegdsfg155} now follows by combining all the estimates. \end{proof} \par \section{The Sobolev norm estimate} \label{sec-sobolev} In this section, we provide an estimate on the Sobolev part of the norm \begin{equation} \sum_{i+j\le 3} \norm{\partial_x^i \partial_y^j \omega}_{S} = \sum_{i+j\le 3} \norm{y \partial_x^i \partial_y^j \omega}_{L^2_{x,y}(y\geq 1/2)} = \sum_{i+j\le 3}\left(\sum_\xi\lVert y\xi^i\partial_{y}^j\omega_\xi\rVert^2_{L^2(y\ge1/2)}\right)^{1/2} \, . \label{EQthisijhhssdinthegdsfg275} \end{equation} For a given norm $\norm{\cdot}$ it is convenient to introduce the notation \begin{align} \lVert D^k u\rVert = \sum_{i+j = k} \lVert \partial_x^i \partial_y^j u \rVert \, . \notag \end{align} \par We first state a lemma which estimates $u$ in terms of $\omega$. \par \cole \begin{lemma} \label{x:u} Let $t$ be such that $\gamma t \leq \mu_0/2$. Then we have \begin{equation} \sum_{0\leq k \leq 2} \lVert D^k u(t) \rVert_{L^\infty_{x,y}(y\ge1/4)} \les \sum_{i+j \leq 2} \sum_\xi \lVert |\xi|^i \partial_y^j u_\xi(t) \rVert_{L^\infty(y\ge1/4)} \lesssim \NORM{\omega}_t \label{EQthisijhhssdinthegdsfg281} \end{equation} and \begin{align} \norm{D^3 u(t)}_{L^2_{x,y}(y\geq 1/4)} \les \NORM{\omega}_t \, . \label{EQthisijhhssdinthegdsfg281b} \end{align} \end{lemma} \colb \par \begin{proof}[Proof of Lemma~\ref{x:u}] The first inequality in \eqref{EQthisijhhssdinthegdsfg281}, in which the $L^\infty$ norm in $x$ is replaced by an $\ell^1$ norm in the $\xi$ variable is merely the Hausdorff-Young inequality. It thus remains to establish the second inequality in \eqref{EQthisijhhssdinthegdsfg281}. The case $j=0$ follows by the same argument as \eqref{EQthisijhhssdinthegdsfg2048}. Indeed, we only replace the norm $L^\infty(y\geq 1+\mu)$ with the norm $L^\infty(y\geq 1/4)$, which has no bearing on the estimates, to obtain \begin{align} \sum_\xi \lVert (\partial_x^i u)_{1,\xi}\rVert_{L^\infty( y\ge1/4 )} + \lVert (\partial_x^i u)_{2,\xi}\rVert_{L^\infty( y\ge1/4 )} \les \norm{\partial_x^i \omega}_{\YY_{\mu/2}} + \norm{\partial_x^i \omega}_{S_{\mu/2}} \label{EQthisijhhssdinthegdsfg2048bb} \end{align} for any $i \leq 2$ and $\mu > 0$. In particular, we may take \begin{align} \mu = \frac{\mu_0 - \gamma t}{10} \, . \label{EQthisijhhssdinthegdsfg:special:mu} \end{align} Note that since $\gamma t \leq \mu_0/2$ we have $\mu \geq \fractext{\mu_0}{20}$. To replace the $S_{\mu/2}$ norm, which is $\ell^1$ in $\xi$, with the $S$ norm, which is $\ell^2$ in $\xi$, we pay an additional price of $1 + |\xi|$ (cf.~Lemma~\ref{L06} below). Additionally, in \eqref{EQthisijhhssdinthegdsfg2048bb} we further appeal to the analyticity recovery for the $\YY$~norm, cf.~Lemma~\ref{ana:rec:l} below, and obtain \begin{align} \sum_{i\leq 2} \sum_\xi \lVert (\partial_x^i u)_{1,\xi}\rVert_{L^\infty( y\ge1/4 )} + \lVert (\partial_x^i u)_{2,\xi}\rVert_{L^\infty( y\ge1/4 )} \les \norm{\omega}_{\YY_{\mu}} + \sum_{i\leq 3} \norm{\partial_x^i \omega}_{S} \les \NORM{\omega}_t \, . \llabel{sldkfjg;ls dslksfgj s;lkdfgj ;sldkfjg ;sldkfgj s;ldfkgj s;ldfkgj gfisasdoifaghskcx,.bvnliahglidhgsd gs sdfgh sdfg sldfg sldfgkj slfgj sl;dfgkj sl;dfgjk sldgfkj jiurwe alskjfa;sd fasdfEQthisijhhssdinthegdsfg2048cc} \end{align} \par This concludes the proof of \eqref{EQthisijhhssdinthegdsfg281} when $j=0$ and $i \leq 2$. For the case $j=1$, we use \eqref{EQthisijhhssdinthegdsfg:Kobe} to convert the $\partial_y$ derivative into a $\partial_x$ derivative, at a cost of an additional term involving $\omega$. Similarly to \eqref{EQthisijhhssdinthegdsfg:Shaq}, appealing to Lemma~\ref{L06}, using that \colW $w(y)\gtrsim 1$ for $y \in [1/4,1/2]$\colb, and with $\mu$ as in \eqref{EQthisijhhssdinthegdsfg:special:mu} we get \begin{align} \sum_\xi \lVert |\xi|^i \omega_{\xi}\rVert_{L^\infty( y\ge1/4 )} &\les \sum_\xi \lVert w(y) e^{\ee(1+\mu-y)_+ |\xi|} |\xi|^i \omega_{\xi}\rVert_{L^\infty(1/4 \leq y\le1/2 )} + \sum_\xi \lVert y \partial_y |\xi|^i \omega_{\xi}\rVert_{L^2( y\ge1/2 )} \notag\\ &\les \norm{\partial_x^i \omega}_{X_\mu} + \norm{\partial_x^i \partial_y \omega}_{S} + \norm{\partial_x^{i+1} \partial_y \omega}_{S} \les \NORM{\omega}_t \label{EQthisijhhssdinthegdsfg2048bbb} \end{align} for $i \leq 1$. The above estimate gives \eqref{EQthisijhhssdinthegdsfg281} for $j=1$. It only remains to consider the case $(i,j)=(0,2)$. For this purpose, note that \begin{align} \partial_y^2 u_1 = -\partial_y \omega - \partial_x^2 u_1 \qquad \mbox{and} \qquad \partial_y^2 u_2 = \partial_x \omega - \partial_x^2 u_2 \, , \label{EQthisijhhssdinthegdsfg:Kobe:*} \end{align} which follows from \eqref{EQthisijhhssdinthegdsfg:Kobe} and incompressibility. The terms with two $x$ derivatives were already estimated in \eqref{EQthisijhhssdinthegdsfg2048bb}, whereas $\partial_x \omega$ was already bounded in \eqref{EQthisijhhssdinthegdsfg2048bbb}. Lastly, for the term $\partial_y \omega$, we have, similarly to \eqref{EQthisijhhssdinthegdsfg2048bbb}, \begin{align} \sum_\xi \lVert \partial_y \omega_{\xi}\rVert_{L^\infty( y\ge 1/4 )} &\les \sum_\xi \lVert e^{\ee(1+\mu-y)_+ |\xi|} w(y) y \partial_y \omega_{\xi}\rVert_{L^\infty(1/4 \leq y \leq 1/2 )} + \sum_\xi \lVert y \partial_y^2 \omega_{\xi}\rVert_{L^2( y\ge1/2 )} \notag\\ &\les \norm{y \partial_y \omega}_{X_\mu} + \norm{ \partial_y^2 \omega}_{S} + \norm{\partial_x \partial_y^2 \omega}_{S} \les \NORM{\omega}_t \,, \notag \end{align} which gives \eqref{EQthisijhhssdinthegdsfg281}. \par In order to conclude the proof of the lemma, we need to establish \eqref{EQthisijhhssdinthegdsfg281b}. For this purpose, fix $(i,j)$ such that $i+j = 3$. To avoid redundancy, we only consider the cases $(i,j)=(3,0)$ and $(i,j) = (0,3)$. First, using Lemma~\ref{bio:sav} and Young's inequality, we have \begin{align} \lVert \partial_{x}^3u \rVert_{L^2(y\ge1/4)}^2 &\lesssim \sum_\xi \lVert |\xi|^3 u_\xi\rVert_{L^2(y\ge1/4)}^2 \notag \\& \lesssim \sum_\xi \left\lVert \int_{0}^\infty e^{-|y-z||\xi|} |\xi|^3|\omega_\xi|(z)\,dz\right\rVert_{L^2(y\ge1/4)}^2 \notag \\& \les \sum_\xi \lVert |\xi|^{5/2}|\omega_\xi| \rVert_{L^1(z\leq 1/2)}^2 + \sum_\xi \lVert |\xi|^2|\omega_\xi| \rVert_{L^2(z\geq 1/2)}^2 \notag \\& \les \sum_\xi \lVert e^{\ee (1+\mu-z)_+ |\xi|} |\omega_\xi| \rVert_{L^1(z\leq 1/2)}^2 |\xi|^{5/2} e^{-\frac{\ee}{2} |\xi|} + \sum_\xi \lVert z |\xi|^2|\omega_\xi| \rVert_{L^2(z\geq 1/2)}^2 \notag \\ & \lesssim \norm{\omega}_{\YY_\mu}^2 + \norm{\partial_x^2 \omega}_S^2 \les \NORM{\omega}_t^2 \, , \llabel{sldkfjg;ls dslksfgj s;lkdfgj ;sldkfjg ;sldkfgj s;ldfkgj s;ldfkgj gfisasdoifaghskcx,.bvnliahglidhgsd gs sdfgh sdfg sldfg sldfgkj slfgj sl;dfgkj sl;dfgjk sldgfkj jiurwe alskjfa;sd fasdfEQthisijhhssdinthegdsfg291} \end{align} with $\mu$ as in \eqref{EQthisijhhssdinthegdsfg:special:mu}. In the last inequality above we used $\norm{ \cdot }_{\ell^2} \leq \norm{\cdot }_{\ell^1}$. Thus, we have proven \eqref{EQthisijhhssdinthegdsfg281b} for $(i,j) = (3,0)$. For the other extremal case, we apply the $y$ derivative to \eqref{EQthisijhhssdinthegdsfg:Kobe:*} and obtain \begin{align} \partial_y^3 u_1 = \partial_y^2 \omega + \partial_{x}^2 \omega - \partial_{x}^{3} u_2 \qquad \mbox{and} \qquad \partial_y^3 u_2 = -\partial_x \partial_y \omega + \partial_x^3 u_1 \, . \llabel{sldkfjg;ls dslksfgj s;lkdfgj ;sldkfjg ;sldkfgj s;ldfkgj s;ldfkgj gfisasdoifaghskcx,.bvnliahglidhgsd gs sdfgh sdfg sldfg sldfgkj slfgj sl;dfgkj sl;dfgjk sldgfkj jiurwe alskjfa;sd fasdfEQthisijhhssdinthegdsfg:Kobe:**} \end{align} The velocity terms $\partial_{x}^{3}u_1$ and $\partial_{x}^{3}u_2$ were already bounded above. Clearly, we have $\norm{D^2 \omega}_{L^2(y\geq 1/2)} \les \norm{D^2 \omega}_S \les \NORM{\omega}_t$. On the other hand, similarly to \eqref{EQthisijhhssdinthegdsfg2048bbb}, we have \[ \norm{D^2 \omega}_{L^2(1/4 \leq y\leq 1/2)} \les \norm{D^2 \omega}_{L^\infty(1/4 \leq y\leq 1/2)} \les \norm{D^2 \omega}_{X_\mu} \les \NORM{\omega}_t . \] In the last inequality we used that $\mu$ in \eqref{EQthisijhhssdinthegdsfg:special:mu} is bounded from below by $\mu_0/20$. This concludes the proof of the lemma. \end{proof} \par The remainder of this section is devoted to an a priori estimate for the norm $\sum_{i+j \leq 3} \lVert{\partial_x^i \partial_y^j \omega \rVert}_S$. For this purpose, denote \begin{align} \phi(y)=y\bar \psi(y) \, , \label{EQthisijhhssdinthegdsfg:phi:def} \end{align} where $\bar\psi\in C^\infty$ is a non-decreasing function such that $\bar\psi=0$ for $0\le y\le 1/4$ and $\bar\psi=1$ for $y\ge 1/2$. In order to estimate the norm in \eqref{EQthisijhhssdinthegdsfg275}, note that \begin{align} \norm{y f}_{L^2_{x,y}(y\geq 1/2)} \leq \norm{\phi f}_{L^2(\HH)} \,, \notag \end{align} so that it suffices to estimate this larger norm. \par \cole \begin{lemma} \label{L12a} For any $0 < t < \frac{\mu_0}{2\gamma}$ the estimate \begin{align} &\sum_{i+j\le 3} \lVert \phi \partial_x^i\partial_{y}^j\omega(t)\rVert^2_{L^2(\HH)} \notag\\ &\qquad \lesssim \Bigl(1+ t \sup_{s\in[0,t]} \NORM{\omega(s)}_s^3 \Bigr) e^{C t (1+ \sup_{s\in[0,t]} \NORM{\omega(s)}_s)} \sum_{i+j\le 3} \lVert \phi \partial_x^i\partial_{y}^j\omega_0\rVert^2_{L^2(\HH)} \llabel{sldkfjg;ls dslksfgj s;lkdfgj ;sldkfjg ;sldkfgj s;ldfkgj s;ldfkgj gfisasdoifaghskcx,.bvnliahglidhgsd gs sdfgh sdfg sldfg sldfgkj slfgj sl;dfgkj sl;dfgjk sldgfkj jiurwe alskjfa;sd fasdfEQthisijhhssdinthegdsfg286a} \end{align} holds, where $C>0$ is a constant independent of $\gamma$. \end{lemma} \colb \par \begin{proof}[Proof of Lemma~\ref{L12a}] Let $\alpha \in {\mathbb N}_0^2$ be a multi-index with $|\alpha|\leq 3$. We apply $\partial^\alpha$ to the vorticity form of the Navier-Stokes equation and test this equation with $\phi^2 \partial^\alpha \omega$ to obtain the energy estimate \begin{align} &\frac 12 \frac{d}{dt} \norm{\phi \partial^\alpha \omega}_{L^2(\HH)}^2 + \nu \norm{\phi \nabla \partial^\alpha \omega}_{L^2(\HH)}^2 \notag\\ &\indeq = 2 \int_{\HH} u_2 \phi' \phi |\partial^\alpha \omega|^2 - \sum_{0 < \beta \leq \alpha} {\alpha \choose \beta} \int_{\HH} \partial^\beta u \cdot \nabla \partial^{\alpha-\beta} \omega \partial^\alpha \omega \phi^2 - 2 \nu \int_{\HH}\phi' \partial^\alpha \omega \partial_y \partial^\alpha \omega \phi \, . \label{EQthisijhhssdinthegdsfg1001} \end{align} Using the pointwise estimate \begin{align} |\phi'(y)| \les \phi(y) + \chi_{ \{1/4 \leq y \leq 1/2\}} \notag \end{align} on the first and the third term, summing over $|\alpha|\leq 3$, and absorbing a part of the third term in \eqref{EQthisijhhssdinthegdsfg1001} we obtain \begin{align} \frac 12 \frac{d}{dt} \sum_{i+j\le 3} \lVert \phi \partial_x^i\partial_{y}^j\omega\rVert^2_{L^2(\HH)} &\les \biggl( \nu + \norm{u_2}_{L^\infty(y\geq 1/4)} + \sum_{1\leq k \leq 2} \lVert D^k u\rVert_{L^\infty_{x,y}(y\ge1/4)} \biggr) \sum_{i+j\le 3} \lVert \phi \partial_x^i\partial_{y}^j\omega\rVert^2_{L^2(\HH)} \notag\\ &\quad + \norm{D^3 u}_{L^2_{x,y}(y\geq 1/4)} \norm{\phi \nabla \omega}_{L^\infty(\HH)} \sum_{i+j\le 3} \lVert \phi \partial_x^i\partial_{y}^j\omega\rVert_{L^2(\HH)} \notag\\ &\quad + \left(\nu + \norm{u_2}_{L^\infty(\HH)} \right) \sum_{i+j\le 3} \lVert \partial_x^i\partial_{y}^j\omega\rVert^2_{L^2_{x,y}(1/4\leq y \leq 1/2)} \, . \label{EQthisijhhssdinthegdsfg1002} \end{align} Next, note that we have the analytic estimate \begin{align} \sum_{i+j\leq3} \lVert \partial_x^i\partial_{y}^j\omega\rVert^2_{L^2_{x,y}(1/4\leq y \leq 1/2)} &\les \sum_{i+j\leq3}\sum_\xi \norm{\xi^i \partial_y^j \omega_\xi}_{L^2(1/4\leq y \leq 1/2)}^2 \notag\\ &\les \sum_{i+j\leq3}\sum_\xi \norm{e^{\epsilon_0 (1+\mu/2-y)_+ |\xi|} \xi^i (y \partial_y)^j \omega_\xi}_{\ZL_{\mu/2,\nu}}^2 \notag\\ &\les \sum_{i+j\leq3}\norm{\partial_x^i (y \partial_y)^j \omega}_{X_{\mu/2}}^2 \llabel{sldkfjg;ls dslksfgj s;lkdfgj ;sldkfjg ;sldkfgj s;ldfkgj s;ldfkgj gfisasdoifaghskcx,.bvnliahglidhgsd gs sdfgh sdfg sldfg sldfgkj slfgj sl;dfgkj sl;dfgjk sldgfkj jiurwe alskjfa;sd fasdfEQthisijhhssdinthegdsfg1003} \end{align} uniformly in $\mu >0$. In particular, we may take $\mu = \frac{\mu_0-\gamma t}{10}$. Here we used essentially that the weight $w(y)$ is comparable to $1$ (independently of $\nu$) in the region $\{1/4 \leq y \leq 1/2\}$. Since $\mu > 0$, we may further use the analyticity recovery Lemma~\ref{ana:rec:x}, and estimate \begin{align} \sum_{i+j\leq 3} \lVert \partial_x^i\partial_{y}^j\omega\rVert^2_{L^2_{x,y}(1/4\leq y \leq 1/2)} &\les \norm{\omega}_{X_{\mu}}^2 \, . \label{EQthisijhhssdinthegdsfg1004} \end{align} Note that we used that $\frac{1}{\mu} = \frac{10}{\mu_0-\gamma t} \leq \frac{20}{\mu_0} \les 1$. \par For the second term on the right side of \eqref{EQthisijhhssdinthegdsfg1002}, we appeal to \eqref{EQthisijhhssdinthegdsfg281b} and to the estimate \begin{align} \norm{\phi \nabla \omega}_{L^\infty(\HH)} &\les \norm{\nabla (\phi \omega)}_{L^\infty(\HH)} + \norm{\phi \omega}_{L^\infty(\HH)} + \norm{\omega}_{L^\infty_{x,y}(1/4 \leq y\leq 1/2)} \notag\\ &\les \sum_{i+j\leq 3} \norm{\partial_x^i \partial_y^j (\phi \omega)}_{L^2(\HH)} + \norm{\omega}_{X_\mu} \, . \notag \end{align} Here we have used the Sobolev embedding $H^2(\HH) \subset L^\infty(\HH)$, the previously established bound \eqref{EQthisijhhssdinthegdsfg1004}, the Leibniz rule, and the definition of $\phi$ in \eqref{EQthisijhhssdinthegdsfg:phi:def}. The resulting inequality is \begin{align} \norm{D^3 u}_{L^2_{x,y}(y\geq 1/4)} \norm{\phi \nabla \omega}_{L^\infty(\HH)} \sum_{i+j\le 3} \lVert \phi \partial_x^i\partial_{y}^j\omega\rVert_{L^2(\HH)} \les \NORM{\omega}_t \sum_{i+j\le 3} \lVert \phi \partial_x^i\partial_{y}^j\omega\rVert_{L^2(\HH)}^2 + \NORM{\omega}_t \norm{\omega}_{X_\mu}^2 \, . \notag \end{align} \par Combining \eqref{EQthisijhhssdinthegdsfg1002}--\eqref{EQthisijhhssdinthegdsfg1004}, and Lemma~\ref{x:u} we deduce \begin{align} \frac 12 \frac{d}{dt} \sum_{i+j\le 3} \lVert \phi \partial_x^i\partial_{y}^j\omega\rVert^2_{L^2(\HH)} &\les (1 + \NORM{\omega(t)}_t) \sum_{i+j\le 3} \lVert \phi \partial_x^i\partial_{y}^j\omega(t)\rVert^2_{L^2(\HH)} + (1+ \NORM{\omega(t)}_t) \norm{\omega(t)}_{X_{\mu}}^2 \, . \notag \end{align} Upon applying the Gr\"onwall inequality, the proof of the lemma is concluded. \end{proof} \par \appendix \par \startnewsection{Proofs of some technical lemmas}{sec06} Here we list some technical lemmas. The next lemma converts an $\ell^1$ norm in $\xi$ to an $\ell^2$ norm, which is necessary when converting $S_\mu$ norms to an $S$ and hence a $Z$ norm. \par \cole \begin{lemma} \label{L06} Let $\mu\in(0,1)$. We have \begin{align} \sum_{i+j\le2}&\lVert \partial_x^i (y\partial_{y})^j\omega\rVert_{S_\mu} \lesssim \sum_{i+j\le2} \norm{\partial_x^i(y\partial_{y})^j\omega}_S + \norm{\partial_x^{i+1}(y\partial_{y})^j\omega}_S \, . \llabel{sldkfjg;ls dslksfgj s;lkdfgj ;sldkfjg ;sldkfgj s;ldfkgj s;ldfkgj gfisasdoifaghskcx,.bvnliahglidhgsd gs sdfgh sdfg sldfg sldfgkj slfgj sl;dfgkj sl;dfgjk sldgfkj jiurwe alskjfa;sd fasdfEQthisijhhssdinthegdsfg165} \end{align} \end{lemma} \colb \par {\begin{proof}[Proof of Lemma~\ref{L06}] We have \begin{align} &\sum_{\xi} |v_{\xi}| \les \biggl( \sum_{\xi} (1+|\xi|^2)|v_{\xi}|^2 \biggr)^{1/2} \label{EQthisijhhssdinthegdsfg07} \end{align} for every $v$ for which the right side is finite. The inequality \eqref{EQthisijhhssdinthegdsfg07} holds since $\sum_{\xi}(1+|\xi|^2)^{-1}<\infty$. \end{proof} \par \cole \begin{lemma} \label{int:t} Assume that the parameters $\mu, \mu_0,\gamma, t > 0$ obey $\mu < \mu_0 - \gamma t$. Then, for $\alpha \in (0,\frac 12)$ we have \begin{equation} \int_0^t \frac{1}{\sqrt{t-s}}\frac{1}{(\mu_0-\mu-\gamma s)^{1+\alpha}}\,ds \leq \frac{C}{\sqrt{\gamma}(\mu_0-\mu-\gamma t)^{1/2+\alpha}} \label{EQthisijhhssdinthegdsfg310} \end{equation} and \begin{equation} \int_0^t \frac{1}{\sqrt{t-s}}\frac{1}{(\mu_0-\mu-\gamma s)^{\alpha}}\,ds \leq \frac{C}{\sqrt{\gamma}} \, , \label{EQthisijhhssdinthegdsfg310a} \end{equation} where $C>0$ is a constant depending on $\mu_0$ and $1/2-\alpha$. \end{lemma} \colb \par \begin{proof}[Proof of Lemma~\ref{int:t}] Changing variables $s'=\gamma s$, $t'=\gamma t $, and letting $\mu_0 - \mu = \mu' > 0$, the left side of \eqref{EQthisijhhssdinthegdsfg310} is rewritten and bounded as \begin{align} & \int_0^{t'} \frac{\sqrt{\gamma}}{\sqrt{t'- s'}}\frac{1}{(\mu'- s')^{1+\alpha}}\, \frac{d s'}{\gamma} \leq \frac{1}{\sqrt{\gamma}(\mu'- t')^{\alpha} } \int_0^{t'} \frac{d s'}{\sqrt{t'- s'}(\mu'- s')} \notag \\&\indeq = \frac{2 \arctan\left(\sqrt{\frac{t'}{\mu'-t'}}\right)}{\sqrt{\gamma}(\mu'- t')^{1/2+ \alpha} } \les \frac{1}{\sqrt{\gamma}(\mu'- t')^{1/2+ \alpha} } = \frac{1}{\sqrt{\gamma}(\mu_0-\mu-\gamma t)^{1/2+ \alpha} } \, . \notag \end{align} \par In order to prove \eqref{EQthisijhhssdinthegdsfg310a}, we proceed similarly and use $\mu' > t'$ to deduce \begin{equation} \int_0^{t'} \frac{\sqrt{\gamma}}{\sqrt{t'- s'}}\frac{1}{(\mu'- s')^{\alpha}}\, \frac{d s'}{\gamma} \leq \frac{1}{\sqrt{\gamma}} \int_0^{t'} \frac{d s'}{(t'- s')^{1/2+\alpha}} \les \frac{1}{\sqrt{\gamma}} \, , \notag \end{equation} where the implicit constant may depend on $\mu_0$ and $1/2-\alpha$. \end{proof} \par \cole \begin{lemma}[Analyticity recovery for the $X$~norm] For $\tilde \mu>\mu\ge0$, we have \label{ana:rec:x} \begin{equation} \sum_{i+j=1}\lVert \partial_{x}^i(y\partial_{y})^jf\rVert_{X_\mu} \lesssim \frac{1}{\tilde \mu-\mu}\lVert f\rVert_{X_{\tilde \mu}} \, . \llabel{sldkfjg;ls dslksfgj s;lkdfgj ;sldkfjg ;sldkfgj s;ldfkgj s;ldfkgj gfisasdoifaghskcx,.bvnliahglidhgsd gs sdfgh sdfg sldfg sldfgkj slfgj sl;dfgkj sl;dfgjk sldgfkj jiurwe alskjfa;sd fasdfEQthisijhhssdinthegdsfg311} \end{equation} \end{lemma} \colb \par \begin{proof}[Proof of Lemma~\ref{ana:rec:x}] First, let $(i,j)=(1,0)$. According to the definition of the $X_\mu$~norm, and using that the bound $(\tilde \mu-\mu) |\xi| e^{\epsilon_0 |\xi| ( (1+\mu-y)_+ - (1+\tilde \mu-y)_+)} \lesssim 1$ holds on $\Omega_\mu$, we have \begin{align} \lVert \partial_{x}f \rVert_{X_\mu} &= \sum_\xi \lVert \xi e^{\ee(1+\mu-y)_+|\xi|} f_\xi\rVert_{\ZL_{\mu,\nu}} \lesssim \frac{1}{\tilde \mu-\mu}\sum_\xi \lVert e^{\ee(1+\tilde \mu-y)_+|\xi|} f_\xi\rVert_{\ZL_{\mu,\nu}} \notag \\& \lesssim \frac{1}{\tilde \mu-\mu}\sum_\xi \lVert e^{\ee(1+\tilde \mu-y)_+|\xi|} f_\xi\rVert_{\ZL_{\tilde \mu, \nu}} = \frac{1}{\tilde \mu-\mu}\lVert f \rVert_{X_{\tilde \mu}} \, . \llabel{sldkfjg;ls dslksfgj s;lkdfgj ;sldkfjg ;sldkfgj s;ldfkgj s;ldfkgj gfisasdoifaghskcx,.bvnliahglidhgsd gs sdfgh sdfg sldfg sldfgkj slfgj sl;dfgkj sl;dfgjk sldgfkj jiurwe alskjfa;sd fasdfEQthisijhhssdinthegdsfg312} \end{align} \par Next, consider $(i,j)=(0,1)$. By the Cauchy integral theorem, we have \begin{align} \partial_{y}f_\xi(y) = \int_{C(y, R_y)} \frac{f_\xi(z)}{(y-z)^2}\,dz \, , \label{EQthisijhhssdinthegdsfg:Cau:int} \end{align} where $C(y, R_y)$ is the circle centered at $y$ with radius $R_y$. Hence, we have \begin{equation} |\partial_{y}f_\xi(y)| \lesssim \frac{1}{R_y}\sup_{z\in C(y, R_y)} |f_\xi(z)| \, . \llabel{sldkfjg;ls dslksfgj s;lkdfgj ;sldkfjg ;sldkfgj s;ldfkgj s;ldfkgj gfisasdoifaghskcx,.bvnliahglidhgsd gs sdfgh sdfg sldfg sldfgkj slfgj sl;dfgkj sl;dfgjk sldgfkj jiurwe alskjfa;sd fasdfEQthisijhhssdinthegdsfg313} \end{equation} By taking $R_y=C^{-1}(\tilde \mu-\mu)\Re y$, for a sufficiently large universal constant $C>0$, we obtain \begin{align} \lVert y\partial_{y}f \rVert_{X_\mu} &= \sum_\xi \lVert e^{\ee(1+\mu-y)_+|\xi|}y\partial_{y}f_\xi\rVert_{\ZL_{\mu,\nu}} \lesssim \frac{1}{\tilde \mu-\mu}\sum_\xi \lVert e^{\ee(1+\mu-y)_+|\xi|} f_\xi\rVert_{\ZL_{\tilde \mu, \nu}} \notag \\& \lesssim \frac{1}{\tilde \mu-\mu}\sum_\xi \lVert e^{\ee(1+\tilde \mu-y)_+|\xi|} f_\xi\rVert_{\ZL_{\tilde \mu, \nu}} = \frac{1}{\tilde \mu-\mu}\lVert f \rVert_{X_{\tilde \mu}} \, , \llabel{sldkfjg;ls dslksfgj s;lkdfgj ;sldkfjg ;sldkfgj s;ldfkgj s;ldfkgj gfisasdoifaghskcx,.bvnliahglidhgsd gs sdfgh sdfg sldfg sldfgkj slfgj sl;dfgkj sl;dfgjk sldgfkj jiurwe alskjfa;sd fasdfEQthisijhhssdinthegdsfg314} \end{align} concluding the proof. \end{proof} \par \cole \begin{lemma}[Analyticity recovery for the $\YY$~norm] Let $\mu_0\geq \tilde \mu>\mu\ge0$. Then we have \label{ana:rec:l} \begin{equation} \sum_{i+j=1}\lVert \partial_{x}^i(y\partial_{y})^jf\rVert_{\YY_\mu} \lesssim \frac{1}{\tilde \mu-\mu}\lVert f\rVert_{\YY_{\tilde \mu}} \, . \label{EQthisijhhssdinthegdsfg315} \end{equation} \end{lemma} \colb \par \begin{proof}[Proof of Lemma~\ref{ana:rec:l}] By the same argument which yielded the $X$ norm estimate in Lemma~\ref{ana:rec:x}, we obtain \begin{align} \lVert \partial_{x}f \rVert_{\YY_\mu} &= \sum_\xi \lVert \xi e^{\ee(1+\mu-y)_+|\xi|} f_\xi\rVert_{\SL_\mu} \lesssim \frac{1}{\tilde \mu-\mu}\sum_\xi \lVert e^{\ee(1+\tilde \mu-y)_+|\xi|} f_\xi\rVert_{\SL_\mu} \notag \\& \lesssim \frac{1}{\tilde \mu-\mu}\sum_\xi \lVert e^{\ee(1+\tilde \mu-y)_+|\xi|} f_\xi\rVert_{\SL_{\tilde \mu}} = \frac{1}{\tilde \mu-\mu}\lVert f \rVert_{\YY_{\tilde \mu}} \, . \llabel{Molče v to prošnjo Črtomír dovoli,
z duhovnim bliža slapu se Savice,
molitve svete mašnik, on z njim moli,
v imeni kŕsti ga svete Trojice.
So na kolenah, kar jih je okoli,
se od veselja svet' obraz device,
ki je bila podpora vere krive,
je opravljála službo bógnje Žive.sldkfjg;ls dslksfgj s;lkdfgj ;sldkfjg ;sldkfgj s;ldfkgj s;ldfkgj gfisasdoifaghskcx,.bvnliahglidhgsd gs sdfgh sdfg sldfg sldfgkj slfgj sl;dfgkj sl;dfgjk sldgfkj jiurwe alskjfa;sd fasdfEQthisijhhssdinthegdsfg316} \end{align} In order to prove the estimate \eqref{EQthisijhhssdinthegdsfg315} for $(i,j)=(0,1)$, we use \eqref{EQthisijhhssdinthegdsfg:Cau:int} to bound \begin{align} \lVert y\partial_{y}f_\xi \rVert_{L^1(\partial\Omega_{\theta})} = \int_{\partial\Omega_{\theta}} \left|\int_{C(y, R_y)} \frac{yf_\xi(z)}{(y-z)^2}\,dz\right|dy \lesssim \int_{\partial\Omega_{\theta}} \int_{C(y, R_y)} \frac{|yf_\xi(z)|}{R_y^2}\,dzdy \llabel{sldkfjg;ls dslksfgj s;lkdfgj ;sldkfjg ;sldkfgj s;ldfkgj s;ldfkgj gfisasdoifaghskcx,.bvnliahglidhgsd gs sdfgh sdfg sldfg sldfgkj slfgj sl;dfgkj sl;dfgjk sldgfkj jiurwe alskjfa;sd fasdfEQthisijhhssdinthegdsfg318} \end{align} for any $0\le \theta<\mu$. By taking $R_y=C^{-1} (\tilde \mu-\mu)\Re y$ for a sufficiently large universal constant $C>0$, using that $|y|$ is comparable to $\Re y$ in this region, and applying Fubini's theorem, we obtain \begin{align} \lVert y\partial_{y}f_\xi \rVert_{L^1(\partial\Omega_{\theta})} &\lesssim \frac{1}{\tilde \mu-\mu}\int_{\partial\Omega_{\theta}} \int_{C(y, R_y)} \frac{|f_\xi(z)|}{R_y}\,dzdy \notag \\& \les \frac{1}{\tilde \mu-\mu}\int_{\partial\Omega_{\theta}} \int_{0}^{2\pi} |f_\xi(y + R_y e^{i\phi} )| \,d\phi dy \notag \\& \les \frac{1}{\tilde \mu-\mu} \sup_{\bar \theta \in (\theta - \frac{2(\tilde \mu-\mu)}{C}, \theta + \frac{2(\tilde \mu-\mu)}{C})} \norm{f_\xi}_{L^1(\partial \Omega_{\bar \theta})} \notag \\ & \lesssim \frac{1}{\tilde \mu-\mu} \lVert f_\xi\rVert_{\SL_{\tilde \mu}} \, , \llabel{sldkfjg;ls dslksfgj s;lkdfgj ;sldkfjg ;sldkfgj s;ldfkgj s;ldfkgj gfisasdoifaghskcx,.bvnliahglidhgsd gs sdfgh sdfg sldfg sldfgkj slfgj sl;dfgkj sl;dfgjk sldgfkj jiurwe alskjfa;sd fasdfEQthisijhhssdinthegdsfg319} \end{align} which proves the claim. {Since $e^{\ee(1+\mu-y)_{+}|\xi|} \leq e^{\ee(1+\tilde \mu-y)_{+}|\xi|}$, for every $y \in \Omega_{\mu}$}, the lemma follows. \end{proof} \par \section*{Acknowledgments} IK was supported in part by the NSF grant DMS-1615239, while VV was supported in part by the NSF CAREER Grant DMS-1652134. \par  \par \end{document}